\documentclass[11pt, twoside, leqno]{article}

\usepackage{amsmath}
\usepackage{amssymb}
\usepackage{titletoc}
\usepackage{mathrsfs}
\usepackage{amsthm}
\usepackage{color}
\usepackage{latexsym,amssymb}

\usepackage{indentfirst}
\usepackage{color}
\usepackage{txfonts}

\allowdisplaybreaks

\def\bint{{\ifinner\rlap{\bf\kern.30em--}
\int\else\rlap{\bf\kern.35em--}\int\fi}\ignorespaces}

\def\sbint{{\ifinner\rlap{\bf\kern.32em--}
\hspace{0.078cm}\int\else\rlap{\bf\kern.45em--}\int\fi}\ignorespaces}

\def\red{\color{red}}

\def\rr{{\mathbb R}}
\def\rn{{\mathbb{R}^n}}

\DeclareMathOperator{\esssup}{ess\,sup}
\DeclareMathOperator{\essinf}{ess\,inf}

\def\nn{{\mathbb N}}
\def\zz{{\mathbb Z}}

\def\fz{\infty }

\def\dz{\delta}

\def\lf{\left}
\def\r{\right}
\def\ls{\lesssim}
\def\noz{\nonumber}
\def\wz{\widetilde}

\def\loc{{\mathrm{loc}}}

\def\XXint#1#2#3{{\setbox0=\hbox{$#1{#2#3}{\int}$ }
\vcenter{\hbox{$#2#3$ }}\kern-.6\wd0}}

\newtheorem{theorem}{Theorem}[section]
\newtheorem{lemma}[theorem]{Lemma}

\newtheorem{proposition}[theorem]{Proposition}

\newtheorem{assumption}[theorem]{Assumption}
\theoremstyle{definition}
\newtheorem{remark}[theorem]{Remark}
\newtheorem{definition}[theorem]{Definition}
\renewcommand{\appendix}{\par
\setcounter{section}{0}%
\setcounter{subsection}{0}%
\setcounter{subsubsection}{0}%
\gdef\thesection{\@Alph\c@section}%
\gdef\thesubsection{\@Alph\c@section.\@arabic\c@subsection}%
\gdef\theHsection{\@Alph\c@section.}%
\gdef\theHsubsection{\@Alph\c@section.\@arabic\c@subsection}%
\csname appendixmore\endcsname
}

\numberwithin{equation}{section}

\begin{document}

\arraycolsep=1pt

\title{\bf\Large
Estimates for Littlewood--Paley Operators on
Ball Campanato-Type Function Spaces
\footnotetext{\hspace{-0.35cm} 2020 {\it
Mathematics Subject Classification}. Primary 42B25;
Secondary 42B20, 42B30, 42B35, 46E35, 47A30.
\endgraf {\it Key words and phrases.}
Littlewood--Paley operator, Campanato space, ball quasi-Banach function space, John--Nirenberg space.
\endgraf This project is supported by
the National Natural Science Foundation of China (Grant Nos.\
11971058, 12071197 and 12122102) and the National
Key Research and Development Program of China
(Grant No.\ 2020YFA0712900).}}
\date{}
\author{Hongchao Jia, Dachun Yang\footnote{Corresponding author,
E-mail: \texttt{dcyang@bnu.edu.cn}/{\red June 19, 2022}/Final version.},\ \
Wen Yuan and Yangyang Zhang}
\maketitle

\vspace{-0.8cm}

\begin{center}
\begin{minipage}{13cm}
{\small {\bf Abstract}\quad
Let $X$ be a ball quasi-Banach function space on ${\mathbb R}^n$
and assume that the Hardy--Littlewood maximal operator satisfies the
Fefferman--Stein vector-valued maximal inequality on $X$, and
let $q\in[1,\infty)$ and $d\in(0,\infty)$.
In this article, the authors prove that, for any $f\in \mathcal{L}_{X,q,0,d}(\mathbb{R}^n)$
(the ball Campanato-type function space associated with $X$),
the Littlewood--Paley $g$-function $g(f)$ is either infinite everywhere or finite
almost everywhere and, in the latter case,
$g(f)$ is bounded on $\mathcal{L}_{X,q,0,d}(\mathbb{R}^n)$.
Similar results for both the Lusin-area function and
the Littlewood--Paley $g_\lambda^*$-function
are also obtained. All these results have a
wide range of applications. Particularly, even when
$X$ is the weighted Lebesgue space,
or the mixed-norm Lebesgue space, or
the variable Lebesgue space, or the Orlicz space,
or the Orlicz-slice space, all
these results are new. The proofs of all these results strongly depend on
several delicate estimates of Littlewood--Paley operators
on the mean oscillation of the locally integrable function $f$ on $\mathbb{R}^n$.
Moreover, the same ideas are also used to obtain the corresponding results
for the special John--Nirenberg--Campanato space via congruent cubes.
}
\end{minipage}
\end{center}

\vspace{0.2cm}



\section{Introduction\label{Introduction}}

It is well known that both Campanato spaces and Hardy spaces play vital roles in
harmonic analysis and partial differential equations
(see, for instance, \cite{CW,EMS1970}).
Recall that John and Nirenberg \cite{JN}
introduced the space $\mathop{\mathrm{BMO}\,}(\rn)$,
namely, the space of functions with bounded mean oscillation,
which proves the dual space
of the Hardy space $H^1(\rn)$ by Fefferman and Stein \cite{FS72}.
Later, as a generalization of the space $\mathop{\mathrm{BMO}\,}(\rn)$,
Campanato \cite{C} introduced
the Campanato space $\mathcal{C}_{\alpha,q,s}(\rn)$
which coincides with
$\mathop{\mathrm{BMO}\,}(\rn)$ when $\alpha=0$.
Moreover, for any given $p\in(0,1)$, any given $q\in[1,\infty]$, and any given integer
$s\in[0,\infty)\cap[n(\frac{1}{p}-1),\infty)$, $\mathcal{C}_{\frac{1}{p}-1,q,s}(\rn)$
was proved to be the dual space of the Hardy space $H^p(\rn)$ by Coifman and Weiss \cite{CW}.

Recall that Stein \cite{EMS} introduced the Littlewood--Paley operators,
namely, the Littlewood--Paley $g$-function, the
Lusin-area function $S$, and the Littlewood--Paley $g_{\lambda}^*$ function.
Since then, estimates of Littlewood--Paley operators on Campanato-type spaces
have attracted more and more attention.
For instance, Wang \cite{W85} proved that,
for any $f\in \mathop{\mathrm{BMO}\,}(\rn)$, $g(f)$
is either infinite everywhere or finite almost everywhere and, in the latter case,
$g(f)$ is bounded on $\mathop{\mathrm{BMO}\,}(\rn)$.
Later, Lu et al. \cite{LTY96} extended the main result of \cite{W85}
to Campanato spaces.
Moreover, Meng et al. \cite{MNY2010} studied the boundedness of the
Lusin-area function $S$ and the Littlewood--Paley $g_{\lambda}^*$ function
on Campanato spaces over spaces of homogeneous type
in the sense of Coifman and Weiss.
Recently, Li et al. \cite{LHY2020} studied the boundedness of multilinear
Littlewood--Paley operators on amalgam-Campanato spaces.
See also, for instance, \cite{HXMY2015, lnyz, MNY2010, MY, Q92,Sun04}
for more studies on Littlewood--Paley operators on Campanato-type spaces.

Sawano et al. \cite{SHYY}
introduced the ball quasi-Banach function space $X$ and
the Hardy space $H_X(\rn)$ associated with $X$.
Ball quasi-Banach function spaces contain many well-known function spaces, for
instance, weighted Lebesgue spaces, Morrey spaces, mixed-norm Lebesgue spaces,
Orlicz--slice spaces, and Musielak--Orlicz spaces,
which implies that the results obtained in \cite{SHYY}
have wide applications.
We refer the reader to
\cite{CWYZ2020,syy,syy2,tyyz,wyy,wyyz,yhyy,yhyy1,zyyw}
for more studies on Hardy-type spaces associated with
ball quasi-Banach function spaces.

As a counterpart of the dual relation
$(H^p(\rn))^*=\mathcal{C}_{\frac{1}{p}-1,q,s}(\rn)$,
the dual space of $H_X(\rn)$ has been widely studied
when $X$ is some concrete function space (see, for instance,
\cite{G1979,HLYY2019,LY2013,NS2012,NS2014,zyyw2019}).
Moreover, Yan et al. \cite{yyy20} introduced the Campanato-type function space
$\mathcal{L}_{X,q,s}(\rn)$ associated with $X$,
which proves the dual space of $H_X(\rn)$ when $X$ is concave.
Very recently, Zhang et al. \cite{zhyy2022}
improved the dual result in \cite{yyy20}
via removing the assumption that $X$ is concave.
To be precise, Zhang et al. \cite{zhyy2022} introduced
a new ball Campanato-type function space $\mathcal{L}_{X,q,s,d}(\rn)$
associated with $X$, and proved that 
$$\left(H_X(\rn)\right)^*=\mathcal{L}_{X,q,s,d}(\rn).$$
However, the mapping properties of Littlewood--Paley operators on
$\mathcal{L}_{X,q,s,d}(\rn)$ are still unknown
even when $X$ is some concrete function space, for instance,
the weighted Lebesgue space, the mixed-norm Lebesgue space,
the variable Lebesgue space, the Orlicz space, and the Orlicz-slice space.

In this article, we study the mapping properties of
Littlewood--Paley operators on the ball Campanato space
$\mathcal{L}_{X,q,0,d}(\rn)$ under the assumption that
the Hardy--Littlewood maximal operator satisfies the
Fefferman--Stein vector-valued maximal inequality on $X$
(see, Assumption \ref{assump1} below).
To be precise, we prove that,
for any $f\in \mathcal{L}_{X,q,0,d}(\mathbb{R}^n)$,
the Littlewood--Paley $g$-function $g(f)$ is either infinite everywhere or finite
almost everywhere and, in the latter case,
$g$ is bounded on $\mathcal{L}_{X,q,0,d}(\mathbb{R}^n)$
(see Theorem \ref{thm-B-Banach} below).
Similar results for both the Lusin-area function and
the Littlewood--Paley $g_\lambda^*$-function
are also obtained (see Theorems \ref{thm-B-Banach-S}
and \ref{g-l-bounded4} below). All these results have a
wide range of applications. Particularly, even when
$X:=L^p_{\omega}(\rn)$ (the weighted Lebesgue space),
$X:=L^{\vec{p}}(\mathbb{R}^n)$
(the mixed-norm Lebesgue space), $X:=L^{p(\cdot)}(\mathbb{R}^n)$
(the variable Lebesgue space), $X:=L^{\Phi}(\mathbb{R}^n)$ (the Orlicz space),
and $X:=(E_{\Phi}^q)_t(\mathbb{R}^n)$ (the Orlicz-slice space), all
these results are new.
The proofs of these results strongly depend on several delicate
estimates of Littlewood--Paley operators on the mean oscillation
of the locally integrable function $f$ on $\rn$.
Due to the generality and the practicability, more
applications of these results are predictable.
Moreover, the same ideas are also used to obtain the corresponding results
for $X:=JN_{p,q,\alpha}^{\mathrm{con}}(\rn)$
(the special John--Nirenberg--Campanato space via congruent cubes).

Indeed, the proofs of the boundedness of Littlewood--Paley operators on Campanato-type spaces
in \cite{L, MNY2010, Sun04} strongly depend on the
definition of the norm of Campanato-type spaces,
in which there exists only one cube. However, in the definition
of the norm of $\mathcal{L}_{X,q,0,d}(\rn)$, there appear countable
balls which cause additional essential difficulties in the proofs of our main results.
To overcome these difficulties, we establish
several delicate estimates of Littlewood--Paley functions
over the mean oscillation of the locally integrable function $f$
(see Lemmas \ref{lem-g-fz}, \ref{lem-s-fz}, and \ref{g-lam-01} below).
Moreover, we make a delicate division of $\rn$ in the proof of the boundedness of
Littlewood--Paley $g_\lambda^*$-function on $\mathcal{L}_{X,q,0,d}(\mathbb{R}^n)$,
and then improve the existing results via widening the ranges of
$\lambda$ in the Littlewood--Paley $g_\lambda^*$-function
(see Theorem \ref{g-l-bounded4} below).

The remainder of this article is organized as follows.

Section \ref{s1} is divided into two subsections to give some basic concepts.
In Subsection \ref{s1-1}, we first recall the concepts
of quasi-Banach function spaces $X$
and two ball Campanato-type function spaces, both
$\mathcal{L}_{X,q,s}(\rn)$ and $\mathcal{L}_{X,q,s,d}(\rn)$, associated with $X$.
Moreover, we present a mild assumption about the
Fefferman--Stein vector-valued maximal inequality on $X$
(see Assumption \ref{assump1} below).
In Subsection \ref{defin of L-P}, we present the definitions
of several general Littlewood--Paley operators,
including the Littlewood--Paley $g$-function,
the Lusin-area function $S$, and the Littlewood--Paley $g_\lambda^*$-function.

Section \ref{g-function} is devoted to studying the mapping properties
of the Littlewood--Paley $g$-function on both
$\mathcal{L}_{X,q,0,d}(\rn)$ and $\mathcal{L}_{X,q,0}(\rn)$. To be precise,
we first establish a technical lemma about $g$-function via borrowing some ideas from \cite{W85}
(see Lemma \ref{lem-g-fz} below).
Then we show that, for any $f\in \mathcal{L}_{X,q,0,d}(\rn)$,
$g(f)$ is either infinite everywhere or finite
almost everywhere and, in the latter case,
$g$ is bounded on $\mathcal{L}_{X,q,0,d}(\rn)$ (see Theorem \ref{thm-B-Banach} below).
Estimates of the Littlewood--Paley $g$-function
on $\mathcal{L}_{X,q,0}(\rn)$ are also obtained in Theorem \ref{thm-B-Banach} below.

In Section \ref{Lusin},
we study the mapping properties of the Lusin-area function $S$ on both
$\mathcal{L}_{X,q,0,d}(\rn)$ and
$\mathcal{L}_{X,q,0}(\rn)$. To be precise, we prove that,
for any $f\in \mathcal{L}_{X,q,0,d}(\rn)$,
$S(f)$ is either infinite everywhere or finite
almost everywhere and, in the latter case,
$S$ is bounded on $\mathcal{L}_{X,q,0,d}(\rn)$
(see Theorem \ref{thm-B-Banach-S} below).
Similar result for the Lusin-area function $S$
on $\mathcal{L}_{X,q,0}(\rn)$ is also obtained.

Section \ref{g*bounded} is devoted to studying the mapping properties of the Littlewood--Paley
$g_\lambda^*$-function on both $\mathcal{L}_{X,q,0}(\rn)$ and $\mathcal{L}_{X,q,0,s}(\rn)$.
To be precise, for any $f\in \mathcal{L}_{X,q,0,d}(\rn)$,
$g_\lambda^*$ is either infinite everywhere or finite
almost everywhere and, in the latter case,
$g_\lambda^*$ is bounded on $\mathcal{L}_{X,q,0,d}(\rn)$
(see Theorem \ref{g-l-bounded4} below).
Similar result for $\mathcal{L}_{X,q,0}(\rn)$ is also obtained in Theorem \ref{g-l-bounded3} below.
Both Theorems \ref{g-l-bounded3} and \ref{g-l-bounded4} even
when $X:=L^p(\rn)$ with $p\in(\frac{n}{n+1},\fz)$, $\dz:=1$, and
$q\in[1,\fz)$ also
improve the existing results via widening the ranges of $\lambda$
[see Remark \ref{g-l-bound-rem-B}(iii) below].

In Section \ref{S-ap}, we apply the main results obtained
in Sections \ref{g-function}, \ref{Lusin}, and \ref{g*bounded}
to five concrete ball quasi-Banach function spaces,
namely, the weighted Lebesgue space $L^p_{\omega}(\rn)$,
the mixed-norm Lebesgue space $L^{\vec{p}}(\rn)$,
the variable Lebesgue space $L^{p(\cdot)}(\rn)$,
the Orlicz space $L^{\Phi}(\rn)$, and the Orlicz-slice space $(E_\Phi^r)_t(\rn)$.
Therefore, all the mapping properties of Littlewood--Paley functions on
Campanato-type spaces $\mathcal{L}_{L^{p}_{\omega}(\rn),q,0,d}(\rn)$,
$\mathcal{L}_{L^{\vec{p}}(\rn),q,0,d}(\rn)$,
$\mathcal{L}_{L^{p(\cdot)}(\rn),q,0,d}(\rn)$,
$\mathcal{L}_{L^{\Phi}(\rn),q,0,d}(\rn)$, and
$\mathcal{L}_{(E_\Phi^r)_t(\rn),q,0,d}(\rn)$ are obtained.
Moreover, using the same ideas,
we also obtain the mapping properties of Littlewood--Paley operators on
the special John--Nirenberg--Campanato spaces via congruent cubes,
$JN_{p,q,\alpha}^{\mathrm{con}}(\rn)$,
introduced in \cite[Definition 1.3]{jtyyz1}
(see Theorems \ref{g-bounded}, \ref{S-bounded''}, and \ref{g*-bounded} below).

Finally, we make some conventions on notation. Let
$\nn:=\{1,2,\ldots\}$, $\zz_+:=\nn\cup\{0\}$, and $\zz_+^n:=(\zz_+)^n$.
We always denote by $C$ a \emph{positive constant}
which is independent of the main parameters,
but it may vary from line to line.
The symbol $f\lesssim g$ means that $f\le Cg$.
If $f\lesssim g$ and $g\lesssim f$, we then write $f\sim g$.
If $f\le Cg$ and $g=h$ or $g\le h$, we then write $f\ls g\sim h$
or $f\ls g\ls h$, \emph{rather than} $f\ls g=h$
or $f\ls g\le h$. For any $q\in[1,\infty]$,
we use $q'$ to denote the \emph{conjugate exponent} of $p$,
namely, $\frac{1}{q}+\frac{1}{q'}=1$.
For any measurable subset $E$ of $\rn$, we denote by $\mathbf{1}_E$ its
characteristic function. We use $\mathbf{0}$ to
denote the \emph{origin} of $\rn$.
Moreover, for any $x\in\rn$ and $r\in(0,\fz)$,
let $B(x,r):=\{y\in\rn:\ |y-x|<r\}$.
Furthermore, for any $\lambda\in(0,\infty)$
and any ball $B(x,r)\subset\rn$ with $x\in\rn$ and
$r\in(0,\fz)$, let $\lambda B(x,r):=B(x,\lambda r)$.

\section{Preliminaries\label{s1}}

In this section, we give some basic concepts.
In Subsection \ref{s1-1}, we present the definitions of both the
ball quasi-Banach function space $X$ and its related
ball Campanato-type function space.
In Subsection \ref{defin of L-P}, we recall the definition of Littlewood--Paley operators
under consideration.

\subsection{Campanato-Type Function Spaces Associated with $X$\label{s1-1}}

The ball quasi-Banach function space was introduced in \cite[Definition 2.2]{SHYY}.
In what follows, we use $\mathscr M(\rn)$ to denote the set of
all measurable functions on $\rn$. Moreover,
for any $x\in\rn$ and $r\in(0,\infty)$, let $B(x,r):=\{y\in\rn:\ |x-y|<r\}$ and
\begin{equation}\label{Eqball}
{\mathbb{B}}(\rn):=\lf\{B(x,r):\ x\in\rn \text{ and } r\in(0,\infty)\r\}.
\end{equation}

\begin{definition}\label{Debqfs}
Let $X\subset\mathscr M(\rn)$ be a quasi-normed linear space equipped with a quasi-norm
$\|\cdot\|_{X}$ which makes sense for all measurable functions on $\rn$. Then $X$ is called a
\emph{ball quasi-Banach function space} if it satisfies:
\begin{enumerate}
\item[(i)] for any $f\in \mathscr M(\rn)$, $\|f\|_X=0$ implies that $f=0$ almost everywhere;

\item[(ii)] for any $f,g\in \mathscr M(\rn)$,
$|g|\le |f|$ almost everywhere implies that $\|g\|_X\le\|f\|_X$;

\item[(iii)] for any $\{f_m\}_{m\in\nn}\subset \mathscr M(\rn)$ and $f\in \mathscr M(\rn)$,
$0\le f_m\uparrow f$ almost everywhere as $m\to\infty$
implies that $\|f_m\|_X\uparrow\|f\|_X$ as $m\to\infty$;

\item[(iv)] $B\in{\mathbb{B}}(\rn)$ implies that $\mathbf{1}_B\in X$,
where ${\mathbb{B}}(\rn)$ is the same as in \eqref{Eqball}.
\end{enumerate}
\end{definition}

\begin{remark}\label{rem-ball-B}
\begin{enumerate}
\item[$\mathrm{(i)}$] Let $X$ be a ball quasi-Banach
function space on $\rn$. Using \cite[Remark 2.6(i)]{yhyy1},
we find that, for any $f\in \mathscr M(\rn)$, $\|f\|_{X}=0$ if and only if $f=0$
almost everywhere.

\item[$\mathrm{(ii)}$] As was mentioned in
\cite[Remark 2.6(ii)]{yhyy1}, we obtain an
equivalent formulation of Definition \ref{Debqfs}
via replacing any ball $B\in\mathbb{B}(\rn)$ by any
bounded measurable set $E$ therein.	

\item[$\mathrm{(iii)}$]
From \cite[Theorem 2]{dfmn2021},
we deduce that both (ii) and (iii) of Definition \eqref{Debqfs}
imply that $X$ is complete and hence a quasi-Banach space.
\end{enumerate}
\end{remark}

Now, we recall the concepts of the convexity and
the concavity of ball quasi-Banach function spaces,
which is a part of \cite[Definition 2.6]{SHYY}.

\begin{definition}\label{Debf}
Let $X$ be a ball quasi-Banach function space and $p\in(0,\infty)$.
\begin{enumerate}
\item[(i)] The \emph{$p$-convexification} $X^p$ of $X$
is defined by setting
$$X^p:=\lf\{f\in\mathscr M(\rn):\ |f|^p\in X\r\}$$
equipped with the \emph{quasi-norm} $\|f\|_{X^p}:=\|\,|f|^p\,\|_X^{1/p}$
for any $f\in X^p$.

\item[(ii)] The space $X$ is said to be
\emph{concave} if there exists a positive constant
$C$ such that, for any $\{f_k\}_{k\in{\mathbb N}}\subset \mathscr M(\rn)$,
$$\sum_{k=1}^{{\infty}}\|f_k\|_{X}
\le C\left\|\sum_{k=1}^{{\infty}}|f_k|\right\|_{X}.$$
In particular, when $C=1$, $X$ is said to be
\emph{strictly concave}.
\end{enumerate}
\end{definition}

Next, we present the definition of the ball Campanato-type function space
associated with $X$. Recall that the \emph{Lebesgue space} $L^q(\rn)$ with $q\in(0,\infty]$
is defined to be the set of all the measurable functions $f$ on $\rn$ such that
$$\|f\|_{L^q(\rn)}:=
\begin{cases}
\displaystyle
\lf[\int_{\rn}|f(x)|^q\, dx\r]^{\frac{1}{q}}
&\text{if}\quad q\in(0,\fz),\\
\displaystyle
\mathop{\mathrm{ess\,sup}}_{x\in\rn}\,|f(x)|
&\text{if}\quad q=\fz
\end{cases}$$
is finite. Besides, for any given $q\in(0,\infty)$,
we use $L^q_{\mathrm{loc}}(\rn)$ to denote
the set of all the measurable functions $f$ such that
$f\mathbf{1}_E\in L^q(\rn)$
for any bounded measurable set $E\subset \rn$
and, for any $f\in L_{\loc}^1(\rn)$
and any finite measurable subset $E\subset\rn$, let
$$f_E:=\fint_Ef(x)\,dx:=\frac{1}{|E|}\int_E f(x)\,dx.$$
Moreover, for any $s\in\zz_+$, we use $\mathcal{P}_s(\rn)$
to denote the set of all the polynomials on $\rn$
with total degree not greater than $s$; for any ball $B\subset\rn$ and
any $g\in L^1_{\loc}(\rn)$, $P^{(s)}_B(g)$
denotes the \emph{minimizing polynomial} of
$g$ with degree not greater than $s$, namely,
$P^{(s)}_B(g)$ is the unique polynomial in $\mathcal{P}_s(\rn)$
such that, for any $P\in\mathcal{P}_s(\rn)$,
$$\int_{B}\lf[g(x)-P^{(s)}_B(g)(x)\r]P(x)\,dx=0.$$
Recall that Yan et al. \cite[Definition 1.11]{yyy20} introduced the following
Campanato space $\mathcal{L}_{X,q,s}(\rn)$ associated with
the ball quasi-Banach function space $X$.

\begin{definition}\label{cqb}
Let $X$ be a ball quasi-Banach function space, $q\in[1,\fz)$,
and $s\in\zz_+$. Then the \emph{Campanato space}
$\mathcal{L}_{X,q,s}(\rn)$, associated with $X$, is defined to be
the set of all the $f\in L^q_{\mathrm{loc}}(\rn)$ such that
$$\|f\|_{\mathcal{L}_{X,q,s}(\rn)}
:=\sup_{B\in {\mathbb{B}}(\rn)}\frac{|B|}{\|\mathbf{1}_B\|_X}
\lf[\fint_{B}\lf|f(x)-P^{(s)}_B(f)(x)\r|^q\,dx\r]^{\frac{1}{q}}$$
is finite.
\end{definition}

In what follows, by abuse of notation, we identify $f\in \mathcal{L}_{X,q,s}(\rn)$
with $f+\mathcal{P}_s(\rn)$.

\begin{remark}\label{rem-ball-B3}
\begin{enumerate}
\item[\rm (i)]	
Let $f\in L^{1}_{\mathrm{loc}}(\rn)$, $s\in\zz_+$, and
$B\in{\mathbb{B}}(\rn)$. By \cite[p.\,54, Lemma 4.1]{L},
we find that there exists a positive
constant $C$, independent of both $f$ and $B$, such that
\begin{align}\label{rem-ball-B2-01}
\left\|P_{B}^{(s)}(f)\right\|_{L^{\fz}(B)}
\le C\fint_B|f(x)|\,dx.
\end{align}
Moreover, it is well known that $P_{B}^{(0)}(f)=f_B$.

\item[\rm (ii)]	
Let $q\in[1,\infty)$, $s\in\zz_+$, and $\alpha\in[0,\fz)$.
Recall that Campanato \cite{C} introduced the \emph{Campanato space}
$\mathcal{C}_{\alpha,q,s}(\rn)$ which is defined to be the
set of all the $f\in L^q_\loc(\rn)$ such that
\begin{align}\label{campanato}
\|f\|_{\mathcal{C}_{\alpha,q,s}(\rn)} \
:=\sup |B|^{-\alpha}\left[\fint_{B}\left|f(x)
-P^{(s)}_B(f)(x)\right|^{q}\right]^{\frac{1}{q}}<\infty,
\end{align}
where the supremum is taken over all balls $B\in \mathbb{B}(\rn)$.
It is well known that $\mathcal{C}_{\alpha,q,s}(\rn)$ when $\alpha=0$
coincides with the space $\mathrm{BMO}\,(\rn)$.
Moreover, when $X:=L^{\frac{1}{\alpha+1}}(\rn)$, then
$\mathcal{L}_{X,q,s}(\rn)=\mathcal{C}_{\alpha,q,s}(\rn)$.
\end{enumerate}
\end{remark}

Recall that, in \cite[Definition 3.2]{zhyy2022}, Zhang et al. introduced the following
new ball Campanato space $\mathcal{L}_{X,q,s,d}(\rn)$ associated with
the ball quasi-Banach function space $X$.

\begin{definition}\label{2d2}
Let $X$ be a ball quasi-Banach function space, $q\in[1,{\infty})$, $d\in(0,\infty)$,
and $s\in\zz_+$. Then the \emph{ball Campanato-type function space}
$\mathcal{L}_{X,q,s,d}(\rn)$, associated with $X$, is defined to be
the set of all the $f\in L^q_{\mathrm{loc}}({{\rr}^n})$ such that
\begin{align*}
\|f\|_{\mathcal{L}_{X,q,s,d}(\rn)}
:&=\sup
\lf\|\lf\{\sum_{i=1}^m
\lf(\frac{{\lambda}_i}{\|{\mathbf{1}}_{B_i}\|_X}\r)^d
{\mathbf{1}}_{B_i}\r\}^{\frac1d}\r\|_{X}^{-1}\\
&\quad\times\sum_{j=1}^m\frac{{\lambda}_j|B_j|}{\|{\mathbf{1}}_{B_j}\|_{X}}
\lf[\fint_{B_j}\lf|f(x)-P^{(s)}_{B_j}(f)(x)\r|^q \,dx\r]^\frac1q
\end{align*}
is finite, where the supremum is taken over all
$m\in\nn$, $\{B_j\}_{j=1}^m\subset {\mathbb{B}}(\rn)$, and
$\{\lambda_j\}_{j=1}^m\subset[0,\infty)$ with $\sum_{j=1}^m\lambda_j\neq0$.
\end{definition}

In what follows, by abuse of notation, we identify $f\in \mathcal{L}_{X,q,s,d}(\rn)$
with $f+\mathcal{P}_s(\rn)$.

\begin{remark}\label{rem-ball-B2}
Let $q\in[1,\fz)$, $s\in\zz_+$, $d\in(0,1]$, and
$X$ be a concave ball quasi-Banach function space.
It was proved in \cite[Proposition 3.7]{zhyy2022} that
$\mathcal{L}_{X,q,s,d}({{\rr}^n})=\mathcal{L}_{X,q,s}(\rn)$
with equivalent quasi-norms.
\end{remark}

To obtain the main results of this article,
we need the following mild assumption of the
Fefferman--Stein vector-valued maximal inequality on $X$
(see, for instance, \cite[Assumption 2.6]{zhyy2022}).

\begin{assumption}\label{assump1}
Let $X$ be a ball quasi-Banach function
space. Assume that there exists a $p_-\in(0,\infty)$ such that,
for any given $p\in(0,p_-)$ and $u\in(1,\infty)$,
there exists a positive constant $C$ such that,
for any $\{f_j\}_{j=1}^\infty\subset\mathscr M(\rn)$,
\begin{align}\label{assume-fs}
\lf\|\lf\{\sum_{j\in\nn}\lf[\mathcal{M}(f_j)\r]^u\r\}^{\frac{1}{u}}\r\|_{X^{1/p}}
\le C\lf\|\lf(\sum_{j\in\nn}|f_j|^u\r)^{\frac{1}{u}}\r\|_{X^{1/p}}.
\end{align}
\end{assumption}

\subsection{Littlewood--Paley Operators\label{defin of L-P}}

In this subsection, we recall the definition of Littlewood--Paley operators.
We first present the concept of Littlewood--Paley functions.

\begin{definition}\label{phi1}
A scalar-valued function $\varphi\in L^1(\rn)$ is called a
\emph{Littlewood--Paley function}
if there exists a positive constant $C$ and a $\dz\in(0,\fz)$ such that
\begin{enumerate}
\item[\rm (i)]
\begin{align}\label{g-1}
\int_{\rn}\varphi(x)\,dx=0;
\end{align}

\item[\rm (ii)]
for any $x\in \rn$,
\begin{align}\label{g-2}
|\varphi(x)|
\leq\frac{C}{(1+|x|)^{n+\delta}};
\end{align}

\item[\rm (iii)]
for any $x\in\rn$,
\begin{align}\label{g-3'}
|\nabla \varphi(x)|\leq\frac{C}{(1+|x|)^{n+1+\dz}},
\end{align}
here and thereafter,
$\nabla:=(\frac{\partial}{\partial x_1},\ldots,\frac{\partial}{\partial x_n})$.
\end{enumerate}
\end{definition}

Now, we recall the concept of Littlewood--Paley operators,
namely, the Littlewood--Paley $g$-function,
the Lusin-area function $S$, and the
Littlewood--Paley $g_\lambda^*$-function.
In what follows, for any measurable function
$\varphi$ on $\rn$, and for any $t\in(0,\fz)$,
let $$\varphi_t(\cdot):=\frac{1}{t^n}\varphi\lf(\frac{\cdot}{t}\r).$$
Moreover, let
$$\rr^{n+1}_+
:=\lf\{(\xi,t)\in \rr^{n+1}:\ \xi\in\rn,\ t\in(0,\fz)\r\}$$
and, for any $x\in\rn$,
$$\Gamma(x):=\lf\{(\xi,t)\in\rr^{n+1}_+:\ |\xi-x|<t\r\}.$$

\begin{definition}\label{g-f}
Let $\varphi$ be the same as in Definition \ref{phi1}.
For any $f\in L^1_{\mathrm{loc}}(\rn)$,
the \textit{Littlewood--Paley $g$-function $g(f)$},
the \textit{Lusin-area function $S(f)$}, and the
\textit{Littlewood--Paley $g_\lambda^*$-function $g_\lambda^*(f)$}
with $\lambda\in(1,\fz)$ of $f$ are defined, respectively,
by setting, for any $x\in\rn$,
\begin{align}\label{g-4}
g(f)(x):=\lf[\int_0^{\infty}\lf|f\ast \varphi_t(x)\r|^2
\,\frac{dt}{t}\r]^{\frac{1}{2}},
\end{align}
\begin{align}\label{S}
S(f)(x):=\lf[\iint_{\Gamma(x)}\lf|f\ast\varphi_t(\xi)
\r|^2\,\frac{d\xi\,dt}{t^{n+1}}\r]^{\frac{1}{2}},
\end{align}
and
\begin{align}\label{lam-function}
g_{\lambda}^{*}(f)(x):=\lf[\iint_{\rr_+^{n+1}}
\lf(\frac{t}{t+|x-\xi|}\r)^{\lambda n}
\lf|f\ast \varphi_t(\xi)\r|^2\,\frac{d\xi\,dt}{t^{n+1}}\r]^{\frac{1}{2}}.
\end{align}
\end{definition}

\section{Littlewood--Paley $g$-Functions on Both
$\mathcal{L}_{X,q,0}(\rn)$ and $\mathcal{L}_{X,q,0,d}(\rn)$\label{g-function}}

This section is devoted to studying the boundedness of the Littlewood--Paley
$g$ function on both $\mathcal{L}_{X,q,0}(\rn)$ and $\mathcal{L}_{X,q,0,d}(\rn)$.
To this end, we first introduce the subspace
$\mathcal{L}_{X,q,0}^{\mathrm{low}}(\rn)$ of $\mathcal{L}_{X,q,s}(\rn)$,
and the subspace $\mathcal{L}_{X,q,0,d}^{\mathrm{low}}(\rn)$
of $\mathcal{L}_{X,q,s,d}(\rn)$.

\begin{definition}\label{cqb'}
Let $X$ be a ball quasi-Banach function space and $q\in[1,\fz)$.
Then the \emph{space} $\mathcal{L}_{X,q,0}^{\mathrm{low}}(\rn)$
is defined to be the set of all the $f\in L^q_{\mathrm{loc}}(\rn)$ such that
$$\|f\|_{\mathcal{L}_{X,q,0}^{\mathrm{low}}(\rn)}
:=\sup_{B\subset\rn}\frac{|B|}{\|\mathbf{1}_B\|_X}\lf[
\fint_{B}\lf\{f(x)-\inf_{\wz x\in B}f(\wz x)\r\}^q\,dx\r]^{\frac{1}{q}}<\infty,$$
where the supremum is taken over all balls $B\in{\mathbb{B}}(\rn)$.
\end{definition}

\begin{definition}\label{2d2'}
Let $X$ be a ball quasi-Banach function space, $q\in[1,{\infty})$, and $d\in(0,\infty)$.
Then the \emph{space} $\mathcal{L}_{X,q,0,d}^{\mathrm{low}}(\rn)$ is defined to be
the set of all the $f\in L^q_{\mathrm{loc}}({{\rr}^n})$ such that
\begin{align*}
\|f\|_{\mathcal{L}_{X,q,0,d}^{\mathrm{low}}(\rn)}
:=&\,\sup
\lf\|\lf\{\sum_{i=1}^m
\lf(\frac{{\lambda}_i}{\|{\mathbf{1}}_{B_i}\|_X}\r)^d
{\mathbf{1}}_{B_i}\r\}^{\frac1d}\r\|_{X}^{-1}
\\
&\quad\times\sum_{j=1}^m\frac{{\lambda}_j|B_j|}{\|{\mathbf{1}}_{B_j}
\|_{X}}\lf[\fint_{B_j}\lf\{f(x)
-\inf_{\wz x\in B_j}f(\wz x)\r\}^q \,dx\r]^\frac1q
\end{align*}
is finite, where the supremum is taken over all
$m\in\nn$, $\{B_j\}_{j=1}^m\subset {\mathbb{B}}(\rn)$, and
$\{\lambda_j\}_{j=1}^m\subset[0,\infty)$ with $\sum_{j=1}^m\lambda_j\neq0$.
\end{definition}

\begin{remark}
Let $X$ be a concave ball quasi-Banach function space,
$q\in[1,{\infty})$, and $d\in(0,1]$. In this case,
by a proof similar to that of \cite[Proposition 3.7]{zhyy2022},
we find that $\mathcal{L}_{X,q,0,d}^{\mathrm{low}}({{\rr}^n})
=\mathcal{L}_{X,q,0}^{\mathrm{low}}(\rn)$
with equivalent quasi-norms.
\end{remark}

Next, we show that $\mathcal{L}_{X,q,0}^{\mathrm{low}}(\rn)\subset \mathcal{L}_{X,q,s}(\rn)$
and $\mathcal{L}_{X,q,0,d}^{\mathrm{low}}(\rn)\subset \mathcal{L}_{X,q,s,d}(\rn)$.

\begin{proposition}\label{pro-sub}
Let $X$ be a ball quasi-Banach function space,
$q\in[1,{\infty})$, $s\in\zz_+$, and $d\in(0,\infty)$.
Then $\mathcal{L}_{X,q,0}^{\mathrm{low}}(\rn)\subset \mathcal{L}_{X,q,s}(\rn)$
and $\mathcal{L}_{X,q,0,d}^{\mathrm{low}}(\rn)\subset \mathcal{L}_{X,q,s,d}(\rn)$.
\end{proposition}

\begin{proof}
Let $X$, $q$, $s$, and $d$ be the same as in the present proposition.
From \eqref{rem-ball-B2-01}, the fact that $P_{B}^{(s)}(P)=P$ for any
ball $B\in \mathbb B(\rn)$ and any $P\in \mathcal{P}_s(\rn)$, and the H\"older inequality,
we deduce that, for any $f\in L^1_{\loc}(\rn)$ and any ball $B\in \mathbb B(\rn)$,
\begin{align}\label{rem-sub}
&\lf[\fint_{B}\lf|f(x)-P_{B}^{(s)}(f)(x)\r|^q\,dx\r]^{\frac{1}{q}}\\
&\quad\leq\lf[\fint_{B}\lf\{f(x)
-\inf_{\widetilde{x}\in B}f(\widetilde{x})\r\}^q\,dx\r]^{\frac{1}{q}}
+\lf[\fint_{B}\lf|P_{B}^{(s)}(f)(x)
-\inf_{\widetilde{x}\in B}f(\widetilde{x})\r|^q\,dx\r]^{\frac{1}{q}}\noz\\
&\quad\leq\lf[\fint_{B}\lf\{f(x)
-\inf_{\widetilde{x}\in B}f(\widetilde{x})\r\}^q\,dx\r]^{\frac{1}{q}}
+\lf\|P_{B}^{(s)}\lf(f-\inf_{\widetilde{x}\in B}f(\widetilde{x})\r)\r\|_{L^{\fz}(B)}\noz\\
&\quad\ls\lf[\fint_{B}\lf\{f(x)
-\inf_{\widetilde{x}\in B}f(\widetilde{x})\r\}^q\,dx\r]^{\frac{1}{q}}
+\fint_{B}\lf\{f(x)-\inf_{\widetilde{x}\in B}f(\widetilde{x})\r\}\,dx\noz\\
&\quad\ls\lf[\fint_{B}\lf\{f(x)
-\inf_{\widetilde{x}\in B}f(\widetilde{x})\r\}^q\,dx\r]^{\frac{1}{q}}.\noz
\end{align}
This, together with the definitions of $\mathcal{L}_{X,q,0}^{\mathrm{low}}(\rn)$,
$\mathcal{L}_{X,q,s}(\rn)$, $\mathcal{L}_{X,q,0,d}^{\mathrm{low}}(\rn)$,
and $\mathcal{L}_{X,q,s,d}(\rn)$, then
finishes the proof of Proposition \ref{pro-sub}.	
\end{proof}

Now, we present the main result of this section.

\begin{theorem}\label{thm-B-Banach}
Let $\dz\in(0,1]$ and $g$ be the Littlewood--Paley $g$-function in \eqref{g-4}.
Let both $X$ and $p_-\in(\frac{n}{n+\dz},\fz)$ satisfy Assumption \ref{assump1},
$q\in(1,\fz)$, and $d\in(\frac{n}{n+\dz},\fz)$. Then,
for any $f\in \mathcal{L}_{X,q,0,d}(\rn)$ [resp., $f\in \mathcal{L}_{X,q,0}(\rn)$],
$g(f)$ is either infinite everywhere or finite almost everywhere and, in the latter case,
there exists a positive constant $C$, independent of $f$, such that
\begin{align}\label{thm-b-1}
\|g(f)\|_{\mathcal{L}^{\mathrm{low}}_{X,q,0,d}(\rn)}
\leq C\|f\|_{\mathcal{L}_{X,q,0,d}(\rn)}
\end{align}
\begin{align*}
\lf[\mathrm{resp}.,\ \|g(f)\|_{\mathcal{L}^{\mathrm{low}}_{X,q,0}(\rn)}
\leq C\|f\|_{\mathcal{L}_{X,q,0}(\rn)}\r].
\end{align*}
\end{theorem}

To show Theorem \ref{thm-B-Banach}, we first establish
five technical lemmas.
The following two lemmas are special
cases, respectively, of \cite[Lemmas 2.20 and 2.21]{jtyyz3}.

\begin{lemma}\label{sum-g}
Let $q\in[1,\infty)$ and $\theta\in(0,1)$.
Then there exists a positive constant $C$
such that, for any $f\in L^1_{\mathrm{loc}}(\rn)$
and any ball $B\subset\rn$,
\begin{align*}
\sum_{k=1}^{\infty}\theta^k\lf[\fint_{2^kB}\lf|f(x)
-f_{B}\r|^q\,dx\r]^{\frac{1}{q}}
\leq C\sum_{k=1}^{\fz}\theta^k\lf[\fint_{2^kB}\lf|f(x)
-f_{2^kB}\r|^q\,dx\r]^{\frac{1}{q}}.
\end{align*}
\end{lemma}

\begin{lemma}\label{I-JN}
Let $\dz\in(0,\fz)$.
Then there exists a positive constant $C$ such that,
for any $f\in L_{\mathrm{loc}}^{1}(\rn)$ and any
ball $B(y,r)\subset \rn$ with $y\in\rn$ and $r\in(0,\fz)$,
\begin{align*}
\int_{\rn\setminus B(y,r)}
\frac{r^{\dz}|f(z)-f_{B(y,r)}|}{|y-z|^{n+\dz}}\,dz
\leq C\sum_{k=1}^{\fz}\frac{1}{2^{k\dz}}
\fint_{2^{k}B(y,r)}\lf|f(z)-f_{2^{k}B(y,r)}\r|\,dz.
\end{align*}
\end{lemma}

The following lemma on the
boundedness of the Littlewood--Paley $g$-function
on Lebesgue spaces is well known
(see, for instance, \cite[p.\,312, (3.8)]{A.T1986}).

\begin{lemma}\label{g-Lp}
Let $q\in(1,\infty)$ and $g$ be
the Littlewood--Paley $g$-function in \eqref{g-4}.
Then $g$ is bounded on $L^q(\rn)$,
namely, there exists a positive constant $C$ such that,
for any $f\in L^q(\rn)$,
\begin{align*}
\lf\|g(f)\r\|_{L^q(\rn)}\leq C\|f\|_{L^q(\rn)}.
\end{align*}
\end{lemma}

We also need the following lemma, which can be easily obtained
by \eqref{assume-fs} and the fact that,
for any ball $B\in \mathbb{B}(\rn)$, any $\beta\in[1,\fz)$, and any $r\in(0,\fz)$,
$\mathbf{1}_{\beta B}\leq (\beta+1)^{\frac{dn}{r}}
[\mathcal{M}(\mathbf{1}_B)]^{\frac{d}{r}}$
(see \cite[Remark 2.11 (iii)]{cjy-01} for more details).

\begin{lemma}\label{key-rem-B}
Let both $X$ and $p_-$ satisfy Assumption \ref{assump1}, and $d\in(0,\fz)$.
Then, for any $\beta\in[1,\fz)$, any $r\in(0,\min\{d,p_-\})$,
any sequence $\{B_j\}_{j\in\nn}\subset {\mathbb{B}}(\rn)$,
and any $\{\lambda_j\}_{j\in\nn}\subset [0,\fz)$,
\begin{align*}
\lf\|\lf(\sum_{j\in\nn}\lambda_j^d\mathbf{1}_{\beta B_j}\r)^{\frac{1}{d}}\r\|_{X}
\leq(2\beta)^{\frac{n}{r}}\lf\|\lf(\sum_{j\in\nn}
\lambda_j^d\mathbf{1}_{B_j}\r)^{\frac{1}{d}}\r\|_{X}.
\end{align*}
Moreover, for any ball $B\in{\mathbb{B}}(\rn)$ and any $r\in(0,p_-)$,
$\|\mathbf{1}_{\beta B}\|_{X}\leq (2\beta)^{\frac{n}{r}}\|\mathbf{1}_{B}\|_{X}$.
\end{lemma}

Via borrowing some ideas from \cite{Sun04,W85}, we have the following lemma
which plays an important role in the proof of Theorem \ref{thm-B-Banach}.

\begin{lemma}\label{lem-g-fz}
Let $q\in(1,\fz)$, $\dz\in(0,1]$,
$g$ be the Littlewood--Paley $g$-function in \eqref{g-4},
and $f\in L^1_{\mathrm{loc}}(\rn)$.
If $g(f)(x_0)<\fz$ for one $x_0\in\rn$, then
there exists a positive constant $C$, independent of both $f$ and $x_0$, such that,
for any ball $B\in\mathbb{B}(\rn)$ containing $x_0$,
\begin{align}\label{lem-g-01}
\lf[\fint_{B}\lf\{g(f)(x)-\inf_{\widetilde{x}\in B}g(f)(\widetilde{x})
\r\}^q\,dx\r]^{\frac{1}{q}}
\leq C\sum_{k=1}^{\infty}\frac{1}{2^{k\delta}}
\lf[\fint_{2^{k}B}\lf|f(z)-f_{2^{k}B}\r|^q\,dz\r]^{\frac{1}{q}}.
\end{align}
\end{lemma}

\begin{proof}
Let $q$, $\dz$, and $g$ be the same as in the present lemma.
Let $f\in L^1_{\mathrm{loc}}(\rn)$ and assume that there exists one $x_0\in\rn$ such that
$g(f)(x_0)<\fz$. Also, let $B:=B(y,r)$ be any given ball of $\rn$
with $y\in\rn$ and $r\in(0,\fz)$ such that $x_0\in B$. Moreover,
for any $x\in \rn$, let
\begin{align*}
g_{r}(f)(x):=\lf[\int_0^{4r}\lf|f\ast \varphi_t(x)
\r|^2\,\frac{dt}{t}\r]^{\frac{1}{2}}
\end{align*}
and
\begin{align*}
g_{\infty}(f)(x):=\lf[\int_{4r}^{\infty}
\lf|f\ast \varphi_t(x)\r|^2\,\frac{dt}{t}\r]^{\frac{1}{2}}.
\end{align*}
Then, by \eqref{g-1} and $g(f)(x_0)<\fz$
[and hence $g_{\fz}(f)(x_0)<\fz$], we conclude that, for any $x\in B$,
\begin{align*}
0&\leq g(f)(x)-\inf_{\widetilde{x}\in B}g(f)(\widetilde{x})\\
&\leq g_{r}(f)(x)+g_{\infty}(f)(x)
-\inf_{\widetilde{x}\in B}g_\fz(f)(\widetilde{x})\\
&\leq g_{r}(f-f_{8B})(x)+\sup_{\{\wz x\in B:\ g_{\infty}(f)(\wz x)<\fz\}}
\lf|g_{\infty}(f)(x)-g_{\infty}(f)(\wz x)\r|\\
&\leq g_{r}\lf([f-f_{8B}]\mathbf{1}_{8B}\r)(x)
+g_{r}\lf([f-f_{8B}]\mathbf{1}_{\rn\setminus8B}\r)(x)\\
&\quad+\sup_{\{\wz x\in B:\ g_{\infty}(f)(\wz x)<\fz\}}
\lf|g_{\infty}(f)(x)-g_{\infty}(f)(\wz x)\r|
\end{align*}
and hence
\begin{align}\label{G-00}
&\lf[\fint_{B}\lf\{g(f)(x)-\inf_{\widetilde{x}\in B}g(f)(\widetilde{x})
\r\}^q\,dx\r]^{\frac{1}{q}}\\
&\quad\leq\lf[\fint_{B}\lf|g_{r}\lf([f-f_{8B}]
\mathbf{1}_{8B}\r)(x)\r|^q\,dx\r]^{\frac{1}{q}}\noz\\
&\qquad+\lf[\fint_{B}\lf|g_{r}\lf([f-f_{8B}]
\mathbf{1}_{\rn\setminus8B}\r)(x)\r|^q\,dx\r]^{\frac{1}{q}}\noz\\
&\qquad+\lf\{\fint_{B}\lf[\sup_{\{\wz x\in B:\ g_{\infty}(f)(\wz x)<\fz\}}
\lf|g_{\infty}(f)(x)-g_{\infty}(f)(\wz x)\r|\r]^q\,dx\r\}^{\frac{1}{q}}\noz\\
&\quad=:\mathrm{G}_1+\mathrm{G}_2+\mathrm{G}_3.\noz
\end{align}

We first estimate $\mathrm{G_1}$. Indeed, from Lemma \ref{g-Lp}, we deduce that
\begin{align}\label{3.3x}
\mathrm{G_1}
&\leq\lf[\frac{1}{|B|}\int_{\rn}
\lf|g\lf([f-f_{8B}]\mathbf{1}_{8B}\r)(x)\r|^q\,dx\r]^{\frac{1}{q}}
\lesssim\lf[\fint_{8B}
\lf|f(x)-f_{8B}\r|^q\,dx\r]^{\frac{1}{q}}.
\end{align}
This is a desired estimate of $\mathrm{G_1}$.

Next, we consider $\mathrm{G_2}$.
Notice that, for any $z\in \rn\setminus 8B$ and $x\in B$,
$$|x-z|\sim|y-z|.$$
Using this, \eqref{g-2}, and Lemma \ref{I-JN}, we conclude that, for any $x\in B$,
\begin{align*}
&g_{r}\lf(\lf[f-f_{8B}\r]\mathbf{1}_{\rn\setminus8B}\r)(x)\\
&\quad\leq\lf\{\int_0^{4r}
\lf[\int_{\rn\setminus8B}\lf|f(z)-f_{8B}
\r|\lf|\frac{1}{t^n}\varphi\lf(\frac{x-z}{t}\r)\r|\,dz
\r]^2\,\frac{dt}{t}\r\}^{\frac{1}{2}}\\
&\quad\lesssim\lf\{\int_{0}^{4r}\lf[\int_{\rn\setminus8B}
\frac{|f(z)-f_{8B}|}{(t+|x-z|)^{n+\delta}}\,dz
\r]^2 t^{2\delta-1}\,dt\r\}^{\frac{1}{2}}\\
&\quad\lesssim\int_{\rn\setminus8B}
\frac{r^\delta|f(z)-f_{8B}|}{|y-z|^{n+\delta}}\,dz
\lesssim\sum_{k=1}^{\infty}\frac{1}{2^{k\delta}}
\fint_{2^{k}B}\lf|f(z)-f_{2^{k}B}\r|\,dz
\end{align*}
and hence
\begin{align}\label{lem-g-02}
\mathrm{G_2}\ls \sum_{k=1}^{\infty}\frac{1}{2^{k\delta}}
\fint_{2^{k}B}\lf|f(z)-f_{2^{k}B}\r|\,dz.	
\end{align}
This is a desired estimate of $\mathrm{G_2}$.

Now, we estimate $\mathrm{G_3}$. To this end,
we first consider $|g_{\infty}(f)(x)-g_{\infty}(f)(\wz x)|$ for any $x\in B$
and any $\wz x\in B$ such that $g_{\fz}(f)(\wz x)<\fz$.
Indeed, from the fact that $g(f)\geq 0$ everywhere on $\rn$, we deduce that,
for any $x\in B$ and any $\wz x\in B$ such that $g_{\fz}(f)(\wz x)<\fz$,
\begin{align}\label{e3-0}
\lf|g_{\infty}(f)(x)-g_{\infty}(f)(\wz x)\r|
&=\lf|\lf[\int_{4r}^{\infty}\lf|f\ast \varphi_t(x)
\r|^2\,\frac{dt}{t}\r]^{\frac{1}{2}}
-\lf[\int_{4r}^{\infty}\lf|f\ast \varphi_t(\wz x)
\r|^2\,\frac{dt}{t}\r]^{\frac{1}{2}}\r|\\
&\leq\lf[\int_{4r}^{\infty}
\lf|f\ast\varphi_t(x)-f\ast\varphi_t(\wz x)
\r|^2\,\frac{dt}{t}\r]^{\frac{1}{2}}.\noz
\end{align}
Noticing that, for any $z\in\rn$
and $t\in(0,\infty)$,
$|\varphi_t(z)|\lesssim t^{-n}$,
by this, \eqref{g-1}, \eqref{g-2}, \eqref{g-3'},
the mean value theorem, and the fact that, for any
$\theta \in(0,1)$, $z\in \rn\setminus 8B$, and $x,\wz x\in B$,
$$|x-z+\theta(x-\wz x)|\sim |y-z|,$$
we find that, for any $t\in (4r,\fz)$, $x\in B$,
and $\wz x\in B$ such that $g_{\fz}(f)(\wz x)<\fz$,
there exists a $\theta\in(0,1)$ such that
\begin{align*}
&\lf|f\ast\varphi_t(x)-f\ast\varphi_t(\wz x)\r|\\
&\quad=\lf|(f-f_{8B})\ast\varphi_t(x)
-(f-f_{8B})\ast\varphi_t(\wz x)\r|\\
&\quad\leq\lf|\lf([f-f_{8B}]
\mathbf{1}_{\rn\setminus8B}\r)\ast\varphi_t(x)
-\lf([f-f_{8B}]\mathbf{1}_{\rn\setminus8B}\r)
\ast\varphi_t(\wz x)\r|\\
&\quad\quad+\lf|\lf([f-f_{8B}]
\mathbf{1}_{8B}\r)\ast\varphi_t(x)\r|
+\lf|\lf([f-f_{8B}]\mathbf{1}_{8B}\r)
\ast\varphi_t(\wz x)\r|\\
&\quad\leq\int_{\rn\setminus8B}\lf|f(z)-f_{8B}\r|
\lf|\varphi\lf(\frac{x-z}{t}\r)
-\varphi\lf(\frac{\wz x-z}{t}\r)\r|\,\frac{dz}{t^n}\\
&\quad\quad+\int_{8B}\lf|f(z)-f_{8B}\r|
\lf|\varphi\lf(\frac{x-z}{t}\r)\r|\,\frac{dz}{t^n}
+\int_{8B}\lf|f(z)-f_{8B}\r|
\lf|\varphi\lf(\frac{\wz x-z}{t}\r)\r|\,\frac{dz}{t^n}\\
&\quad\ls\int_{\rn\setminus8B}
\frac{t^\delta|f(z)-f_{8B}||x-\wz x|}
{(t+|x-z+\theta(x-\wz x)|)^{n+1+\delta}}\,dz
+\int_{8B}\frac{|f(z)-f_{8B}|}{t^n}\,dz\\
&\quad\sim\int_{\rn\setminus8B}\lf|f(z)-f_{8B}\r|
\frac{t^\delta|x-\wz x|}{(t+|y-z|)^{n+1+\delta}}\,dz
+\int_{8B}\frac{|f(z)-f_{8B}|}{t^n}\,dz.
\end{align*}
By this, \eqref{e3-0}, the Minkowski inequality,
and Lemma \ref{I-JN}, we conclude that, for any $x\in B$
and $\wz x\in B$ such that $g_{\fz}(f)(\wz x)<\fz$,
\begin{align}\label{g10}
&\lf|g_{\infty}(f)(x)-g_{\infty}(f)(\wz x)\r|\\
&\quad\lesssim\lf\{\int_{4r}^{\infty}\lf[\int_{\rn\setminus8B}
\lf|f(z)-f_{8B}\r|\frac{t^\delta|x-\wz x|}
{(t+|y-z|)^{n+1+\delta}}\,dz\r]^2\,\frac{dt}{t}\r\}^{\frac{1}{2}}\nonumber\\
&\quad\quad+\lf\{\int_{4r}^{\infty}
\lf[\int_{8B}\frac{|f(z)-f_{8B}|}{t^n}\,dz\r]^2
\,\frac{dt}{t}\r\}^{\frac{1}{2}}\nonumber\\
&\quad\lesssim\int_{\rn\setminus8B}r
\lf|f(z)-f_{8B}\r|\lf[\int_{4r}^{\infty}
\frac{t^{2\dz-1}}{(t+|y-z|)^{2n+2+2\delta}}\,dt\r]^{\frac{1}{2}}\,dz\nonumber\\
&\quad\quad+\int_{8B}\lf|f(z)-f_{8B}\r|
\lf(\int_{4r}^{\infty}\frac{1}{t^{2n+1}}\,dt\r)^\frac{1}{2}\,dz\nonumber\\
&\quad\sim\int_{\rn\setminus8B}r\lf|f(z)-f_{8B}\r|\lf[
\lf(\int_{4r}^{|y-z|}+\int_{|y-z|}^{\infty}\r)
\frac{t^{2\dz-1}}{(t+|y-z|)^{2n+2+2\delta}}\,dt\r]^\frac{1}{2}\,dz\nonumber\\
&\quad\quad+\frac{1}{r^n}\int_{8B}
\lf|f(z)-f_{8B}\r|\,dz\nonumber\\
&\quad\sim\int_{\rn\setminus8B}
\frac{r|f(z)-f_{8B}|}{|y-z|^{n+1}}\,dz
+\fint_{8B}\lf|f(z)-f_{8B}\r|\,dz\nonumber\\
&\quad\lesssim\sum_{k=1}^{\infty}\frac{1}{2^{k}}
\fint_{2^{k+1}B}\lf|f(z)-f_{2^{k+1}B}\r|\,dz,\noz
\end{align}
which, combined with \eqref{G-00}, \eqref{3.3x}, \eqref{lem-g-02}, $\dz\in(0,1]$,
and the H\"older inequality, further implies that
\eqref{lem-g-01} holds true. This finishes the proof of Lemma \ref{lem-g-fz}.
\end{proof}

Next, we show Theorem \ref{thm-B-Banach}

\begin{proof}[Proof of Theorem \ref{thm-B-Banach}]
Let $\dz$, $g$, $X$, $p_-$, $q$, and $d$ be the same as in the present theorem.
We only show the case $f\in \mathcal{L}_{X,q,0,d}(\rn)$
because the proof of $f\in \mathcal{L}_{X,q,0}(\rn)$ is similar.
We first claim that, if $g(f)(x_0)<\fz$ for one $x_0\in \rn$, then
$g(f)$ is finite almost everywhere. To this end,
let $m\in\nn$, $\{B_j\}_{j=1}^m\subset {\mathbb{B}}(\rn)$
satisfy that $x_0\in B_j$ for any $j\in\nn$,
and $\{\lambda_j\}_{j=1}^m\subset[0,\fz)$ that $\sum_{j\in\nn}\lambda_j\neq 0$.
Moreover, for any $i\in\{1,\ldots,m\}$ and $k\in\nn$, let
$\lambda_{i,k}:=\frac{\lambda_i\|\mathbf{1}_{2^kB_i}\|_{X}}{\|\mathbf{1}_{B_i}\|_{X}}$.
Then, from Lemma \ref{key-rem-B}, we deduce that, for
any given $r\in(0,\min\{d,p_-\})$ and any $k\in\nn$,
\begin{align}\label{g-thm-B-02}
\lf\|\lf\{\sum_{i=1}^m\lf(\frac{\lambda_{i,k}}{\|\mathbf{1}_{2^kB_i}\|_{X}}\r)^d
\mathbf{1}_{2^kB_i}\r\}^{\frac{1}{d}}\r\|_{X}
&\ls2^{\frac{kn}{r}}\lf\|\lf\{\sum_{i=1}^m
\lf(\frac{\lambda_{i,k}}{\|\mathbf{1}_{2^kB_i}\|_{X}}\r)^d
\mathbf{1}_{B_i}\r\}^{\frac{1}{d}}\r\|_{X}\\
&\sim2^{\frac{kn}{r}}\lf\|\lf\{\sum_{i=1}^m
\lf(\frac{\lambda_{i}}{\|\mathbf{1}_{B_i}\|_{X}}\r)^d
\mathbf{1}_{B_i}\r\}^{\frac{1}{d}}\r\|_{X}.\noz
\end{align}
By Lemma \ref{lem-g-fz}, \eqref{g-thm-B-02},
the Fubini theorem, the definition of
$\|\cdot\|_{\mathcal{L}_{X,q,0,d}(\rn)}$, and
$\min\{d,p_-\}\in(\frac{n}{n+\dz},\fz)$, we conclude that,
for any given $r\in(\frac{n}{n+\dz},\min\{d,p_-\})$,
\begin{align}\label{thm-B-07}
&\lf\|\lf\{\sum_{i=1}^m\lf(\frac{\lambda_{i}}{\|\mathbf{1}_{B_i}\|_{X}}\r)^d
\mathbf{1}_{B_i}\r\}^{\frac{1}{d}}\r\|_{X}^{-1}
\sum_{j=1}^m\frac{\lambda_j|B_j|}{\|\mathbf{1}_{B_j}\|_{X}}
\lf[\fint_{B_j}\lf\{g(f)(x)-\inf_{\wz x\in B_j}g(f)(\wz x)\r\}^q\,dx\r]^{\frac{1}{q}}\\
&\quad\ls\lf\|\lf\{\sum_{i=1}^m\lf(\frac{\lambda_{i}}{\|\mathbf{1}_{B_i}\|_{X}}\r)^d
\mathbf{1}_{B_i}\r\}^{\frac{1}{d}}\r\|_{X}^{-1}
\sum_{j=1}^m\frac{\lambda_j|B_j|}{\|\mathbf{1}_{B_j}\|_{X}}
\sum_{k\in\nn}\lf[\frac{1}{2^{k\dz}}\fint_{2^kB_j}
\lf|f(x)-f_{2^kB_j}\r|^q\,dx\r]^{\frac{1}{q}}\noz\\
&\quad\ls\sum_{k\in\nn}2^{k(-n-\dz+\frac{n}{r})}\lf\|\lf\{\sum_{i=1}^m
\lf(\frac{\lambda_{i,k}}{\|\mathbf{1}_{2^kB_i}\|_{X}}\r)^d
\mathbf{1}_{2^kB_i}\r\}^{\frac{1}{d}}\r\|_{X}^{-1}\noz\\
&\qquad\times\sum_{j=1}^m\frac{\lambda_{j,k}|2^kB_j|}
{\|\mathbf{1}_{2^kB_j}\|_{X}}\lf[\frac{1}{2^{k\dz}}\fint_{2^kB_j}
\lf|f(x)-f_{2^kB_j}\r|^q\,dx\r]^{\frac{1}{q}}\noz\\
&\quad\ls\sum_{k\in\nn}2^{k(-n-\dz+\frac{n}{r})}
\|f\|_{\mathcal{L}_{X,q,0,d}(\rn)}
\ls\|f\|_{\mathcal{L}_{X,q,0,d}(\rn)}<\fz.\noz
\end{align}
From this, we deduce that, for any $j\in\nn$,
$$\lf[\fint_{B_j}\lf\{g(f)(x)-\inf_{\wz x\in B_j}g(f)(\wz x)\r\}^q\,dx\r]^{\frac{1}{q}}<\fz$$
and hence $g(f)(x)<\fz$ for almost every $x\in B_j$, which, together with
the arbitrariness of $B_j$, further implies that, for almost every $x\in\rn$, $g(f)(x)<\fz$.
Thus, the above claim holds true.

Now, we show that \eqref{thm-b-1} holds true
if $g(f)(x_0)<\fz$ for one $x_0\in \rn$.
Indeed, by the above claim and an argument similar to the estimation of \eqref{thm-B-07},
we conclude that, for any $m\in\nn$, any balls $\{B_j\}_{j=1}^m\subset {\mathbb{B}}(\rn)$,
and any $\{\lambda_j\}_{j=1}^m\subset[0,\fz)$ with $\sum_{j=1}^m\lambda_j\neq 0$,
\begin{align*}
&\lf\|\lf\{\sum_{i=1}^m\lf(\frac{\lambda_{i}}{\|\mathbf{1}_{B_i}\|_{X}}\r)^d
\mathbf{1}_{B_i}\r\}^{\frac{1}{d}}\r\|_{X}^{-1}
\sum_{j=1}^m\frac{\lambda_j|B_j|}{\|\mathbf{1}_{B_j}\|_{X}}
\lf[\fint_{B_j}\lf\{f(x)-\inf_{\wz x\in B_j}g(f)(\wz x)\r\}^q\,dx\r]^{\frac{1}{q}}\\
&\quad\ls\|f\|_{\mathcal{L}_{X,q,0,d}(\rn)},
\end{align*}
which, combined with the definition of
$\|\cdot\|_{\mathcal{L}^{\mathrm{low}}_{X,q,0,d}(\rn)}$,
further implies that \eqref{thm-b-1} holds true.
This finishes the proof of Theorem \ref{thm-B-Banach}.
\end{proof}

\begin{remark}\label{g-bounded-remark-B}
\begin{enumerate}
\item[\rm (i)]
Let $p\in(\frac{n}{n+1},\fz)$, $X:=L^p(\rn)$, $\dz=1$, and $q\in(1,\fz)$. In this case,
$\mathcal{L}_{X,q,0}(\rn)=\mathcal{C}_{\frac{1}{p}-1,q,0}(\rn)$ and hence
the corresponding result on $\mathcal{L}_{X,q,0}(\rn)$ of
Theorem \ref{thm-B-Banach} coincides with \cite[Theorem 1]{Sun04}.
Moreover, when $p=1$, we have $\mathcal{L}_{X,q,0}(\rn)=\mathrm{BMO\,}(\rn)$
and hence the corresponding result on $\mathcal{L}_{X,q,0}(\rn)$ of Theorem \ref{thm-B-Banach}
in this case coincides with \cite[Corollary 1.1]{MY}.

\item[\rm (ii)]
Let $s\in\nn$. In this case, it is still unclear
whether or not the Littlewood--Paley $g$-function
is bounded on either $\mathcal{L}_{X,q,s}(\rn)$ or $\mathcal{L}_{X,q,s,d}(\rn)$.
\end{enumerate}
\end{remark}

\section{Lusin-Area Functions on Both
$\mathcal{L}_{X,q,0}(\rn)$ and $\mathcal{L}_{X,q,0,d}(\rn)$\label{Lusin}}

In this section, we study the boundedness of the
Lusin-area function on both $\mathcal{L}_{X,q,0}(\rn)$ and $\mathcal{L}_{X,q,0,d}(\rn)$.
To be precise, we prove the following theorem.

\begin{theorem}\label{thm-B-Banach-S}
Let $\dz\in(0,1]$ and $S$ be the Lusin-area function in \eqref{S}.
Let both $X$ and $p_-\in(\frac{n}{n+\dz},\fz)$ satisfy Assumption \ref{assump1},
$q\in(1,\fz)$, and $d\in(\frac{n}{n+\dz},\fz)$.
Then, for any $f\in \mathcal{L}_{X,q,0,d}(\rn)$
[resp., $f\in \mathcal{L}_{X,q,0}(\rn)$],
$S(f)$ is either infinite everywhere or finite almost everywhere and, in the latter case,
there exists a positive constant $C$, independent of $f$, such that
\begin{align*}
\|S(f)\|_{\mathcal{L}^{\mathrm{low}}_{X,q,0,d}(\rn)}
\leq C\|f\|_{\mathcal{L}_{X,q,0,d}(\rn)}
\end{align*}
$$
\lf[\mathrm{resp}.,\ \|S(f)\|_{\mathcal{L}^{\mathrm{low}}_{X,q,0}(\rn)}
\leq C\|f\|_{\mathcal{L}_{X,q,0}(\rn)}\r].
$$
\end{theorem}

To show Theorem \ref{thm-B-Banach-S}, we establish two technique lemmas first.
We begin with the boundedness of the Lusin-area function on Lebesgue spaces
(see, for instance, \cite[p.\,46]{EMS2}).

\begin{lemma}\label{S-Lp}
Let $q\in(1,\infty)$ and $S$ be the Lusin-area function in \eqref{S}.
Then $S$ is bounded on $L^q(\rn)$, namely,
there exists a positive constant $C$ such that,
for any $f\in L^q(\rn)$,
\begin{align*}
\|S(f)\|_{L^q(\rn)}\leq C\|f\|_{L^q(\rn)}.
\end{align*}
\end{lemma}

We also need the following technical lemma.

\begin{lemma}\label{lem-s-fz}
Let $q\in(1,\fz)$, $\dz\in(0,1]$,
$S$ be the Lusin-area function in \eqref{g-4},
and $f\in L^1_{\mathrm{loc}}(\rn)$.
If $S(f)(x_0)<\fz$ for one $x_0\in\rn$, then
there exists a positive constant $C$, independent of both $f$ and $x_0$, such that,
for any ball $B\in\mathbb{B}(\rn)$ containing $x_0$,
\begin{align}\label{lem-s-01}
\lf[\fint_{B}\lf\{S(f)(x)-\inf_{\widetilde{x}\in B}S(f)(\widetilde{x})
\r\}^q\,dx\r]^{\frac{1}{q}}
\leq C\sum_{k=1}^{\infty}\frac{1}{2^{k\delta}}
\lf[\fint_{2^{k}B}\lf|f(z)-f_{2^{k}B}\r|^q\,dz\r]^{\frac{1}{q}}.
\end{align}
\end{lemma}

\begin{proof}
Let $q$, $\dz$, and $S$ be the same as in the present lemma.
Let $f\in L^1_{\mathrm{loc}}(\rn)$ and assume that there exists one $x_0\in\rn$ such that
$S(f)(x_0)<\fz$. Also, let $B:=B(y,r)$ be any given ball of $\rn$
with $y\in\rn$ and $r\in(0,\fz)$ such that $x_0\in B$. Moreover,
for any $x\in \rn$, let
\begin{align*}
S_{r}(f)(x):=\lf[\int_0^{4r}\int_{|\xi-x|<t}\lf|f\ast
\varphi_t(\xi)\r|^2\,\frac{d\xi\,dt}{t^{n+1}}\r]^{\frac{1}{2}}
\end{align*}
and
\begin{align*}
S_{\infty}(f)(x):=\lf[\int_{4r}^{\infty}\int_{|\xi-x|<t}
\lf|f\ast \varphi_t(\xi)\r|^2\,\frac{d\xi\,dt}{t^{n+1}}\r]^{\frac{1}{2}}.
\end{align*}
Then, using \eqref{g-1} and $S(f)(x_0)<\fz$
[and hence $S_{\fz}(f)(x_0)<\fz$], we conclude that, for any $x\in B$,
\begin{align*}
0&\leq S(f)(x)-\inf_{\widetilde{x}\in B}S(f)(\widetilde{x})\\	
&\leq S_r(f)(x)+S_\fz(f)(x)-\inf_{\widetilde{x}\in B}S_{\fz}(f)(\widetilde{x})\\
&\leq S_r\lf([f-f_{8B}]\mathbf{1}_{8B}\r)(x)
+S_r\lf([f-f_{8B}]\mathbf{1}_{\rn\setminus 8B}\r)(x)\\
&\quad+\sup_{\{\wz x\in B:\ S_{\fz}(f)(\wz x)<\fz\}}\lf|S_{\fz}(f)(x)-S_{\fz}(f)(\wz x)\r|
\end{align*}
and hence
\begin{align}\label{S-00}
&\lf[\fint_{B}\lf\{S(f)(x)-\inf_{\widetilde{x}\in B}S(f)(\widetilde{x})
\r\}^q\,dx\r]^{\frac{1}{q}}\\
&\quad\leq\lf[\fint_{B}\lf|S_{r}\lf([f-f_{8B}]
\mathbf{1}_{8B}\r)(x)\r|^q\,dx\r]^{\frac{1}{q}}\noz\\
&\qquad+\lf[\fint_{B}\lf|S_{r}\lf([f-f_{8B}]
\mathbf{1}_{\rn\setminus8B}\r)(x)\r|^q\,dx\r]^{\frac{1}{q}}\noz\\
&\qquad+\lf\{\fint_{B}\lf[\sup_{\{\wz x\in B:\ S_{\fz}(f)(\wz x)<\fz\}}
\lf|S_{\fz}(f)(x)-S_{\fz}(f)(\wz x)\r|\r]^q\,dx
\r\}^{\frac{1}{q}}\noz\\
&\quad=:\mathrm{S}_1+\mathrm{S}_2+\mathrm{S}_3.\noz
\end{align}

We first estimate $\mathrm{S}_1$.
Indeed, from Lemma \ref{S-Lp}, we deduce that
\begin{align}\label{S-01}
\mathrm{S}_1
&\leq\lf[\frac{1}{|B|}\int_{\rn}
\lf|S\lf([f-f_{8B}]
\mathbf{1}_{8B}\r)(x)\r|^q\,dx\r]^{\frac{1}{q}}
\lesssim\lf[\fint_{8B}
\lf|f(x)-f_{8B}\r|^q\,dx\r]^{\frac{1}{q}}.
\end{align}
This is a desired estimate of $\mathrm{S}_1$.

Next, we consider $\mathrm{S}_2$. Indeed, noticing that,
for any $t\in (0,4r)$, $\xi\in B(x,t)$,
and $z\in \rn\setminus 8B$,
\begin{align*}
t+|z-\xi|&\geq t+|z-y|-|y-x|-|x-\xi|\\
&\geq  t+\frac{|z-y|}{2}+4r-|y-x|-|x-\xi|
\gtrsim |z-y|,
\end{align*}
by this, \eqref{g-2}, and Lemma \ref{I-JN}, we conclude that,
for any $x\in B$,
\begin{align*}
&S_{r}\lf([f-f_{8B}]\mathbf{1}_{\rn\setminus8B}\r)(x)\\
&\quad\leq\lf\{\int_0^{4r}\int_{|\xi-x|<t}
\lf[\int_{\rn\setminus8B}\lf|f(z)-f_{8B}\r|
\frac{1}{t^n}\varphi\lf(\frac{\xi-z}{t}\r)\,dz\r]^2
\,\frac{d\xi\,dt}{t^{n+1}}\r\}^{\frac{1}{2}}\\
&\quad\lesssim\lf\{\int_{0}^{4r}
\int_{|\xi-x|<t}\lf[\int_{\rn\setminus8B}
\frac{|f(z)-f_{8B}|}{t^n}\frac{t^{n+\dz}}
{(t+|\xi-z|)^{n+\dz}}\,dz\r]^2
\,\frac{d\xi\,dt}{t^{n+1}}\r\}^{\frac{1}{2}}\\
&\quad\lesssim\lf\{\int_{0}^{4r}\int_{|\xi-x|<t}
\lf[\int_{\rn\setminus8B}
\frac{|f(z)-f_{8B}|}{|y-z|^{n+\dz}}\,dz\r]^2
\,\frac{d\xi\,dt}{t^{n+1-2\dz}}\r\}^{\frac{1}{2}}\\
&\quad\sim\lf(\int_{0}^{4r}\int_{|\xi-x|<t}
\,\frac{d\xi\,dt}{t^{n+1-2\dz}}\r)^{\frac{1}{2}}
\int_{\rn\setminus8B}
\frac{|f(z)-f_{8B}|}{|y-z|^{n+\dz}}\,dz\\
&\quad\sim\int_{\rn\setminus8B}
\frac{r^\dz|f(z)-f_{8B}|}{|y-z|^{n+\dz}}\,dz
\lesssim\sum_{k=3}^{\infty}\frac{1}{2^{k\dz}}
\fint_{2^{k+1}B}\lf|f(z)-f_{2^{k+1}B}\r|\,dz
\end{align*}
and hence
\begin{align}\label{G2}
\mathrm{S}_2&\lesssim\sum_{k=3}^{\infty}\frac{1}{2^{k\dz}}
\fint_{2^{k+1}B}\lf|f(z)-f_{2^{k+1}B}\r|\,dz.
\end{align}
This is a desired estimate of $\mathrm{S}_2$.

Now, we estimate $\mathrm{S}_3$. To this end,
we first estimate $|S_{\infty}(f)(x)-S_{\infty}(f)(\wz x)|$ for any $x\in B$
and any $\wz x\in B$ such that $S_{\infty}(f)(\wz x)<\fz$.
Indeed, from the fact that $S(f)\geq 0$ everywhere on $\rn$,
we deduce that, for any $x\in B$ and any $\wz x\in B$
such that $S_{\infty}(f)(\wz x)<\fz$,
\begin{align}\label{G3-0}
&\lf|S_{\infty}(f)(x)-S_{\infty}(f)(\wz x)\r|\\
&\quad=\lf|\lf[\int_{4r}^{\infty}\int_{|\xi|<t}
\lf|f\ast \varphi_t(x+\xi)\r|^2\,\frac{d\xi\,dt}{t^{n+1}}\r]^{\frac{1}{2}}
-\lf[\int_{4r}^{\infty}\int_{|\xi|<t}
\lf|f\ast \varphi_t(\wz x+\xi)\r|^2
\,\frac{d\xi\,dt}{t^{n+1}}\r]^{\frac{1}{2}}\r|\noz\\
&\quad\leq\lf[\int_{4r}^{\infty}\int_{|\xi|<t}
\lf|f\ast\varphi_t(x+\xi)-f\ast\varphi_t(\wz x+\xi)\r|^2
\,\frac{d\xi\,dt}{t^{n+1}}\r]^{\frac{1}{2}}.\noz
\end{align}
Moreover, noticing that, for any
$t\in(4r,\fz)$, $\xi\in B(\mathbf{0},t)$, $x,\wz x\in B$,
$z\in \rn\setminus 8B$, and $\theta\in(0,1)$,
$$
t+|x+\xi-z+\theta(x-\wz x)|
\sim t+|y-z|,
$$
by this, \eqref{g-1}, the mean value theorem,
\eqref{g-2}, \eqref{g-3'}, and
the fact that $|\varphi_t(z)|\lesssim t^{-n}$ for any $z\in\rn$
and $t\in(0,\infty)$, we conclude that, for any
$t\in(4r,\fz)$, $\xi\in B(\mathbf{0},t)$, $x\in B$, and $\wz x\in B$
such that $S_{\infty}(f)(\wz x)<\fz$,
there exists a $\theta \in(0,1)$ such that
\begin{align*}
&\lf|f\ast\varphi_t(x+\xi)-f\ast\varphi_t(\wz x+\xi)\r|\\
&\quad=\lf|(f-f_{8B})\ast\varphi_t(x+\xi)
-(f-f_{8B})\ast\varphi_t(\wz x+\xi)\r|\\
&\quad\leq\lf|\lf([f-f_{8B}]
\mathbf{1}_{\rn\setminus8B}\r)\ast\varphi_t(x+\xi)
-\lf([f-f_{8B}]\mathbf{1}_{\rn\setminus8B}\r)
\ast\varphi_t(x_0+\xi)\r|\\
&\quad\quad+\lf|\lf([f-f_{8B}]\mathbf{1}_{8B}\r)\ast\varphi_t(x+\xi)\r|
+\lf|\lf([f-f_{8B}]\mathbf{1}_{8B}\r)
\ast\varphi_t(\wz x+\xi)\r|\\
&\quad\leq\int_{\rn\setminus8B}\lf|f(z)-f_{8B}\r|
\lf|\varphi\lf(\frac{x+\xi-z}{t}\r)
-\varphi\lf(\frac{\wz x+\xi-z}{t}\r)\r|\,\frac{dz}{t^n}\\
&\quad\quad+\int_{8B}\lf|f(z)-f_{8B}\r|
\lf|\varphi\lf(\frac{x+\xi-z}{t}\r)\r|\,\frac{dz}{t^n}\\
&\quad\quad+\int_{8B}\lf|f(z)-f_{8B}\r|
\lf|\varphi\lf(\frac{\wz x+\xi-z}{t}\r)\r|\,\frac{dz}{t^n}\\
&\quad\lesssim\int_{\rn\setminus8B}
\frac{t^\dz|f(z)-f_{8B}||x-\wz x|}
{(t+|x+\xi-z+\theta(x-\wz x)|)^{n+1+\dz}}\,dz
+\int_{8B}\frac{|f(z)-f_{8B}|}{t^n}\,dz\\
&\quad\sim\int_{\rn\setminus8B}\lf|f(z)-f_{8B}\r|
\frac{t^\dz|x-\wz x|}{(t+|y-z|)^{n+1+\dz}}\,dz
+\int_{8B}\frac{|f(z)-f_{8B}|}{t^n}\,dz.
\end{align*}
Using this and an argument similar to that used in the estimation of \eqref{g10},
we conclude that, for any $x\in B$ and any $\wz x\in B$
such that $S_{\infty}(f)(\wz x)<\fz$,
\begin{align*}
&\lf[\int_{4r}^{\infty}\int_{|\xi|< t}
\lf|f\ast\varphi_t(x+\xi)-f\ast\varphi_t(\wz x+\xi)
\r|^2\,\frac{dt}{t^{n+1}}\r]^{\frac{1}{2}}\\
&\quad\lesssim\lf\{\int_{4r}^{\infty}
\int_{|\xi|< t}\lf[\int_{\rn\setminus8B}
\lf|f(z)-f_{8B}\r|\frac{t^\dz|x-\wz x|}{(t+|y-z|)^{n+1+\dz}}\,dz\r]^2
\,\frac{d\xi\,dt}{t^{n+1}}\r\}^{\frac{1}{2}}\\
&\quad\quad+\lf\{\int_{4r}^{\infty}\int_{|\xi|< t}
\lf[\int_{8B}\frac{|f(z)-f_{8B}|}{t^n}\,dz\r]^2
\,\frac{d\xi\,dt}{t^{n+1}}\r\}^{\frac{1}{2}}\\
&\quad\sim\lf\{\int_{4r}^{\infty}\lf[\int_{\rn\setminus8B}
\lf|f(z)-f_{8B}\r|\frac{t^\dz|x-\wz x|}{(t+|y-z|)^{n+1+\dz}}\,dz\r]^2
\,\frac{dt}{t}\r\}^{\frac{1}{2}}\\
&\quad\quad+\lf\{\int_{4r}^{\infty}\lf[
\int_{8B}\frac{|f(z)-f_{8B}|}{t^n}\,dz\r]^2
\,\frac{dt}{t}\r\}^{\frac{1}{2}}\\
&\quad\lesssim\sum_{k=1}^{\infty}\frac{1}{2^{k}}
\fint_{2^{k+1}B}\lf|f(z)-f_{2^{k+1}B}\r|\,dz.
\end{align*}
By this and \eqref{G3-0}, we conclude that
\begin{align*}
\mathrm{S}_3&\lesssim\sum_{k=1}^{\infty}\frac{1}{2^{k}}
\fint_{2^{k}B}\lf|f(z)-f_{2^{k}B}\r|\,dz,
\end{align*}
which, together with \eqref{S-00}, \eqref{S-01}, \eqref{G2}, $\dz\in(0,1]$,
and the H\"older inequality,
further implies that \eqref{lem-s-01} holds true, and hence finishes
the proof of Lemma \ref{lem-s-fz}.
\end{proof}

Next, we show Theorem \ref{thm-B-Banach-S}.
\begin{proof}[Proof of Theorem \ref{thm-B-Banach-S}]
Let all the symbols be the same as in the present theorem.
Indeed, by repeating the proof of Theorem \ref{thm-B-Banach} with $g$ and Lemma \ref{lem-g-fz}
therein replaced, respectively, by $S$ and Lemma \ref{lem-s-fz}, we
complete the proof of Theorem \ref{thm-B-Banach-S}.
\end{proof}

\begin{remark}\label{S-bounded-remark-B}
\begin{enumerate}
\item[\rm (i)]
Let $p\in(\frac{n}{n+1},\fz)$,
$X:=L^p(\rn)$, $\dz=1$, and $q\in(1,\fz)$. In this case,
$\mathcal{L}_{X,q,0}(\rn)=\mathcal{C}_{\frac{1}{p}-1,q,0}(\rn)$ and hence
the corresponding result on $\mathcal{L}_{X,q,0}(\rn)$ of
Theorem \ref{thm-B-Banach-S} coincides with \cite[Theorem 2]{Sun04}.
Moreover, when $p=1$, we have $\mathcal{L}_{X,q,0}(\rn)=\mathrm{BMO\,}(\rn)$
and hence the corresponding result on $\mathcal{L}_{X,q,0}(\rn)$ of Theorem \ref{thm-B-Banach}
in this case coincides with \cite[Corollary 1.2]{MY}.

\item[\rm (ii)]
Let $s\in\nn$. In this case, it is still unclear
whether or not the Littlewood--Paley $g$-function
is bounded on either $\mathcal{L}_{X,q,s}(\rn)$ or $\mathcal{L}_{X,q,s,d}(\rn)$.
\end{enumerate}
\end{remark}

\section{Littlewood--Paley $g_\lambda^*$-Functions on Both
$\mathcal{L}_{X,q,0}(\rn)$ and $\mathcal{L}_{X,q,0,d}(\rn)$\label{g*bounded}}

The aim of this section is to estimate the Littlewood--Paley
$g_\lambda^*$-function on both $\mathcal{L}_{X,q,0}(\rn)$ and $\mathcal{L}_{X,q,0,d}(\rn)$.
Moreover, even for the Campanato space $\mathcal{C}_{\alpha,q,0}(\rn)$,
we obtain a better range of the index $\lambda$.
To be precise, we prove the following two theorems.

\begin{theorem}\label{g-l-bounded3}
Let both $X$ and $p_-$ satisfy Assumption \ref{assump1},
$q\in[1,\fz)$, $\dz\in(0,1]$,
$\lambda\in (1,\infty)$, and $g_{\lambda}^{*}$ be the same as in \eqref{lam-function}.
If either of the following statements holds true:
\begin{enumerate}
\item[\rm (i)]
$p_-\in(\frac{n}{n+\frac{\dz}{2}},\fz)$
and $\lambda\in(\max\{1,\frac{2}{q}\},\infty)$;
\item[\rm (ii)]
$p_-\in(\frac{n}{n+\dz}, \frac{n}{n+\frac{\dz}{2}}]$ and
$\lambda\in(\max\{1,\frac{2}{q}\}+\frac{2}{n},\infty)\cap(\frac{2}{p_-},\infty)$,
\end{enumerate}
then, for any $f\in \mathcal{L}_{X,q,0}(\rn)$, $g_{\lambda}^{*}(f)$ is either
infinite everywhere or finite almost everywhere and, in the latter case,
there exists a positive constant $C$, independent of $f$, such that
\begin{align*}
\lf\|g_{\lambda}^{*}(f)\r\|_{\mathcal{L}^{\mathrm{low}}_{X,q,0}(\rn)}
\leq C\|f\|_{\mathcal{L}_{X,q,0}(\rn)}.
\end{align*}
\end{theorem}

\begin{theorem}\label{g-l-bounded4}
Let both $X$ and $p_-\in(0,\fz)$ satisfy Assumption \ref{assump1},
$q\in[1,\fz)$, $d\in(0,\fz)$, $\dz\in(0,1]$,
$\lambda\in (1,\infty)$, and $g_{\lambda}^{*}$ be the same as in \eqref{lam-function}.
If either of the following statements holds true:
\begin{enumerate}
\item[\rm (i)]
$\min\{d,p_-\}\in(\frac{n}{n+\frac{\dz}{2}},\fz)$
and $\lambda\in(\max\{1,\frac{2}{q}\},\infty)$;
\item[\rm (ii)]
$\min\{d,p_-\}\in(\frac{n}{n+\dz},\frac{n}{n+\frac{\dz}{2}}]$ and
$\lambda\in(\max\{1,\frac{2}{q}\}+\frac{2}{n},\infty)
\cap(\frac{2}{\min\{d,p_-\}},\infty)$,
\end{enumerate}
then, for any $f\in \mathcal{L}_{X,q,0,d}(\rn)$, $g_{\lambda}^{*}(f)$ is either
infinite everywhere or finite almost everywhere and, in the latter case,
there exists a positive constant $C$, independent of $f$, such that
\begin{align}\label{thm-b-3'}
\lf\|g_{\lambda}^{*}(f)\r\|_{\mathcal{L}^{\mathrm{low}}_{X,q,0,d}(\rn)}
\leq C\|f\|_{\mathcal{L}_{X,q,0,d}(\rn)}.
\end{align}
\end{theorem}

To show the above two theorems, we give some technique lemmas first.
The following lemma on the boundedness of the Littlewood--Paley
$g_\lambda^*$-function on Lebesgue spaces
is well known (see, for instance, \cite[pp.\,316-318]{EMS2}).

\begin{lemma}\label{Lam-Lp}
Let $q\in(1,\infty)$, $\lambda\in(\max\{1,\frac{2}{q}\},\infty)$,
and $g_\lambda^*$ be the same as in \eqref{lam-function}.
Then $g_\lambda^*$ is bounded on $L^q(\rn)$, namely,
there exists a positive constant $C$ such that,
for any $f\in L^q(\rn)$,
\begin{align*}
\lf\|g_\lambda^*(f)\r\|_{L^q(\rn)}\leq C\|f\|_{L^q(\rn)}.
\end{align*}
\end{lemma}

The following concept of fractional integrals can be found in, for instance, \cite[p.\,117]{EMS1970}.

\begin{definition}\label{Def-I}
Let $\beta\in(0,n)$ and $p\in[1,\frac{n}{\beta})$.
The \emph{fractional integral operator} $I_\beta$ is defined by setting,
for any $f\in L^p(\rn)$ and $x\in \rn$,
$$I_\beta(f)(x):=\int_{\rn}\frac{f(y)}{|x-y|^{n-\beta}}\,dy.$$
\end{definition}

The following lemma is well known (see, for instance, \cite[p.\,119]{EMS1970}).
\begin{lemma}\label{fractional}
Let $\beta\in(0,n)$, $a\in(1,\frac{n}{\beta})$,
and $\frac{1}{\widetilde{a}}:=\frac{1}{a}-\frac{\beta}{n}$.
Then $I_\beta$ is bounded from $L^a(\rn)$ into $L^{\widetilde{a}}(\rn)$,
namely, there exists a positive constant $C$ such that,
for any $f\in L^a(\rn)$,
$$\lf\|I_\beta(f)\r\|_{L^{\widetilde{a}}(\rn)}\leq C\|f\|_{L^a(\rn)}.$$
\end{lemma}

The following lemma plays an important role in the proofs
of both Theorems \ref{g-l-bounded3} and \ref{g-l-bounded4}.

\begin{lemma}\label{g-lam-01}
Let $q\in(1,\fz)$, $\dz\in(0,1]$, $\lambda\in(1,\fz)$,
$g_\lambda^*$ be the Littlewood--Paley $g_{\lambda}^*$-function in \eqref{g-4},
and $f\in L^1_{\mathrm{loc}}(\rn)$. Then the following two statements hold true:
\begin{enumerate}
\item[\rm (i)]
if $\lambda\in(\max\{1,\frac{2}{q}\},\infty)$ and
$g_{\lambda}^*(f)(x_0)<\fz$ for one $x_0\in\rn$,
then there exists a positive constant $C$, independent of both $f$ and $x_0$, such that,
for any ball $B\in \mathbb{B}(\rn)$ containing $x_0$,
\begin{align*}
&\lf[\fint_{B}\lf\{g_\lambda^*(f)(x)-\inf_{\widetilde{x}\in B}
g_\lambda^*(f)(\widetilde{x})\r\}^q\,dx\r]^{\frac{1}{q}}
\leq C\sum_{k=1}^{\infty}\frac{1}{2^{\frac{k\delta}{2}}}
\lf[\fint_{2^{k}B}\lf|f(z)-f_{2^{k}B}\r|^q\,dz\r]^{\frac{1}{q}};
\end{align*}		

\item[\rm (ii)]
if $\lambda\in(\max\{1,\frac{2}{q}\}+\frac{2}{n},\infty)\cap (2,\fz)$ and
$g_{\lambda}^*(f)(x_0)<\fz$ for one $x_0\in\rn$,
then there exists a positive constant $C$, independent of both $f$ and $x_0$, such that,
for any ball $B\in \mathbb{B}(\rn)$ containing $x_0$,
\begin{align*}
&\lf[\fint_{B}\lf\{g_\lambda^*(f)(x)-\inf_{\widetilde{x}\in B}
g_\lambda^*(f)(\widetilde{x})\r\}^q\,dx\r]^{\frac{1}{q}}\\
&\quad\leq C\sum_{k=1}^{\infty}\lf[\frac{1}{2^{k\dz}}
+\frac{1}{2^{\frac{kn(\lambda-2)}{2}}}\r]
\lf[\fint_{2^{k}B}\lf|f(z)-f_{2^{k}B}\r|^q\,dz\r]^{\frac{1}{q}}.
\end{align*}
\end{enumerate}		
\end{lemma}

\begin{proof}
Let $q$, $\dz$, $\lambda$, $g_\lambda^*$, and $f$ be the same as in the present lemma.
We first show (i). We do this via borrowing some ideas from the proof of \cite[Theorem 1.3]{MY}.
Assume that there exists
one $x_0\in\rn$ such that $g_{\lambda}^*(f)(x_0)<\fz$. Let $B:=B(y,r)$ with $y\in\rn$
and $r\in(0,\fz)$ be such that $x_0\in B$ and, for any $k\in\zz_+$, let
\begin{align}\label{j(k,y)}
J(y,k):=\lf\{(\xi,t)\in\rr_+^{n+1}
:\ |\xi-y|<2^{k+2}r\text{ and }0<t<2^{k+2}r\r\}.
\end{align}
Moreover, for any $x,y\in\rn$, let
\begin{align}\label{g*y0}
g_{\lambda,y,0}^{*}(f)(x)
:=\lf[\iint_{J(y,0)}\lf(\frac{t}{t+|x-\xi|}\r)^{\lambda n}
\lf|f\ast \varphi_t(\xi)\r|^2\,\frac{d\xi\,dt}{t^{n+1}}\r]^{\frac{1}{2}}
\end{align}
and
\begin{align*}
g_{\lambda,y,\infty}^{*}(f)(x):=\lf[\iint_{\rr_+^{n+1}
\setminus J(y,0)}\lf(\frac{t}{t+|x-\xi|}\r)^{\lambda n}
\lf|f\ast \varphi_t(\xi)\r|^2\,\frac{d\xi\,dt}{t^{n+1}}\r]^{\frac{1}{2}}.
\end{align*}
Using \eqref{g-1} and $g_{\lambda}^*(f)(x_0)<\fz$
[and hence $g_{\lambda,y,\fz}^*(f)(x_0)<\fz$], we conclude that, for any $x\in B$,
\begin{align*}
0&\leq g_\lambda^*(f)(x)-\inf_{\widetilde{x}\in B}
g_\lambda^*(f)(\widetilde{x})\\
&\leq g_{\lambda,y,0}^{*}(f)(x)+
g_{\lambda,y,\infty}^{*}(f)(x)-
\inf_{\wz x\in B}g_{\lambda,y,\fz}^*(f)(\wz x)\\
&\leq g_{\lambda,y,0}^{*}(f-f_{B})(x)
+\sup_{\{\wz x\in B:\ g_{\lambda,y,\fz}^*(\wz x)<\fz\}}
\lf|g_{\lambda,y,\infty}^{*}(f)(x)
-g_{\lambda,y,\infty}^{*}(f)(\wz x)\r|\\
&\leq g_{\lambda,y,0}^{*}\lf([f-f_{B}]\mathbf{1}_{B}\r)(x)
+g_{\lambda,y,0}^{*}\lf([f-f_{B}]\mathbf{1}_{\rn\setminus B}\r)(x)\\
&\quad+\sup_{\{\wz x\in B:\ g_{\lambda,y,\fz}^*(\wz x)<\fz\}}
\lf|g_{\lambda,y,\infty}^{*}(f)(x)
-g_{\lambda,y,\infty}^{*}(f)(\wz x)\r|
\end{align*}
and hence
\begin{align}\label{g-lam-00}
&\lf[\fint_{B}\lf\{g_\lambda^*(f)(x)
-\inf_{\widetilde{x}\in B}g_\lambda^*(f)(\widetilde{x})
\r\}^q\,dx\r]^{\frac{1}{q}}\\
&\quad\leq\lf[\fint_{B}\lf|g_{\lambda,y,0}^{*}\lf([f-f_{B}]
\mathbf{1}_{B}\r)(x)\r|^q\,dx\r]^{\frac{1}{q}}\noz\\
&\qquad+\lf[\fint_{B}\lf|g_{\lambda,y,0}^{*}\lf([f-f_{B}]
\mathbf{1}_{\rn\setminus B}\r)(x)\r|^q\,dx\r]^{\frac{1}{q}}\noz\\
&\qquad+\lf\{\fint_{B}\lf[\sup_{\{\wz x\in B:\ g_{\lambda,y,\fz}^*(f)(\wz x)<\fz\}}
\lf|g_{\lambda,y,\infty}^{*}(f)(x)
-g_{\lambda,y,\infty}^{*}(f)(x_0)\r|\r]^q\,dx
\r\}^{\frac{1}{q}}\noz\\
&\quad=:\mathrm{\Lambda}_1+\mathrm{\Lambda}_2+\mathrm{\Lambda}_3.\noz
\end{align}

We first estimate $\mathrm{\Lambda}_1$.
Indeed, by $\lambda\in (\max\{1,\frac{2}{q}\},\infty)$
and Lemma \ref{Lam-Lp}, we find that
\begin{align}\label{g-lam-000}
\mathrm{\Lambda}_1
&\leq\lf[\frac{1}{|B|}
\int_{\rn}\lf|g_{\lambda}^{*}\lf([f-f_{B}]
\mathbf{1}_{8B}\r)(x)\r|^q\,dx\r]^{\frac{1}{q}}\\
&\ls\lf[\fint_{8B}\lf|f(x)-f_{B}\r|^q\,dx\r]^{\frac{1}{q}}
\ls\lf[\fint_{8B}\lf|f(x)-f_{8B}\r|^q\,dx\r]^{\frac{1}{q}}.\noz
\end{align}
This is a desired estimate of $\mathrm{\Lambda}_1$.

Now, we estimate $\mathrm{\Lambda}_2$. Indeed, it is easy to prove that,
for any $(\xi,t)\in J(y,0)$, $x\in B$, and $z\in \rn\setminus8B$,
$|\xi-z|\sim|z-y|$ and $4B(y,r)\subset 8B(x,r)$.
From this, \eqref{g-2}, and Lemma \ref{I-JN},
we deduce that, for any $x\in B$,
\begin{align*}
&g_{\lambda,y,0}^*\lf([f-f_{B}]\mathbf{1}_{\rn\setminus8B}\r)(x)\\
&\quad=\lf\{\int_{0}^{4r}\int_{|\xi-y|<4r}
\lf(\frac{t}{t+|x-\xi|}\r)^{\lambda n}
\lf[\int_{\rn\setminus8B}
\frac{|f(z)-f_{B}|}{t^n}\varphi\lf(\frac{\xi-z}{t}\r)\,dz\r]^2
\,\frac{d\xi\,dt}{t^{n+1}}\r\}^{\frac{1}{2}}\\
&\quad\lesssim\lf\{\int_{0}^{4r}\int_{|\xi-y|<4r}
\lf(\frac{t}{t+|x-\xi|}\r)^{\lambda n}
\lf[\int_{\rn\setminus8B}
\frac{|f(z)-f_{B}|}{(t+|\xi-z|)^{n+\delta}}\,dz\r]^2
\,\frac{d\xi\,dt}{t^{n+1-2\delta}}\r\}^{\frac{1}{2}}\\
&\quad\ls\lf\{\int_{0}^{4r}
\int_{|\xi-y|<4r}\lf(\frac{t}{t+|x-\xi|}\r)^{\lambda n}
\lf[\int_{\rn\setminus 8B}
\frac{|f(z)-f_{B}|}{|z-y|^{n+\delta}}\,dz\r]^2
\,\frac{d\xi\,dt}{t^{n+1-2\delta}}\r\}^{\frac{1}{2}}\\
&\quad\lesssim\lf[\sum_{k=3}^{\infty}\frac{1}{2^{k\delta}}
\fint_{2^{k+1}B}\lf|f(z)-f_{B}\r|\,dz\r]
\lf[\int_{0}^{4r}\int_{|\xi-x|<8r}\lf(\frac{t}{t+|x-\xi|}\r)^{\lambda n}
\frac{1}{r^{2\delta}}\,\frac{d\xi\,dt}{t^{n+1-2\delta}}\r]^{\frac{1}{2}},
\end{align*}
and, moreover, using the change of variable
$s:=\frac{\xi-x}{t}$ and the fact that $\lambda>1$, we have
\begin{align*}
&\lf[\int_{0}^{4r}\int_{|\xi-y|<8r}\lf(\frac{t}{t+|x-\xi|}\r)^{\lambda n}
\frac{1}{r^{2\delta}}\,\frac{d\xi\,dt}{t^{n+1-2\delta}}\r]^{\frac{1}{2}}\\
&\quad=\lf[\int_{0}^{4r}\int_{|s-\frac{y-x}{t}|<\frac{8r}{t}}
\lf(\frac{t}{t+t|s|}\r)^{\lambda n}
\frac{t^{2\delta-1}}{r^{2\delta}}\,ds\,dt\r]^{\frac{1}{2}}\\
&\quad\leq\lf[\int_{0}^{4r}\int_{\rn}
\frac{t^{2\delta-1}}{(1+|s|)^{\lambda n}r^{2\delta}}\,ds\,dt\r]^{\frac{1}{2}}
\quad\sim\lf(\int_0^{4r}
\frac{t^{2\delta-1}}{r^{2\delta}}\,dt\r)^{\frac{1}{2}}
\sim 1.
\end{align*}
By this and Lemma \ref{sum-g} with $\theta:=2^{-\dz}$,
we conclude that, for any $x\in B$,
\begin{align*}
g_{\lambda,y,0}^*\lf(\lf[f-f_{B}\r]\mathbf{1}_{8B}\r)(x)
\lesssim\sum_{k=3}^{\infty}\frac{1}{2^{k\delta}}\fint_{2^{k+1}B}
\lf|f(z)-f_{2^{k+1}B}\r|\,dz
\end{align*}
and hence
\begin{align}\label{g-lam-001}
\mathrm{\Lambda}_2
\lesssim\sum_{k=3}^{\infty}\frac{1}{2^{k\delta}}\fint_{2^{k+1}B}
\lf|f(z)-f_{2^{k+1}B}\r|\,dz.
\end{align}
This is a desired estimate of $\mathrm{\Lambda}_2$.

Next, we estimate $\mathrm{\Lambda}_3$. To this end,
we first estimate $|g_{\lambda,y,\infty}^*(f)(x)
-g_{\lambda,y,\infty}^*(f)(\wz x)|$ for any $x\in B(y,r)$ and $\wz x\in B$
such that $g_{\lambda,y,\fz}^*(f)(\wz x)<\fz$.
Indeed, it is easy to show that, for any $x,\wz x\in B(y,r)$ and
$(\xi,t)\in \rn\setminus J(y,0)$,
$$
t+|x-\xi|\sim t+|y-\xi|\sim t+|\wz x-\xi|.
$$
From this, the definition of $g_{\lambda,y,\infty}^*(f)$,
the mean value theorem, \eqref{g-1}, and \eqref{g-2}, we deduce that,
for any $x\in B$ and $\wz x\in B$
such that $g_{\lambda,y,\fz}^*(f)(\wz x)<\fz$, there exists a $\theta\in(0,1)$ such that
\begin{align}\label{g-lam-002}
&\lf|g_{\lambda,y,\infty}^*(f)(x)
-g_{\lambda,y,\infty}^*(f)(\wz x)\r|^2\\
&\quad\leq\iint_{\rr_+^{n+1}
\setminus J(y,0)}\lf|\lf(\frac{t}{t+|x-\xi|}\r)^{\lambda n}
-\lf(\frac{t}{t+|\wz x-\xi|}\r)^{\lambda n}\r|\lf|f\ast
\varphi_t(\xi)\r|^2\,\frac{d\xi\,dt}{t^{n+1}}\noz\\
&\quad=\iint_{\rr_+^{n+1}\setminus J(y,0)}
\frac{||x-\xi|-|\wz x-\xi||}
{(t+\theta|x-\xi|+(1-\theta)|\wz x-\xi|)^{\lambda n+1}}
\lf|f\ast \varphi_t(\xi)\r|^2\,\frac{d\xi\,dt}{t^{n+1-\lambda n}}\noz\\
&\quad\ls\iint_{\rr_+^{n+1}\setminus J(y,0)}
\frac{r}{(t+|y-\xi|)^{\lambda n+1}}
\lf[\int_{\rn}\lf|f(z)-f_{B}\r|
\frac{1}{t^n}\varphi\lf(\frac{\xi-z}{t}\r)\,dz\r]^2
\,\frac{d\xi\,dt}{t^{n+1-\lambda n}}\noz\\
&\quad\ls
r\sum_{k=1}^{\infty}\iint_{J(y,k)\setminus J(y,k-1)}
\lf(\frac{t}{2^{k+2}r}\r)^{\lambda n+1}
\lf[\int_{2^{k+3}B}\frac{|f(z)-f_{B}|}{(t+|\xi-z|)^{n+\delta}}\,dz\r]^2
\,\frac{d\xi\,dt}{t^{n-2\delta+2}}\noz\\
&\quad\quad +r\sum_{k=1}^{\infty}\iint_{J(y,k)\setminus J(y,k-1)}
\lf(\frac{t}{2^{k+2}r}\r)^{\lambda n+1}
\lf[\int_{\rn\setminus2^{k+3}B}\frac{|f(z)-f_{B}|}{(t+|\xi-z|)^{n+\delta}}\,dz\r]^2
\,\frac{d\xi\,dt}{t^{n-2\delta+2}}\noz\\
&\quad=:g_{\lambda,y,\infty}^{*,1}(x)
+g_{\lambda,y,\infty}^{*,2}(x).\noz
\end{align}

Now, we estimate $g_{\lambda,y,\infty}^{*,1}(x)$ for any $x\in B$.
Notice that, for any $k\in\nn$, $y\in\rn$,
and $(\xi,t)\in J(y,k)$, $\frac{t}{2^{k+2}r}\in(0,1)$,
which implies that, for any $\lambda_1,\lambda_2\in(1,\infty)$
satisfying $\lambda_1>\lambda_2$,
$(\frac{t}{2^{k+2}r})^{\lambda_1 n+1}\leq(\frac{t}{2^{k+2}r})^{\lambda_2 n+1}$.
Thus, in the following estimation for
$g_{\lambda,y,\infty}^{*,1}(x,\widetilde{x})$,
we may always assume that
$\lambda\in(\max\{1,\frac{2}{q}\},2)$ because $q\in(1,\fz)$.
From this, the Minkowski inequality,
Lemma \ref{fractional}(ii) with $\beta:=\frac{n(\lambda-1)}{2}\in(0,n)$,
$a:=\frac{2}{\lambda}\in(1,q)\cap(1,\frac{n}{\beta})$, and $\widetilde{a}:=2$,
Lemma \ref{sum-g} with $\theta:=\frac{1}{2}$, and the H\"older inequality,
we deduce that, for any $x\in B$,
\begin{align}\label{g-2-01}
g_{\lambda,y,\infty}^{*,1}(x)
&\ls r\sum_{k=1}^{\infty}\int_{|\xi-y|<2^{k+2}r}\int_{0}^{2^{k+2}r}
\lf(\frac{t}{2^{k+2}r}\r)^{\lambda n+1}
\lf[\int_{2^{k+3}B}\frac{|f(z)-f_{B}|}{(t+|\xi-z|)^{n+\delta}}\,dz\r]^2
\,\frac{d\xi\,dt}{t^{n-2\delta+2}}\\
&\lesssim\sum_{k=1}^{\infty}\frac{r}{(2^kr)^{\lambda n+1}}
\int_{|\xi-y|<2^{k+2}r}\lf\{\int_{2^{k+3}B}
\lf|f(z)-f_{B}\r|\r.\noz\\
&\quad\times\lf.\lf[\int_{0}^{2^{k+2}r}
\frac{t^{\lambda n-n-1+2\delta}}{(t+|\xi-z|)^{2n+2\delta}}
\,dt\r]^{\frac{1}{2}}\,dz\r\}^2\,d\xi\noz\\
&\lesssim\sum_{k=1}^{\infty}\frac{r}{(2^kr)^{\lambda n+1}}
\int_{|\xi-y|<2^{k+2}r}\lf\{\int_{2^{k+3}B}
\lf|f(z)-f_{B}\r|\r.\noz\\
&\quad\times\lf.\lf[\int_{0}^{|\xi-z|}
\frac{t^{\lambda n-n-1+2\delta}}{(t+|\xi-z|)^{2n+2\delta}}\,dt
+\int_{|\xi-z|}^{\infty}\frac{t^{\lambda n-n-1+2\delta}}
{(t+|\xi-z|)^{2n+2\delta}}\,dt\r]^{\frac{1}{2}}\,dz\r\}^2\,d\xi\noz\\
&\lesssim\sum_{k=1}^{\infty}\frac{r}{(2^kr)^{\lambda n+1}}
\int_{\rn}\lf[\int_{2^{k+3}B}
\frac{|f(z)-f_{B}|}{|\xi-z|^{n-\frac{n}{2}(\lambda-1)}}\,dz\r]^2\,d\xi\noz\\
&\lesssim\sum_{k=1}^{\infty}\frac{r}{(2^kr)^{\lambda n+1}}
\lf[\int_{2^{k+3}B}\lf|f(z)-f_{B}\r|^{\frac{2}{\lambda}}\,dz\r]^\lambda\noz\\
&\ls\lf\{\sum_{k=1}^{\infty}\frac{1}{2^\frac{k}{2}}\lf[\fint_{2^{k+3}B}
\lf|f(z)-f_{2^{k+3}B}
\r|^{\frac{2}{\lambda}}\,dz\r]^{\frac{\lambda}{2}}\r\}^2\noz\\
&\ls\lf\{\sum_{k=1}^{\infty}\frac{1}{2^\frac{k}{2}}\lf[\fint_{2^{k+3}B}
\lf|f(z)-f_{2^{k+3}B}\r|^{q}\,dz\r]^{\frac{1}{q}}\r\}^2.\noz
\end{align}
This is a desired estimate of $g_{\lambda,y,\infty}^{*,1}(x)$.

Next, we estimate $g_{\lambda,y,\infty}^{*,2}(x)$ for any $x\in B$.
Indeed, notice that, for any $z\in \rn\setminus2^{k+3}B$
and $(\xi,t)\in J(y,k)\setminus J(y,k-1)$,
$$|\xi-z|\sim|y-z|.$$
Using this and the Fubini theorem, we conclude that, for any $x\in B$,
\begin{align*}
&g_{\lambda,y,\infty}^{*,2}(x)\\
&\quad\sim\sum_{k=1}^{\infty}\frac{r}{(2^{k}r)^{\lambda n+1}}
\iint_{J(y,k)\setminus J(y,k-1)}
\lf[\sum_{l=k+3}^{\infty}\int_{2^{l+1}B\setminus2^{l}B}
\frac{|f(z)-f_{B}|}{(t+|y-z|)^{n+\delta}}\,dz\r]^2
\,\frac{d\xi\,dt}{t^{n-\lambda n-2\delta+1}}\\
&\quad\ls\sum_{k=1}^{\infty}\frac{r}{(2^{k}r)^{\lambda n+1}}
\iint_{J(y,k)}\lf[\sum_{l=k+3}^{\infty}\int_{2^{l+1}B}\frac{|f(z)
-f_{B}|}{(2^lr)^{n+\delta}}\,dz\r]^2
\,\frac{d\xi\,dt}{t^{n-\lambda n-2\delta+1}}\\
&\quad\sim\sum_{k=1}^{\infty}
\frac{r}{(2^{k}r)^{\lambda n+1}}\iint_{J(y,k)}
\frac{1}{r^{2\delta}}\lf[\sum_{l=k+3}^{\infty}\frac{1}{2^{l\delta}}
\fint_{2^{{l+1}}B}\lf|f(z)-f_{B}\r|\,dz\r]^2
\,\frac{d\xi\,dt}{t^{n-\lambda n-2\delta+1}}\\
&\quad\ls\lf\{\sum_{k=1}^{\infty}
\sum_{l=k+3}^{\infty}\lf[\frac{1}{2^{l\delta}}
\fint_{2^{l+1}B}\lf|f(z)-f_{B}\r|\,dz\r]
\lf[\frac{r^{1-2\dz}}{(2^kr)^{\lambda n+1}}
\int_{0}^{2^{k+2}r}\int_{|\xi-y|<2^{k+2}r}
\,\frac{d\xi\,dt}{t^{n-\lambda n-2\delta+1}}\r]^{\frac{1}{2}}\r\}^2\\
&\quad\lesssim\lf[\sum_{l=4}^{\infty}\sum_{k=1}^{l-3}
\frac{1}{2^{l\delta}}\fint_{2^{l+1}B}
\lf|f(z)-f_{B}\r|\,dz\,2^{\frac{k(2\delta-1)}{2}}\r]^{2}
\ls\lf[\sum_{l=4}^{\infty}
\frac{1}{2^{\frac{l\dz}{2}}}\fint_{2^{l+1}B}
\lf|f(z)-f_{B}\r|\,dz\r]^{2},
\end{align*}
where, in the last step, we used the fact that
\begin{align*}
\frac{1}{2^{l\delta}}\sum_{k=1}^{l-3}2^{\frac{k(2\delta-1)}{2}}
&\ls
\begin{cases}
\displaystyle
\frac{1}{2^{\frac{l}{2}}},
& \dz\in(\frac{1}{2},1),\\
\displaystyle
\frac{l}{2^{l\dz}}, & \dz\in(0,\frac{1}{2}]
\end{cases}
\ls\frac{1}{2^{\frac{l\dz}{2}}}.
\end{align*}
From this and Lemma \ref{sum-g} with
$\theta:=2^{-\frac{\dz}{2}}$, we deduce that
\begin{align}\label{Lam-3-1}
g_{\lambda,y,\infty}^{*,2}(x)\lesssim
\lf\{\sum_{l=4}^{\infty}\frac{1}{2^{\frac{l\dz}{2}}}
\lf[\fint_{2^{l}B}\lf|f(z)-f_{2^{l}B}\r|^q\,dz\r]^{\frac{1}{q}}\r\}^2.
\end{align}
This is a desired estimate of $g_{\lambda,y,\infty}^{*,2}(x)$.

Moreover, by \eqref{g-lam-002}, \eqref{g-2-01}, and \eqref{Lam-3-1}, we find that
\begin{align*}
\mathrm{\Lambda}_3\ls\sum_{k=1}^{\infty}\frac{1}{2^\frac{k}{2}}\lf[\fint_{2^{k+3}B}
\lf|f(z)-f_{2^{k+3}B}\r|^{q}\,dz\r]^{\frac{1}{q}},
\end{align*}
which, combined with \eqref{g-lam-00}, \eqref{g-lam-000},
\eqref{g-lam-001}, and $\dz\in(0,1]$, further implies that
\begin{align*}
\lf[\fint_{B}\lf\{g_\lambda^*(f)(x)-\inf_{\widetilde{x}\in B}
g_\lambda^*(f)(\widetilde{x})\r\}^q\,dx\r]^{\frac{1}{q}}
\lesssim\sum_{k=1}^{\infty}\frac{1}{2^{\frac{k\delta}{2}}}
\lf[\fint_{2^{k+1}B}\lf|f(z)-f_{2^{k+1}B}\r|^q\,dz\r]^{\frac{1}{q}}.
\end{align*}
This finishes the proof of (i).

Now, we prove (ii). To this end, assume that there exists
one $x_0\in\rn$ such that $g_{\lambda}^*(f)(x_0)<\fz$
and a ball $B:=B(y,r)\subset \rn$ with $y\in\rn$
and $r\in(0,\fz)$ such that $x_0\in B$, and let
\begin{align*}
\widetilde{J}(0):=\lf\{(\xi,t)\in\rr_+^{n+1}:\ \xi\in\rn\text{ and }t\in(0,4r)\r\}.
\end{align*}
Moreover, for any $x\in\rn$, let
$g_{\lambda,y,0}^{*}(f)(x)$
be the same as in \eqref{g*y0},
$$g_{\lambda,0}^{*}(f)(x)
:=\lf[\iint_{\widetilde{J}(0)\setminus J(y,0)}\lf(\frac{t}{t+|x-\xi|}\r)^{\lambda n}
\lf|f\ast \varphi_t(\xi)\r|^2\,\frac{d\xi\,dt}{t^{n+1}}\r]^{\frac{1}{2}},$$
and
\begin{align}\label{def-g-lam}
g_{\lambda,\infty}^{*}(f)(x):=&\lf[\iint_{\rr_+^{n+1}
\setminus \widetilde{J}(0)}\lf(\frac{t}{t+|x-\xi|}\r)^{\lambda n}
\lf|f\ast \varphi_t(\xi)\r|^2\,\frac{d\xi\,dt}{t^{n+1}}\r]^{\frac{1}{2}}\\
=&\lf[\iint_{\rr_+^{n+1}\setminus \widetilde{J}(0)}\lf(\frac{t}{t+|\xi|}\r)^{\lambda n}
\lf|f\ast \varphi_t(x+\xi)\r|^2\,\frac{d\xi\,dt}{t^{n+1}}\r]^{\frac{1}{2}},\noz
\end{align}
where $J(y,0)$ is the same as in \eqref{j(k,y)} with $k:=0$.
Then, using \eqref{g-1}, we conclude that, for any $x\in B$,
\begin{align*}
0&\leq g_\lambda^*(f)(x)-\inf_{\widetilde{x}\in B}g_\lambda^*(f)(x)\\
&\leq g_{\lambda,y,0}^{*}(f)(x)
+g_{\lambda,0}^{*}(f)(x)+g_{\lambda,\infty}^{*}(f)(x)
-\inf_{\widetilde{x}\in B}g_{\lambda,\fz}^*(f)(x)\\
&\leq g_{\lambda,y,0}^{*}\lf(f-f_{B}\r)(x)
+g_{\lambda,0}^{*}\lf(f-f_{B}\r)(x)\\
&\quad+\sup_{\{\widetilde{x}\in B:\ g_{\lambda,\infty}^{*}(f)(\widetilde{x})<\fz\}}
\lf|g_{\lambda,\infty}^{*}(f)(x)
-g_{\lambda,\infty}^{*}(f)(\widetilde{x})\r|\\
&\leq g_{\lambda,y,0}^{*}\lf([f-f_{B}]\mathbf{1}_{B}\r)(x)
+g_{\lambda,y,0}^{*}\lf([f-f_{B}]\mathbf{1}_{\rn\setminus B}\r)(x)
+g_{\lambda,0}^{*}\lf(f-f_{B}\r)(x)\\
&\quad+\sup_{\{\widetilde{x}\in B:\ g_{\lambda,\infty}^{*}(f)(\widetilde{x})<\fz\}}
\lf|g_{\lambda,\infty}^{*}(f)(x)-g_{\lambda,\infty}^{*}(f)(\widetilde{x})\r|
\end{align*}
and hence
\begin{align}\label{g-lam-2-000}
&\lf[\fint_{B}\lf\{g_\lambda^*(f)(x)
-\inf_{\widetilde{x}\in B}g_\lambda^*(f)(\widetilde{x})
\r\}^q\,dx\r]^{\frac{1}{q}}\\
&\quad\leq\lf[\fint_{B}\lf|g_{\lambda,y,0}^{*}\lf([f-f_{B}]
\mathbf{1}_{B}\r)(x)\r|^q\,dx\r]^{\frac{1}{q}}\noz\\
&\quad\quad+\lf[\fint_{B}\lf|g_{\lambda,y,0}^{*}\lf([f-f_{B}]
\mathbf{1}_{\rn\setminus B}\r)(x)\r|^q\,dx\r]^{\frac{1}{q}}\noz\\
&\quad\quad+\lf[\fint_{B}\lf|g_{\lambda,0}^{*}
(f-f_{B})(x)\r|^q\,dx\r]^{\frac{1}{q}}\noz\\
&\quad\quad+\lf\{\fint_{B}\lf[\sup_{\{\widetilde{x}\in B:
\ g_{\lambda,\infty}^{*}(f)(\widetilde{x})<\fz\}}
\lf|g_{\lambda,\infty}^{*}(f)(x)
-g_{\lambda,\infty}^{*}(f)(\widetilde{x})\r|\r]^q\,dx\r\}^{\frac{1}{q}}\noz\\
&\quad=:\Lambda_1+\Lambda_2+\widetilde{\Lambda}_3+\widetilde{\Lambda}_4,\noz
\end{align}
where $\Lambda_1$ and $\Lambda_2$ are the same as in \eqref{g-lam-00}.
Moreover, by \eqref{g-lam-000} and \eqref{g-lam-001}, we find that
\begin{align}\label{g-lam-2-01}
\Lambda_1+\Lambda_2
\lesssim\sum_{k=1}^{\infty}\frac{1}{2^{k\delta}}\fint_{2^{k+1}B}
\lf|f(z)-f_{2^{k+1}B}\r|\,dz.
\end{align}

Next, we estimate $\widetilde{\Lambda}_3$. To this end, we first estimate
$g_{\lambda,0}^{*}(f-f_{B})(x)$ for any $x\in B$.
Indeed, by \eqref{g-2},
we find that, for any $x\in B$,
\begin{align}\label{Pi-2}
&g_{\lambda,0}^{*}\lf(f-f_{B}\r)(x)\\
&\quad=\lf\{\iint_{\widetilde{J}(0)\setminus J(y,0)}
\lf(\frac{t}{t+|x-\xi|}\r)^{\lambda n}
\lf|\lf(f-f_{B}\r)
\ast \varphi_t(\xi)\r|^2\,\frac{d\xi\,dt}{t^{n+1}}\r\}^{\frac{1}{2}}\noz\\
&\quad=\lf\{\int_{0}^{4r}\int_{\rn\setminus 4B}
\lf(\frac{t}{t+|x-\xi|}\r)^{\lambda n}
\lf|\int_{\rn}[f(z)-f_{B}] \varphi_t(\xi-z)\,dz\r|^2
\,\frac{d\xi\,dt}{t^{n+1}}\r\}^{\frac{1}{2}}\noz\\
&\quad\le\lf\{\sum_{k=1}^{\infty}
\int_{0}^{4r}\int_{2^{k+2}B\setminus2^{k+1}B}
\lf(\frac{t}{2^{k}r}\r)^{\lambda n}
\lf[\int_{2^{k+3}B}
\frac{|f(z)-f_{B}|}{(t+|\xi-z|)^{n+\dz}}\,dz\r]^2
\,\frac{d\xi\,dt}{t^{n+1-2\dz}}\r\}^{\frac{1}{2}}\noz\\
&\qquad+\lf\{\sum_{k=1}^{\infty}
\int_{0}^{4r}\int_{2^{k+2}B\setminus2^{k+1}B}
\lf(\frac{t}{2^{k}r}\r)^{\lambda n}
\lf[\int_{\rn\setminus2^{k+3}B}
\frac{|f(z)-f_{B}|}{(t+|\xi-z|)^{n+\dz}}\,dz\r]^2
\,\frac{d\xi\,dt}{t^{n+1-2\dz}}\r\}^{\frac{1}{2}}\noz\\
&\quad=:\widetilde{\Lambda}_3^{(1)}+\widetilde{\Lambda}_3^{(2)}.\noz
\end{align}

Now, we estimate $\widetilde{\Lambda}_3^{(1)}$.
Indeed, notice that, for any $t\in(0,\fz)$,
\begin{align*}
&\int_{\rn}\frac{1}{(t+|\xi-z|)^{2n+2\dz}}\,d\xi\\
&\quad=\int_{\rn}\frac{1}{(t+|\xi|)^{2n+2\dz}}\,d\xi
=\int_{0}^{t}\frac{u^{n-1}}{(t+u)^{2n+2\dz}}\,du
+\int_{t}^{\fz}\cdots\\
&\quad\leq \int_{0}^{t}\frac{u^{n-1}}{t^{2n+2\dz}}\,du+\int_{t}^{\fz}u^{-n-1-2\dz}\,du
=t^{-n-2\dz}.
\end{align*}
By this, the Minkowski inequality, $\lambda>2$,
and Lemma \ref{sum-g} with $\theta:=2^{-\frac{n(\lambda-2)}{2}}$,
we conclude that
\begin{align}\label{g-lam-0001}
\widetilde{\Lambda}_3^{(1)}
&\ls\lf\{\sum_{k=1}^{\infty}
\int_{0}^{4r}\int_{\rn}
\lf(\frac{t}{2^{k}r}\r)^{\lambda n}
\lf[\int_{2^{k+3}B}\frac{|f(z)
-f_{B}|}{(t+|\xi-z|)^{n+\dz}}\,dz\r]^2
\,\frac{d\xi\,dt}{t^{n+1-2\dz}}\r\}^{\frac{1}{2}}\\
&\lesssim\lf[\sum_{k=1}^{\infty}\lf(\frac{1}{2^kr}\r)^{\lambda n}
\int_{0}^{4r}t^{\lambda n-n-1+2\dz}\lf\{\int_{2^{k+3}B}
\lf[\int_{\rn}\frac{|f(z)-f_{B}|^2}
{(t+|\xi-z|)^{2n+2\dz}}\,d\xi\r]^{\frac{1}{2}}\,dz\r\}^2
\,dt\r]^{\frac{1}{2}}\noz\\
&\ls\lf\{\sum_{k=1}^{\infty}\lf(\frac{1}{2^kr}\r)^{\lambda n}
\int_{0}^{4r}t^{\lambda n-2n-1}\,dt\lf[
\int_{2^{k+3}B}\lf|f(z)-f_{B}\r|
\,dz\r]^2\,d\xi\r\}^{\frac{1}{2}}\noz\\
&\sim\lf\{\sum_{k=1}^{\infty}\frac{r^{-2n}}{2^{\lambda kn}}
\lf[\int_{2^{k+3}B}\lf|f(z)-f_{B}\r|\,dz
\r]^2\r\}^{\frac{1}{2}}
\lesssim\sum_{k=1}^{\infty}
\frac{1}{2^{\frac{kn(\lambda-2)}{2}}}\fint_{2^{k+3}B}
\lf|f(z)-f_{B}\r|\,dz\noz\\
&\lesssim\sum_{k=1}^{\infty}
\frac{1}{2^{\frac{kn(\lambda-2)}{2}}}\fint_{2^{k}B}
\lf|f(z)-f_{2^{k}B}\r|\,dz.\noz
\end{align}
This is a desired estimate of $\widetilde{\Lambda}_3^{(1)}$.

Next, we estimate $\widetilde{\Lambda}_3^{(2)}$.
Indeed, for any $k\in\nn$, $z\in\rn\setminus2^{k+3}B$,
$t\in(0,4r)$, and $\xi\in 2^{k+2}B$,
$$t+|\xi-z|\sim|y-z|.$$
From this, the Fubini theorem, and $\lambda\in(1,\fz)$, we deduce that
\begin{align*}
\widetilde{\Lambda}_3^{(2)}
&\ls\lf\{\sum_{k=1}^{\infty}\lf(\frac{1}{2^{k}r}\r)^{\lambda n}
\int_{0}^{4r}\int_{2^{k+2}B}
\lf[\sum_{l=k+3}^{\infty}\int_{2^{l+1}B\setminus2^{l}B}
\frac{|f(z)-f_{B}|}{|y-z|^{n+\dz}}\,dz\r]^2
\,\frac{d\xi\,dt}{t^{n-\lambda n+1-2\dz}}\r\}^{\frac{1}{2}}\\
&\ls\lf\{\sum_{k=1}^{\infty}\lf(\frac{1}{2^{k}r}\r)^{\lambda n}
\int_{0}^{4r}\int_{2^{k+2}B}
\lf[\sum_{l=k+3}^{\infty}\int_{2^{l+1}B}
\frac{|f(z)-f_{B}|}{(2^lr)^{n+\dz}}\,dz\r]^2
\,\frac{d\xi\,dt}{t^{n-\lambda n+1-2\dz}}\r\}^{\frac{1}{2}}\\
&\sim\lf\{\sum_{k=1}^{\infty}\lf(\frac{1}{2^{k}r}\r)^{\lambda n}
\int_{0}^{4r}\int_{2^{k+2}B}
\frac{1}{r^{2\dz}}\lf[\sum_{l=k+3}^{\infty}\frac{1}{2^{l\dz}}
\fint_{2^{{l+1}}B}\lf|f(z)-f_{B}\r|\,dz\r]^2
\,\frac{d\xi\,dt}{t^{n-\lambda n+1-2\dz}}\r\}^{\frac{1}{2}}\\
&\lesssim\sum_{k=1}^{\infty}\sum_{l=k+3}^{\infty}\lf[\frac{1}{2^{l\dz}}
\fint_{2^{l+1}B}\lf|f(z)-f_{B}\r|\,dz\r]
\lf[\lf(\frac{1}{2^{k}r}\r)^{\lambda n}\frac{1}{r^{2\dz}}\int_{0}^{4r}
\int_{|\xi-y|<2^{k+2}r}\,\frac{d\xi\,dt}{t^{n-\lambda n+1-2\dz}}
\r]^{\frac{1}{2}}\\
&\sim\sum_{l=4}^{\infty}\sum_{k=1}^{l-3}\frac{1}{2^{l\dz}}
\fint_{2^{l+1}B}\lf|f(z)-f_{B}\r|\,dz
2^{\frac{kn(1-\lambda)}{2}}
\sim\sum_{l=4}^{\infty}\frac{1}{2^{l\dz}}
\fint_{2^{l+1}B}\lf|f(z)-f_{B}\r|\,dz,
\end{align*}
which, together with Lemma \ref{sum-g} with
$\theta:=2^{-\dz}$, further implies that
\begin{align}\label{Lam-3-1'}
\widetilde{\Lambda}_3^{(2)}\lesssim
\sum_{l=1}^{\infty}\frac{1}{2^{l\dz}}
\fint_{2^{l}B}\lf|f(z)-f_{2^{l}B}\r|\,dz.
\end{align}
This is a desired estimate of $\widetilde{\Lambda}_3^{(2)}$.

Moreover, by \eqref{Pi-2}, \eqref{g-lam-0001}, and \eqref{Lam-3-1'},
we find that, for any $x\in B$,
\begin{align*}
g_{\lambda,0}^{*}\lf(f-f_{B}\r)(x)
\lesssim\sum_{k=1}^{\infty}
\lf[\frac{1}{2^{k\dz}}+\frac{1}{2^{\frac{kn(\lambda-2)}{2}}}\r]
\fint_{2^{k}B}\lf|f(z)-f_{2^{k}B}\r|\,dz
\end{align*}
and hence
\begin{align}\label{g-lam-2-03}
\widetilde{\Lambda}_3\ls\sum_{k=1}^{\infty}
\lf[\frac{1}{2^{k\dz}}+\frac{1}{2^{\frac{kn(\lambda-2)}{2}}}\r]
\fint_{2^{k}B}\lf|f(z)-f_{2^{k}B}\r|\,dz.
\end{align}
This is a desired estimate of $\widetilde{\Lambda}_3$.

Now, we estimate $\widetilde{\Lambda}_4$. To this end,
for any $k\in\nn$ and $y\in\rn$, let
$$\widetilde{J}(y,k):=J(y,k)\setminus\widetilde{J}(0),$$
where $J(y,k)$ is the same as in \eqref{j(k,y)}.
Then
\begin{align}\label{wJ-rn}
\rr_+^{n+1}\setminus\widetilde{J}(0)=\bigcup_{k\in\nn}\widetilde{J}(y,k).
\end{align}
We first estimate
$|g_{\lambda,\infty}^*(f)(x)-g_{\lambda,\infty}^*(f)(\widetilde{x})|$
for any $x\in B$ and $\wz x\in B$ such that $g_{\lambda,\fz}^*(f)(\wz x)<\fz$.
Indeed, from \eqref{def-g-lam}, we deduce that,
for any $x\in B$ and $\wz x\in B$ such that $g_{\lambda,\fz}^*(f)(\wz x)<\fz$,
\begin{align}\label{5.19x}
&\lf|g_{\lambda,\infty}^*(f)(x)-g_{\lambda,\infty}^*(f)(\widetilde{x})\r|\\
&\quad\leq\lf[\iint_{\rr_+^{n+1}\setminus \widetilde{J}(0)}
\lf(\frac{t}{t+|\xi|}\r)^{\lambda n}
\lf|f\ast\varphi_t(x+\xi)-f\ast\varphi_t(\widetilde{x}+\xi)
\r|^2\,\frac{d\xi\,dt}{t^{n+1}}\r]^{\frac{1}{2}}.\noz
\end{align}
Moreover, using \eqref{g-1}, we find that, for any $t\in (4r,\fz)$, $\xi\in\rn$,
$x\in B$, and $\wz x\in B$ such that $g_{\lambda,\fz}^*(f)(\wz x)<\fz$, and for any $k\in\nn$,
\begin{align*}
&\lf|f\ast \varphi_t(x+\xi)-f\ast\varphi_t(\widetilde{x}+\xi)\r|\\
&\quad=\lf|(f-f_{B})\ast \varphi_t(x+\xi)
-(f-f_{B})\ast\varphi_t(\widetilde{x}+\xi)\r|\\
&\quad\le\lf|\lf([f-f_{B}]
\mathbf{1}_{2^{k+3}B}\r)\ast\varphi_t(x+\xi)
-\lf([f-f_{B}]\mathbf{1}_{2^{k+3}B}\r)
\ast\varphi_t(\widetilde{x}+\xi)\r|\\
&\quad\quad+\lf|\lf([f-f_{B}]
\mathbf{1}_{\rn\setminus 2^{k+3}B}\r)\ast \varphi_t(x+\xi)
-\lf([f-f_{B}]\mathbf{1}_{\rn\setminus
2^{k+3}B}\r)\ast \varphi_t(\widetilde{x}+\xi)\r|\\
&\quad\le\int_{2^{k+3}B}\lf|f(z)-f_{B}\r|
\lf|\varphi\lf(\frac{x+\xi-z}{t}\r)-\varphi
\lf(\frac{\widetilde{x}+\xi-z}{t}\r)\r|\,\frac{dz}{t^n}\\
&\quad\quad+\int_{\rn\setminus 2^{k+3}B}\lf|f(z)-f_{B}\r|
\lf|\varphi\lf(\frac{x+\xi-z}{t}\r)-\varphi
\lf(\frac{\widetilde{x}+\xi-z}{t}\r)\r|\,\frac{dz}{t^n}.
\end{align*}
Thus, by this, \eqref{wJ-rn}, and \eqref{5.19x}, we find that,
for any $x\in B$ and $\wz x\in B$ such that $g_{\lambda,\infty}^{*}(f)(\wz x)<\fz$,
\begin{align}\label{glam*a2}
&\lf|g_{\lambda,\infty}^*(f)(x)-g_{\lambda,\infty}^*(f)(\widetilde{x})\r|\\
&\quad\leq\lf\{\sum_{k=1}^{\infty}
\iint_{\widetilde{J}(\mathbf{0},k)\setminus
\widetilde{J}(\mathbf{0},k-1)}
\lf(\frac{t}{t+|\xi|}\r)^{\lambda n}
\lf[\int_{2^{k+3}B}\lf|f(z)-f_{B}\r|\r.\r.\noz\\
&\qquad\lf.\lf.\times\lf|\varphi\lf(\frac{x+\xi-z}{t}\r)
-\varphi\lf(\frac{\widetilde{x}+\xi-z}{t}\r)\r|\,\frac{dz}{t^n}\r]^2
\,\frac{d\xi\,dt}{t^{n+1}}\r\}^{\frac{1}{2}}\noz\\
&\qquad+\lf\{\sum_{k=1}^{\infty}
\iint_{\widetilde{J}(\mathbf{0},k)\setminus
\widetilde{J}(\mathbf{0},k-1)}\lf(\frac{t}{t+|\xi|}\r)^{\lambda n}
\lf[\int_{\rn\setminus 2^{k+3}B}
\lf|f(z)-f_{B}\r|\r.\r.\noz\\
&\qquad\lf.\lf.\times\lf|\varphi\lf(\frac{x+\xi-z}{t}\r)
-\varphi\lf(\frac{\widetilde{x}+\xi-z}{t}\r)\r|\,\frac{dz}{t^n}\r]^2
\,\frac{d\xi\,dt}{t^{n+1}}\r\}^{\frac{1}{2}}\noz\\
&\quad=:\widetilde{\Lambda}_4^{(1)}(x,\wz x)+\widetilde{\Lambda}_4^{(2)}(x,\wz x).\noz
\end{align}

Next, we estimate $\wz \Lambda_4^{(1)}(x,\wz x)$
for any $x\in B$ and $\widetilde{x}\in B$
such that $g_{\lambda,\infty}^{*}(f)(\wz x)<\fz$.
Indeed, for any $(\xi,t)\in\rr_+^{n+1}$ and $\lambda_1,\lambda_2\in(1,\infty)$
satisfying $\lambda_1>\lambda_2$, we conclude that
$$\lf(\frac{t}{t+|\xi|}\r)^{\lambda_1 n}\leq\lf(\frac{t}{t+|\xi|}\r)^{\lambda_2 n}.$$
Using this and $q\in(1,\fz)$,
in the following estimates for $\widetilde{\Lambda}_4^{(1)}(x,\wz x)$,
we may always assume that
\begin{align}\label{lam-g*-1-01}
\lambda\in\lf(\max\lf\{1,\frac{2}{q}\r\}+\frac{2}{n},2+\frac{2}{n}\r)
\cap\lf(2,2+\frac{2}{n}\r).
\end{align}
Moreover, by the mean value theorem and \eqref{g-3'},
we conclude that, for any $t\in(4r,\fz)$, $x,\widetilde{x}\in B$,
$z\in 2^{k+3}B$ with $y\in\rn$ and $r\in(0,\fz)$,
and $(\xi,t)\in \widetilde{J}(\mathbf{0},k)
\setminus \widetilde{J}(\mathbf{0},k-1)$ with $k\in\nn$,
there exists a $\theta\in (0,1)$ such that
\begin{align*}
&\lf|\varphi\lf(\frac{x+\xi-z}{t}\r)
-\varphi\lf(\frac{\widetilde{x}+\xi-z}{t}\r)\r|\\
&\quad\lesssim\frac{t^{n+\dz}|x-\widetilde{x}|}{[t+|x+\xi-z
+\theta(\widetilde{x}-x)|]^{n+1+\dz}}
\lesssim\frac{t^{n+\dz}r}{(t+|x+\xi-z|-2r)^{n+1+\dz}}\\
&\quad\sim\frac{t^{n+\dz}r}{(t+|x+\xi-z|)^{n+1+\dz}},
\end{align*}
where, in the last step, we used $t\in(4r,\fz)$.
From this, the Minkowski inequality, and \eqref{lam-g*-1-01},
we deduce that, for any $x\in B$ and $\widetilde{x}\in B$
such that $g_{\lambda,\infty}^{*}(f)(\wz x)<\fz$,
\begin{align*}
\widetilde{\Lambda}_4^{(1)}(x,\wz x)
&\lesssim\lf\{\sum_{k=1}^{\infty}
\iint_{\widetilde{J}(\mathbf{0},k)\setminus
\widetilde{J}(\mathbf{0},k-1)}\lf(\frac{t}{t+|\xi|}\r)^{\lambda n}
\lf[\int_{2^{k+3}B}\frac{t^{\dz}r|f(z)-f_{B}|}{(t+|x+\xi-z|)^{n+1+\dz}}
\,dz\r]^2\,\frac{d\xi\,dt}{t^{n+1}}\r\}^{\frac{1}{2}}\\
&\lesssim \lf\{r^2\sum_{k=1}^{\infty}
\int_{\rn}\int_{0}^{2^{k+2}r}\lf(\frac{1}{2^kr}\r)^{\lambda n}\r.\\
&\quad\lf.\times\lf[\int_{2^{k+3}B}\frac{|f(z)
-f_{B}|}{(t+|x+\xi-z|)^{n+1+\dz}}\,dz\r]^{2}
\,\frac{dt\,d\xi}{t^{n-\lambda n+1-2\dz}}\r\}^{\frac{1}{2}}\\
&\lesssim \lf[r^2\sum_{k=1}^{\infty}
\lf(\frac{1}{2^kr}\r)^{\lambda n}\int_{\rn}\lf\{\int_{2^{k+3}B}
\lf|f(z)-f_{B}\r|\r.\r.\\
&\quad\lf.\lf.\times\lf[\int_0^{2^{k+2}r}
\frac{t^{\lambda n-n-1+2\dz}}{(t+|x+\xi-z|)^{2(n+1+\dz)}}\,dt\r]^\frac{1}{2}\,dz
\r\}^2\,d\xi\r]^{\frac{1}{2}}\\
&\lesssim \lf[r^2\sum_{k=1}^{\infty}
\lf(\frac{1}{2^kr}\r)^{\lambda n}\int_{\rn}\lf\{\int_{2^{k+3}B}
\lf|f(z)-f_{B}\r|\r.\r.\\
&\quad\lf.\lf.\times\lf[\lf(\int_0^{|x+\xi-z|}+\int_{|x+\xi-z|}^{\fz}\r)
\frac{t^{\lambda n-n-1+2\dz}}{(t+|x+\xi-z|)^{2(n+1+\dz)}}\,dt\r]^\frac{1}{2}\,dz
\r\}^2\,d\xi\r]^{\frac{1}{2}}\\
&\sim \lf\{r^2\sum_{k=1}^{\infty}\lf(\frac{1}{2^kr}\r)^{\lambda n}
\int_{\rn}\lf[\int_{2^{k+3}B}\frac{|f(z)-f_{B}|}
{|u-z|^{n-\frac{\lambda n-n-2}{2}}}\,dz\r]^2\,du\r\}^{\frac{1}{2}},
\end{align*}
which, combined with \eqref{lam-g*-1-01}, Lemma \ref{fractional}(ii)
with $\beta:=\frac{\lambda n-n-2}{2}\in (0,n)$,
$a:=\frac{1}{\frac{\lambda}{2}-\frac{1}{n}}\in (1,q)\cap(1,\frac{n}{\beta})$,
and $\widetilde{a}:=2$,
and Lemma \ref{sum-g} with $\theta:=\frac{1}{2}$, further implies that,
for any $x\in B$ and $\widetilde{x}\in B$
such that $g_{\lambda,\infty}^{*}(f)(\wz x)<\fz$,
\begin{align}\label{lam-04-02}
\widetilde{\Lambda}_4^{(1)}(x,\wz x)
&\lesssim \lf\{r^2\sum_{k=1}^{\infty}\lf(\frac{1}{2^kr}\r)^{\lambda n}
\lf[\int_{2^{k+3}B}\lf|f(z)-f_{B}
\r|^a\,dz\r]^\frac{2}{a}\r\}^{\frac{1}{2}}\\
&\lesssim r\sum_{k=1}^{\infty}\lf(\frac{1}{2^kr}
\r)^{\frac{\lambda n}{2}-n(\frac{\lambda}{2}-\frac{1}{n})}
\lf[\fint_{2^{k+3}B}\lf|f(z)-f_{B}\r|^a\,dz\r]^\frac{1}{a}\noz\\
&\lesssim\sum_{k=1}^{\infty}\frac{1}{2^{k}}
\lf[\fint_{2^{k+3}B}\lf|f(z)
-f_{B}\r|^q\,dz\r]^\frac{1}{q}\noz\\
&\lesssim\sum_{k=1}^{\infty}\frac{1}{2^{k}}
\lf[\fint_{2^{k}B}\lf|f(z)
-f_{2^{k}B}\r|^q\,dz\r]^\frac{1}{q}.\noz
\end{align}
This is a desired estimate of $\widetilde{\Lambda}_4^{(1)}(x,\wz x)$.

Now, we estimate $\wz \Lambda_4^{(2)}(x,\wz x)$.
Indeed, by the mean value theorem and \eqref{g-3'},
we conclude that, for any $t\in(4r,\fz)$, $x\in B$,
$\wz x\in B$ such that $g_{\lambda}^*(f)(\wz x)<\fz$,
$(\xi,t)\in \widetilde{J}(\mathbf{0},k)
\setminus \widetilde{J}(\mathbf{0},k-1)$ with $k\in\nn$,
and $z\in 2^{l+1}B\setminus 2^lB$ with integer $l\geq k+3$,
there exists a $\theta\in (0,1)$ such that
\begin{align*}
&\lf|\varphi\lf(\frac{x+\xi-z}{t}\r)
-\varphi\lf(\frac{\widetilde{x}+\xi-z}{t}\r)\r|\\
&\quad\lesssim\frac{t^{n+\dz}|x-\widetilde{x}|}{[t+|x+\xi-z
+\theta(\widetilde{x}-x)|]^{n+1+\dz}}
\lesssim\frac{t^{n+\dz}r}{|x+\xi-z|^{n+1+\dz}}
\sim\frac{t^{n+\dz}r}{(2^lr)^{n+1+\dz}},
\end{align*}
where, in the last step, we used
$$
|x+\xi-z|\geq |z-y|-|y-x|-|\xi|\geq
|y-z|-r-2^{k+2}r\sim |y-z|\sim 2^lr.
$$
From this, the Fubini theorem, and $\lambda>1$, we deduce that,
for any $x\in B$ and $\wz x\in B$ such that $g_{\lambda,\infty}^{*}(f)(\wz x)<\fz$,
\begin{align*}
\widetilde{\Lambda}_4^{(2)}(x,\wz x)
&\lesssim\lf\{\sum_{k=1}^{\infty}
\iint_{\widetilde{J}(\mathbf{0},k)\setminus
\widetilde{J}(\mathbf{0},k-1)}\lf(\frac{t}{2^{k+2}r}\r)^{\lambda n}
\lf[\sum_{l=k+3}^{\infty}\int_{2^{l+1}B\setminus2^lB}
\frac{rt^\dz|f(z)-f_{B}|}{(2^lr)^{n+1+\dz}}
\,dz\r]^2\,\frac{d\xi\,dt}{t^{n+1}}\r\}^{\frac{1}{2}}\\
&\lesssim\lf\{\sum_{k=1}^{\infty}
\iint_{\widetilde{J}(\mathbf{0},k)\setminus \widetilde{J}(\mathbf{0},k-1)}
\lf(\frac{1}{2^{k}r}\r)^{\lambda n}\lf[
\sum_{l=k+3}^{\infty}\frac{r^{-\dz}}{2^{(1+\dz)l}}\fint_{2^{l+1}B}
\lf|f(z)-f_{B}\r|\,dz\r]^2
\,\frac{d\xi\,dt}{t^{n-\lambda n+1-2\dz}}\r\}^{\frac{1}{2}}\\
&\lesssim\sum_{k=1}^{\infty}\sum_{l=k+3}^{\infty}\frac{r^{-\dz}}{2^{(1+\dz)l}}
\fint_{2^{l+1}B}\lf|f(z)-f_{B}\r|\,dz
\lf[\lf(\frac{1}{2^{k}r}\r)^{\lambda n}
\iint_{\widetilde{J}(\mathbf{0},k)\setminus\widetilde{J}(\mathbf{0},k-1)}
\,\frac{d\xi\,dt}{t^{n-\lambda n+1-2\dz}}\r]^{\frac{1}{2}}\\
&\lesssim\sum_{k=1}^{\infty}\sum_{l=k+3}^{\infty}
\frac{r^{-\dz}}{2^{(1+\dz)l}}\fint_{2^{l+1}B}\lf|f(z)-f_{B}\r|\,dz
\lf[\lf(\frac{1}{2^{k}r}\r)^{\lambda n}\int_{0}^{2^{k+2}r}
\int_{2^{k+2}B}\,\frac{d\xi\,dt}{t^{n-\lambda n+1-2\dz}}\r]^{\frac{1}{2}}\\
&\sim\sum_{l=4}^{\infty}\sum_{k=1}^{l-3}\frac{2^{k\dz}}{2^{(1+\dz)l}}
\fint_{2^{l+1}B}\lf|f(z)-f_{B}\r|\,dz
\sim\sum_{l=4}^{\infty}\frac{1}{2^l}\fint_{2^{l+1}B}
\lf|f(z)-f_{B}\r|\,dz,
\end{align*}
which, together with Lemma \ref{I-JN}, further implies that, for any $x\in B$ and
$\wz x\in B$ such that $g_{\lambda,\infty}^{*}(f)(\wz x)<\fz$,
\begin{align}\label{glam*a3}
&\widetilde{\Lambda}_4^{(2)}(x,\wz x)
\lesssim\sum_{l=1}^{\infty}\frac{1}{2^{l}}\fint_{2^{l}B}
\lf|f(z)-f_{2^{l}B}\r|\,dz.
\end{align}
This is a desired estimate of $\widetilde{\Lambda}_4^{(2)}(x,\wz x)$.

Moreover, by \eqref{glam*a2}, \eqref{lam-04-02}, \eqref{glam*a3},
and the H\"older inequality, we conclude that,
for any $x\in B$ and $\widetilde{x}\in B$
such that $g_{\lambda,\infty}^{*}(f)(\wz x)<\fz$,
\begin{align*}
\lf|g_{\lambda,y,\infty}^*(f)(x)-g_{\lambda,y,\infty}^*(f)(\widetilde{x})\r|
\lesssim\sum_{k=1}^{\infty}\frac{1}{2^{k}}
\lf[\fint_{2^{k}B}\lf|f(z)-f_{2^kB}\r|^q\,dz\r]^\frac{1}{q}.
\end{align*}
Thus,
\begin{align*}
\widetilde{\Lambda}_4
\lesssim\sum_{k=1}^{\infty}\frac{1}{2^{k}}
\lf[\fint_{2^{k}B}\lf|f(z)-f_{2^kB}\r|^q\,dz\r]^\frac{1}{q},
\end{align*}
which, combined with \eqref{g-lam-2-000}, \eqref{g-lam-2-01},
and \eqref{g-lam-2-03}, further implies that
\begin{align*}
&\lf[\fint_{B}\lf\{g_\lambda^*(f)(x)-\inf_{\widetilde{x}\in B}
g_\lambda^*(f)(\widetilde{x})\r\}^q\,dx\r]^{\frac{1}{q}}\\
&\quad\lesssim\sum_{k=1}^{\infty}\lf[\frac{1}{2^{k\dz}}
+\frac{1}{2^{\frac{kn(\lambda-2)}{2}}}\r]
\lf[\fint_{2^{k}B}\lf|f(z)-f_{2^{k}B}\r|^q\,dz\r]^{\frac{1}{q}}.
\end{align*}	
This finishes the proof of (ii), and hence of Theorem \ref{g-lam-01}.
\end{proof}

Next, we prove the main results of this section.
Indeed, the proof of Theorem \ref{g-l-bounded3} is similar to that
of Theorem \ref{g-l-bounded4} and hence we only prove the second one.

\begin{proof}[Proof of Theorem \ref{g-l-bounded4}]
Let $X$, $p_-$, $q$, $\dz$, $r$, $d$, $\lambda$, and $g_{\lambda}^*$
be the same as in the present theorem. The proofs of
both (i) and (ii) of the present theorem
are strongly depend, respectively, on both (i) and (ii) of Lemma \ref{g-lam-01},
and we only prove (ii) because the proof of (i) is similar.

We first claim that, if $g_{\lambda}^{*}(f)(x_0)<\fz$ for one $x_0\in \rn$, then
$g_{\lambda}^{*}(f)$ is finite almost everywhere. To this end,
let $m\in\nn$, $\{B_j\}_{j=1}^m\subset {\mathbb{B}}(\rn)$ be
such that $x_0\in B_j$ for any $j\in\nn$,
and $\{\lambda_j\}_{j=1}^m\subset[0,\fz)$ such that $\sum_{j=1}^m\lambda_j\neq 0$.
Also, for any $i\in\{1,\ldots,m\}$ and $k\in\nn$, let
$\lambda_{i,k}:=\frac{\lambda_i\|\mathbf{1}_{2^kB_i}\|_{X}}{\|\mathbf{1}_{B_i}\|_{X}}$.
By Lemma \ref{g-lam-01}, \eqref{g-thm-B-02}, the definition of
$\|\cdot\|_{\mathcal{L}_{X,q,0,d}(\rn)}$,
$\lambda\in(\frac{2}{\min\{d,p_-\}},\infty)$, and
$\min\{d,p_-\}\in(\frac{n}{n+\dz},\fz)$, we conclude that,
for any given $r\in(\frac{n}{n+\dz},\min\{d,p_-\})
\cap(\frac{2}{\lambda},\min\{d,p_-\})$,
\begin{align}\label{thm-B-l-07}
&\lf\|\lf\{\sum_{i=1}^m\lf(\frac{\lambda_{i}}{\|\mathbf{1}_{B_i}\|_{X}}\r)^d
\mathbf{1}_{B_i}\r\}^{\frac{1}{d}}\r\|_{X}^{-1}
\sum_{j=1}^m\frac{\lambda_j|B_j|}{\|\mathbf{1}_{B_j}\|_{X}}
\lf[\fint_{B_j}\lf\{g_{\lambda}^{*}(f)(x)
-\inf_{\wz x\in B_j}g(f)(\wz x)\r\}^q\,dx\r]^{\frac{1}{q}}\\
&\quad\ls\lf\|\lf\{\sum_{i=1}^m\lf(\frac{\lambda_{i}}{\|\mathbf{1}_{B_i}\|_{X}}\r)^d
\mathbf{1}_{B_i}\r\}^{\frac{1}{d}}\r\|_{X}^{-1}\noz\\
&\qquad\times\sum_{j=1}^m\frac{\lambda_j|B_j|}{\|\mathbf{1}_{B_j}\|_{X}}
\sum_{k=1}^{\infty}\lf[\frac{1}{2^{k\dz}}
+\frac{1}{2^{\frac{kn(\lambda-2)}{2}}}\r]
\lf[\fint_{2^{k}B_j}\lf|f(z)-f_{2^{k}B_j}\r|^q\,dz\r]^{\frac{1}{q}}\noz\\
&\quad\ls\sum_{k\in\nn}2^{k(-n+\frac{n}{r})}\lf\|\lf\{\sum_{i=1}^m
\lf(\frac{\lambda_{i,k}}{\|\mathbf{1}_{2^kB_i}\|_{X}}\r)^d
\mathbf{1}_{2^kB_i}\r\}^{\frac{1}{d}}\r\|_{X}^{-1}\noz\\
&\qquad\times\sum_{j=1}^m\frac{\lambda_{j,k}|2^kB_j|}
{\|\mathbf{1}_{2^kB_j}\|_{X}}
\lf[\frac{1}{2^{k\dz}}+\frac{1}{2^{\frac{kn(\lambda-2)}{2}}}\r]
\lf[\fint_{2^kB_j}\lf|f(x)-f_{2^kB_j}\r|^q\,dx\r]^{\frac{1}{q}}\noz\\
&\quad\ls\sum_{k\in\nn}\lf[2^{k(-n-\dz+\frac{n}{r})}
+2^{kn(-\frac{\lambda}{2}+\frac{1}{r})}\r]
\|f\|_{\mathcal{L}_{X,q,0,d}(\rn)}\noz
\ls\|f\|_{\mathcal{L}_{X,q,0,d}(\rn)}<\fz.\noz
\end{align}
From this, we deduce that, for any $j\in\nn$,
$$\lf[\fint_{B_j}\lf\{g_{\lambda}^{*}(f)(x)-\inf_{\wz x\in B_j}
g_{\lambda}^{*}(f)(\wz x)\r\}^q\,dx\r]^{\frac{1}{q}}<\fz$$
and hence $g_{\lambda}^{*}(f)(x)<\fz$ for almost every $x\in B_j$,
which, together with the arbitrariness of $B_j$,
further implies that, for almost every $x\in\rn$, $g_{\lambda}^{*}(f)(x)<\fz$.
Thus, the above claim holds true.

Now, we show that \eqref{thm-b-3'} holds true
if $g_{\lambda}^{*}(f)(x_0)<\fz$ for one $x_0\in \rn$.
In this case, by the above claim and an argument similar to
that used in the estimation of \eqref{thm-B-l-07},
we conclude that, for any $m\in\nn$, any balls $\{B_j\}_{j=1}^m\subset {\mathbb{B}}(\rn)$,
and any $\{\lambda_j\}_{j=1}^m\subset[0,\fz)$ with $\sum_{j=1}^m\lambda_j\neq 0$,
\begin{align*}
&\lf\|\lf\{\sum_{i=1}^m\lf(\frac{\lambda_{i}}{\|\mathbf{1}_{B_i}\|_{X}}\r)^d
\mathbf{1}_{B_i}\r\}^{\frac{1}{d}}\r\|_{X}^{-1}
\sum_{j=1}^m\frac{\lambda_j|B_j|}{\|\mathbf{1}_{B_j}\|_{X}}
\lf[\fint_{B_j}\lf\{g_{\lambda}^*(f)(x)-\inf_{\wz x\in B_j}
g_{\lambda}^{*}(f)(\wz x)\r\}^q\,dx\r]^{\frac{1}{q}}\\
&\quad\ls\|f\|_{\mathcal{L}_{X,q,0,d}(\rn)},
\end{align*}
which, combined with the definition of
$\|\cdot\|_{\mathcal{L}^{\mathrm{low}}_{X,q,0,d}(\rn)}$,
further implies that \eqref{thm-b-3'} holds true.
This finishes the proof of (ii), and hence of Theorem \ref{g-l-bounded4}.
\end{proof}

\begin{remark}\label{g-l-bound-rem-B}
\begin{enumerate}
\item[\rm (i)] 		
Let $X$ be concave and $d\in(0,1]$. In this case, from Remark \ref{rem-ball-B3}(i),
we deduce that Theorem \ref{g-l-bounded4} coincides with Theorem \ref{g-l-bounded3}.
	
\item[\rm (ii)]
Let $X:=L^{1}(\rn)$, $\dz=1$, and $q\in(1,\fz)$.
In this case, $\mathcal{L}_{X,q,0}(\rn)=\mathrm{BMO}(\rn)$
and the case (i) of Theorem \ref{g-l-bounded3} coincides with \cite[Corollary 1.3]{MY}.	
	
\item[\rm (iii)]
Let $p\in(\frac{n}{n+1},\fz)$, $X:=L^p(\rn)$, $\dz=1$, and $q\in(1,\fz)$. In this case,
$\mathcal{L}_{X,q,0}(\rn)=\mathcal{C}_{\frac{1}{p}-1,q,0}(\rn)$ and
Theorem \ref{g-l-bounded3} improves the range
$\lambda\in(3+\frac{2}{n},\fz)$ in \cite[Theorem 3]{Sun04}
into $\lambda\in(\max\{1,\frac{2}{q}\},\infty)$
when $p\in(\frac{n}{n+1/2},\fz)$,
or $\lambda\in(\max\{1,\frac{2}{q}\}+\frac{2}{n},\infty)\cap(\frac{2}{p},\infty)$
when $p\in(\frac{n}{n+1}, \frac{n}{n+1/2}]$.

\item[\rm (iv)]
Let $s\in\nn$. In this case,
it is still unclear
whether or not the Littlewood--Paley $g_{\lambda}^*$-function
is bounded on either $\mathcal{L}_{X,q,s}(\rn)$ or $\mathcal{L}_{X,q,s,d}(\rn)$.
\end{enumerate}
\end{remark}

\section{Applications\label{S-ap}}

Theorems \ref{thm-B-Banach}, \ref{thm-B-Banach-S},
\ref{g-l-bounded3}, and \ref{g-l-bounded4} can be applied to a plenty of
concrete examples of ball quasi-Banach function spaces, for instance,
weighted Lebesgue spaces, mixed-norm Lebesgue spaces, variable Lebesgue spaces,
Orlicz spaces, and Orlicz--slice spaces.
Therefore, we obtain the boundedness of Littlewood--Paley operators on these spaces.
Moreover, to the best of our knowledge, all these results are new.
These examples explicitly reveal that more applications of the
main results of this article to new function spaces are possible. Furthermore,
we also apply Lemmas \ref{lem-g-fz}, \ref{lem-s-fz},
and \ref{g-lam-01} to $JN_{p,q,\alpha}^{\mathrm{con}}(\rn)$
[the special John--Nirenberg--Campanato spaces via congruent cubes
(see, for instance, \cite[Definition 1.3]{jtyyz1})], and then
obtain the boundedness of Littlewood--Paley operators on
$JN_{p,q,\alpha}^{\mathrm{con}}(\rn)$
(see Subsection \ref{g-bounded-JN} below), which is also new.

\subsection{Weighted Lebesgue Spaces\label{weight}}

In this subsection, we apply
Theorems \ref{thm-B-Banach}, \ref{thm-B-Banach-S},
\ref{g-l-bounded3}, and \ref{g-l-bounded4}
to weighted Lebesgue spaces.
To this end, we first recall the definition of Muckenhoupt weights
$A_p(\rn)$ (see, for instance, \cite[Definitions 7.1.2 and 7.1.3]{G1}).

\begin{definition}
Let $p\in[1,\infty)$ and $\omega$ be a nonnegative locally
integrable function on $\rn$. Then
$\omega$ is called an \emph{$A_{p}(\rn)$ weight},
denoted by $\omega\in A_p(\rn)$, if, when $p\in(1,\infty)$,
\begin{align*}
[\omega]_{A_{p}\left(\mathbb{R}^{n}\right)}
:=\sup _{B \subset \mathbb{R}^{n}}\frac{1}{|B|}
\lf[\int_{B}\omega(x)\,dx\r]\left\{\frac{1}{|B|}
\int_{B}[\omega(x)]^{-\frac{1}{p-1}} \,dx\right\}^{p-1}<\infty,
\end{align*}
and
\begin{align*}
[\omega]_{A_{1}\left(\mathbb{R}^{n}\right)}
:=\sup _{B \subset \mathbb{R}^{n}}\frac{1}{|B|}
\lf[\int_{B}\omega(x)\,dx\r]\left\{\mathop{\mathrm{ess\,sup}}_{x\in\rn}
[\omega(x)]^{-1}\right\}<\infty,
\end{align*}
where the suprema are taken over all balls $B \in\mathbb{B}(\rn)$.
Moreover, the \emph{class $A_\infty(\rn)$} is defined by setting
$$
A_\infty(\rn):=\bigcup_{p\in[1,\infty)}A_p(\rn).
$$
\end{definition}

\begin{definition}\label{def-wei}
Let $p\in(0,\infty)$ and $\omega\in A_\infty(\rn)$.
The \emph{weighted Lebesgue space} $L_{\omega}^{p}(\mathbb{R}^{n})$ is
defined to be the set of all the measurable
functions $f$ on $\rn$ such that
\begin{align*}
\lf\|f\r\|_{L^p_\omega(\rn)}=\lf[\int_{\rn}|f(x)|^p\omega(x)\,dx\r]
^{\frac{1}{p}}<\infty.
\end{align*}
\end{definition}

\begin{remark}		
Let $p\in(0,\fz)$ and $\omega\in A_{\fz}(\rn)$.
As was pointed out in \cite[Section 7.1]{SHYY}, $L^p_\omega(\rn)$ is a
ball quasi-Banach space but it may not be a quasi-Banach function space.
\end{remark}

In what follows, for any
$\omega\in A_\fz(\rn)$, let
\begin{align*}
q_{\omega}:=\inf\lf\{q\in[1,\fz):\ \omega\in A_q(\rn)\r\}.
\end{align*}
The following theorem is a corollary of Theorems \ref{thm-B-Banach}, \ref{thm-B-Banach-S},
\ref{g-l-bounded3}, and \ref{g-l-bounded4}.

\begin{theorem}\label{weighted}
Let $\dz\in(0,1]$, $p\in(0,\fz)$, and $\omega\in A_\fz(\rn)$ satisfy
$p_-:=\frac{p}{q_\omega}\in(\frac{n}{n+\dz},\fz)$, and let $d\in(\frac{n}{n+\dz},\fz)$.
Let $\lambda\in(\max\{1,\frac{2}{q}\},\infty)$
when $\min\{d,p_-\}\in(\frac{n}{n+\frac{\dz}{2}},\fz)$,
or $\lambda\in(\max\{1,\frac{2}{q}\}+\frac{2}{n},\infty)
\cap(\frac{2}{\min\{d,p_-\}},\infty)$
when $\min\{d,p_-\}\in(\frac{n}{n+\dz},\frac{n}{n+\frac{\dz}{2}}]$.
Let $\Lambda$ be $g$, $S$, or $g_{\lambda}^*$ in Definition \ref{g-f}.
Then, for any $f\in \mathcal{L}_{L^{p}_{\omega}(\rn),q,0,d}(\rn)$,
$\Lambda(f)$ is either infinite everywhere or finite almost everywhere and, in the latter case,
there exists a positive constant $C$, independent of $f$, such that
\begin{align*}
\|\Lambda(f)\|_{\mathcal{L}^{\mathrm{low}}_{L^{p}_{\omega}(\rn),q,0,d(\rn)}}
\leq C\|f\|_{L^p_{\omega}(\rn),q,0,d}(\rn).
\end{align*}
Similar results also hold true for the space $\mathcal{L}_{L^{p}_{\omega}(\rn),q,0}(\rn)$.
\end{theorem}

\begin{proof}
Let all the symbols be the same as in the present theorem.
It is easy to show that $L^{p}_{\omega}(\rn)$ is a ball quasi-Banach function space
(see, for instance, \cite[Section 7.1]{SHYY}).
Moreover, by \cite[Theorem 3.1(b)]{aj1980}, we conclude that $L^p_\omega(\rn)$
satisfies Assumption \ref{assump1} with both $X:=L^p_\omega(\rn)$
and $p_-$ in the present theorem.
Thus, applying Theorems \ref{thm-B-Banach}, \ref{thm-B-Banach-S},
\ref{g-l-bounded3}, and \ref{g-l-bounded4}
with $X:=L_{\omega}^p(\rn)$, we then obtain the
desired conclusions and hence complete the proof of Theorem \ref{weighted}.
\end{proof}

\subsection{Mixed-Norm Lebesgue Spaces\label{3s1}}

Recall that Benedek and Panzone \cite{BP} introduced the
mixed-norm Lebesgue space $L^{\vec{p}}(\rn)$ with $\vec{p}\in (0,\fz]^n$,
which is a natural generalization of the Lebesgue space $L^p(\rn)$.
We refer the reader to \cite{cgn17,HLYY,HLYY2019,hlyy20,hy}
for more studies on mixed-norm Lebesgue spaces.
To be precise, the mixed-norm Lebesgue space is defined as follows.

\begin{definition}\label{mix}
Let $\vec{p}:=(p_1,\ldots,p_n)\in(0,\infty]^n$.
The \emph{mixed-norm Lebesgue space $L^{\vec{p}}(\rn)$} is defined
to be the set of all the measurable functions $f$ on $\rn$ such that
$$
\|f\|_{L^{\vec{p}}(\rn)}:=\lf\{\int_{\rr}\cdots
\lf[\int_{\rr}|f(x_1,\ldots,x_n)|^{p_1}\,dx_1\r]
^{\frac{p_2}{p_1}}\cdots\,dx_n\r\}^{\frac{1}{p_n}}<\infty
$$
with the usual modifications made when $p_i=\infty$ for some $i\in\{1,\ldots,n\}$.
Moreover, let
\begin{align}\label{3e1}
p_-:= \min\{p_1,\ldots,p_n\}.
\end{align}
\end{definition}

\begin{remark}\label{mix-r}
\begin{enumerate}	
\item[\rm (i)]		
Let $\vec{p}\in(0, \infty]^{n}$.	
It is easy to show that
$L^{\vec{p}}(\mathbb{R}^{n})$
is a ball quasi-Banach function space (see, for instance, \cite[p.\,2047]{zyyw}).
However, as was pointed in \cite[Remark 7.21]{zyyw}, $L^{\vec{p}}(\rn)$
may not be a quasi-Banach function space.
\item[\rm (ii)]	
We refer the reader to \cite{Ho2016,Ho-2020} for more studies on the
boundedness of (sub)linear operators on mixed-norm (Hardy) spaces.
\end{enumerate}
\end{remark}

The following theorem is a corollary of Theorems \ref{thm-B-Banach}, \ref{thm-B-Banach-S},
\ref{g-l-bounded3}, and \ref{g-l-bounded4}.

\begin{theorem}\label{3t1}
Let $\dz\in(0,1]$, $\vec{p}\in(0,\fz)^n$, $p_-\in(\frac{n}{n+\dz},\fz)$ be the same as in \eqref{3e1},
$q\in(1,\fz)$, and $d\in(\frac{n}{n+\dz},\fz)$.
Let $\lambda\in(\max\{1,\frac{2}{q}\},\infty)$ when $\min\{d,p_-\}\in(\frac{n}{n+\frac{\dz}{2}},\fz)$,
or $\lambda\in(\max\{1,\frac{2}{q}\}+\frac{2}{n},\infty)
\cap(\frac{2}{\min\{d,p_-\}},\infty)$
when $\min\{d,p_-\}\in(\frac{n}{n+\dz},\frac{n}{n+\frac{\dz}{2}}]$.
Let $\Lambda$ be $g$, $S$, or $g_{\lambda}^*$ in Definition \ref{g-f}.
Then, for any $f\in \mathcal{L}_{L^{\vec{p}}(\rn),q,0,d}(\rn)$,
$\Lambda(f)$ is either infinite everywhere or finite almost everywhere and, in the latter case,
there exists a positive constant $C$, independent of $f$, such that
\begin{align*}
\|\Lambda(f)\|_{\mathcal{L}^{\mathrm{low}}_{L^{\vec{p}}(\rn),q,0,d}(\rn)}
\leq C\|f\|_{\mathcal{L}_{L^{\vec{p}}(\rn),q,0,d}(\rn)}.
\end{align*}
Similar results also hold true for the space $\mathcal{L}_{L^{\vec{p}}(\rn),q,0}(\rn)$.
\end{theorem}

\begin{proof}
Let all the symbols be the same as in the present theorem.
It is easy to show that $L^{\vec{p}}(\rn)$ is a ball quasi-Banach function space
(see, for instance, \cite[p.\,2047]{zyyw}).
Moreover, from \cite[Lemma 3.5]{HLYY}, we deduce that $L^{\vec{p}}(\rn)$ satisfies
Assumption \ref{assump1} with both $X:=L^{\vec{p}}(\rn)$ and $p_-$ in \eqref{3e1}.
Thus, applying Theorems \ref{thm-B-Banach}, \ref{thm-B-Banach-S},
\ref{g-l-bounded3}, and \ref{g-l-bounded4}
with $X:=L^{\vec{p}}(\rn)$, we then obtain the
desired conclusions and hence complete the proof of Theorem \ref{3t1}.
\end{proof}

\subsection{Variable Lebesgue Spaces\label{3s2}}

Let $p(\cdot):\ \rn\to[0,\infty)$ be a measurable function.
Recall that the \emph{variable Lebesgue space $L^{p(\cdot)}(\rn)$} is defined to be
the set of all the measurable functions $f$ on $\rn$ such that
$$
\|f\|_{L^{p(\cdot)}(\rn)}:=\inf\lf\{\lambda\in(0,\infty):\ \int_\rn\lf[\frac{|f(x)|}{\lambda}\r]^{p(x)}\,dx\le1\r\}<\infty.
$$
We refer the reader to \cite{CUF,DHR,Ho2018,is2012,int2014,ins2015,ms2018}
for more studies on variable Lebesgue spaces.
In the reminder of this subsection, for any measurable function
$p(\cdot):\ \rn\to(0,\infty)$, we always let
\begin{align}\label{3e2}
p_-:=\underset{x\in\rn}{\essinf}\,p(x)\quad\text{and}\quad
p_+:=\underset{x\in\rn}{\esssup}\,p(x).
\end{align}
Recall that $p(\cdot):\ \rn\to(0,\infty)$ is said to be
\emph{globally log-H\"older continuous} if there
exists a $p_{\infty}\in\rr$ such that, for any $x,y\in\rn$,
$$|p(x)-p(y)|\lesssim\frac{1}{\log(e+1/|x-y|)}
\quad\text{and}\quad
|p(x)-p_\infty|\lesssim\frac{1}{\log(e+|x|)},$$
where the implicit positive constants are independent of both $x$ and $y$.
As was pointed out in \cite[Section 7.1]{SHYY}, $L^{p(\cdot)}(\rn)$ is a
ball quasi-Banach space but it may not be a quasi-Banach function space.

The following theorem is a corollary of Theorems \ref{thm-B-Banach}, \ref{thm-B-Banach-S},
\ref{g-l-bounded3}, and \ref{g-l-bounded4}.

\begin{theorem}\label{3t2}
Let $\dz\in(0,1]$, $p(\cdot):\ \rn\to(0,\infty)$ be a globally
log-H\"older continuous function satisfying
$\frac{n}{n+\dz}<p_-\le p_+<\infty$,
$q\in(1,\fz)$, and $d\in(\frac{n}{n+\dz},\fz)$.
Let $\lambda\in(\max\{1,\frac{2}{q}\},\infty)$ when $\min\{d,p_-\}\in(\frac{n}{n+\frac{\dz}{2}},\fz)$,
or $\lambda\in(\max\{1,\frac{2}{q}\}+\frac{2}{n},\infty)
\cap(\frac{2}{\min\{d,p_-\}},\infty)$
when $\min\{d,p_-\}\in(\frac{n}{n+\dz},\frac{n}{n+\frac{\dz}{2}}]$.
Let $\Lambda$ be $g$, $S$, or $g_{\lambda}^*$ in Definition \ref{g-f}.
Then, for any $f\in \mathcal{L}_{L^{p(\cdot)}(\rn),q,0,d}(\rn)$,
$\Lambda(f)$ is either infinite everywhere or finite almost everywhere and, in the latter case,
there exists a positive constant $C$, independent of $f$, such that
\begin{align*}
\|\Lambda(f)\|_{\mathcal{L}^{\mathrm{low}}_{L^{p(\cdot)}(\rn),q,0,d}(\rn)}
\leq C\|f\|_{\mathcal{L}_{L^{p(\cdot)}(\rn),q,0,d}(\rn)}.
\end{align*}
Similar results also hold true for the space $\mathcal{L}_{L^{p(\cdot)}(\rn),q,0}(\rn)$.	
\end{theorem}

\begin{proof}
Let all the symbols be the same as in the present theorem.
It is easy to show that $L^{p(\cdot)}(\rn)$ is a
ball quasi-Banach function space
(see, for instance, \cite[Section 7.1]{SHYY}).
Moreover, from \cite[Corollary 2.1]{cfmp06}, we deduce that
Assumption \ref{assump1} with both $X:=L^{p(\cdot)}(\rn)$
and $p_-$ in \eqref{3e2} is satisfied.
Thus, applying Theorems \ref{thm-B-Banach}, \ref{thm-B-Banach-S},
\ref{g-l-bounded3}, and \ref{g-l-bounded4} with $X:=L^{p(\cdot)}(\rn)$,
we then obtain the desired conclusions and hence complete the proof of Theorem \ref{3t2}.
\end{proof}

\subsection{Orlicz Spaces}\label{3s4}

Recall that a function $\Phi:\ [0,\infty)\to[0,\infty)$
is called an \emph{Orlicz function} if it is
non-decreasing and satisfies $\Phi(0)= 0$, $\Phi(t)>0$ whenever $t\in(0,\infty)$,
and $\lim_{t\to\infty}\Phi(t)=\infty$.
An Orlicz function $\Phi$ is said to be
of \emph{lower} (resp., \emph{upper}) \emph{type} $p$ with $p\in(-\infty,\infty)$ if
there exists a positive constant $C_{(p)}$, depending on $p$,
such that, for any $t\in[0,\infty)$
and $s\in(0,1)$ [resp., $s\in [1,\infty)$],
\begin{align*}
\Phi(st)\le C_{(p)}s^p \Phi(t).
\end{align*}
An Orlicz function $\Phi$ is said to be of
\emph{positive lower} (resp., \emph{upper}) \emph{type} if it is of lower
(resp., upper) type $p$ for some $p\in(0,\infty)$.
Then we recall the definition of Orlicz spaces as follows
(see, for instance, \cite[p.\,58, Definition 2]{RR}).

\begin{definition}\label{fine}
Let $\Phi$ be an Orlicz function with positive lower type
$p_{\Phi}^-$ and positive upper type $p_{\Phi}^+$.
The \emph{Orlicz space $L^\Phi(\rn)$} is defined
to be the set of all the measurable functions $f$ on $\rn$ such that
$$\|f\|_{L^\Phi(\rn)}:=\inf\lf\{\lambda\in(0,\infty):\ \int_{\rn}
\Phi\lf(\frac{|f(x)|}{\lambda}\r)\,dx\le1\r\}<\infty.$$
\end{definition}

\begin{remark}
\begin{enumerate}	
\item[\rm (i)]		
Let $\Phi$ be an Orlicz function on $\rn$ with positive lower type
$p_{\Phi}^-\in(0,\fz)$ and positive upper type $p_{\Phi}^+\in(0,\fz)$.
As was pointed out in \cite[Section 7.6]{SHYY},
$L^{\Phi}(\rn)$ is a quasi-Banach
function space and hence a ball quasi-Banach function space.
\item[\rm (ii)]	
We refer the reader to \cite{ans2021,Ho2013,hns2016,Na2001,ss2018,ss2019,sst2012} for more studies on
the boundedness of (sub)linear operators on Orlicz-type spaces.
\end{enumerate}
\end{remark}

The following theorem is a corollary of Theorems \ref{thm-B-Banach}, \ref{thm-B-Banach-S},
\ref{g-l-bounded3}, and \ref{g-l-bounded4}.

\begin{theorem}\label{3t3}
Let $\dz\in(0,1]$, $\Phi$ be an Orlicz function with lower type $p_{\Phi}^-$
and upper type $p_{\Phi}^+$ satisfying $\frac{n}{n+\dz}<p_{\Phi}^-\le p_{\Phi}^+<\fz$,
$q\in(1,\fz)$, and $d\in(\frac{n}{n+\dz},\fz)$.
Let $\lambda\in(\max\{1,\frac{2}{q}\},\infty)$ when
$\min\{d,p_{\Phi}^-\}\in(\frac{n}{n+\frac{\dz}{2}},\fz)$,
or $\lambda\in(\max\{1,\frac{2}{q}\}+\frac{2}{n},\infty)
\cap(\frac{2}{\min\{d,p_{\Phi}^-\}},\infty)$
when $\min\{d,p_{\Phi}^-\}\in(\frac{n}{n+\dz},\frac{n}{n+\frac{\dz}{2}}]$.
Let $\Lambda$ be $g$, $S$, or $g_{\lambda}^*$ in Definition \ref{g-f}.
Then, for any $f\in \mathcal{L}_{L^\Phi(\rn),q,0,d}(\rn)$,
$\Lambda(f)$ is either infinite everywhere or finite almost everywhere and, in the latter case,
there exists a positive constant $C$, independent of $f$, such that
\begin{align*}
\|\Lambda(f)\|_{\mathcal{L}^{\mathrm{low}}_{L^\Phi(\rn),q,0,d}(\rn)}
\leq C\|f\|_{\mathcal{L}_{L^\Phi(\rn),q,0,d}(\rn)}.
\end{align*}
Similar results also hold true for the space $\mathcal{L}_{L^\Phi(\rn),q,0}(\rn)$.	
\end{theorem}

\begin{proof}
Let all the symbols be the same as in the present theorem.
It is easy to show that $L^\Phi(\rn)$ is a
ball quasi-Banach function space (see, for instance, \cite[Section 7.6]{SHYY}).
Moreover, from \cite[Theorem 2.10]{lhy12}
(see also \cite[Theorem 7.12]{SHYY}),
we deduce that Assumption \ref{assump1} with both $X:=L^\Phi(\rn)$
and $p_-:=p_{\Phi}^{-}$ is satisfied.
Thus, applying Theorems \ref{thm-B-Banach}, \ref{thm-B-Banach-S},
\ref{g-l-bounded3}, and \ref{g-l-bounded4}
with $X:=L^\Phi(\rn)$, we then obtain the
desired conclusions and hence complete the proof of Theorem \ref{3t3}.
\end{proof}

\subsection{Orlicz-Slice Spaces\label{3s3}}

In this subsection, we apply
the main results of this article to the Orlicz-slice space $(E_\Phi^r)_t(\rn)$
which was first introduced in \cite{zyyw2019}.
To this end, we first recall the following
definition of Orlicz-slice spaces which is just \cite[Definition 2.8]{zyyw2019}.

\begin{definition}\label{so}
Let $t,r\in(0,\infty)$ and $\Phi$ be an Orlicz function
with positive lower type $p_{\Phi}^-$ and
positive upper type $p_{\Phi}^+$. The \emph{Orlicz-slice space} $(E_\Phi^r)_t(\rn)$
is defined to be the set of all the measurable functions $f$ on $\rn$ such that
$$
\|f\|_{(E_\Phi^r)_t(\rn)}
:=\lf\{\int_{\rn}\lf[\frac{\|f\mathbf{1}_{B(x,t)}\|_{L^\Phi(\rn)}}
{\|\mathbf{1}_{B(x,t)}\|_{L^\Phi(\rn)}}\r]^r\,dx\r\}^{\frac{1}{r}}<\infty.
$$
\end{definition}

\begin{remark}
\begin{enumerate}
\item[\rm (i)]	
Let $t,r\in(0,\infty)$ and $\Phi$ be an Orlicz function on $\rn$
with $0<p_{\Phi}^-\le p_{\Phi}^+<\infty$.
By \cite[Lemma 2.28]{zyyw2019},
we find that $(E_\Phi^r)_t(\rn)$ is a
ball quasi-Banach function space, but it may not be a quasi-Banach function space
(see, for instance, \cite[Remark 7.43(i)]{zyyw}).

\item[\rm (ii)]	
We refer the reader to \cite{Ho-01,Ho2022,zyw2022} for more studies on the boundedness of
(sub)linear operators on (local) Orlicz-slice spaces.
\end{enumerate}
\end{remark}

The following theorem is a corollary of Theorems \ref{thm-B-Banach}, \ref{thm-B-Banach-S},
\ref{g-l-bounded3}, and \ref{g-l-bounded4}.

\begin{theorem}\label{3t4}
Let $\dz\in(0,1]$ and $\Phi$ be an Orlicz function with both lower type $p_{\Phi}^-$
and upper type $p_{\Phi}^+$ satisfying $\frac{n}{n+\dz}<p_{\Phi}^-\le p_{\Phi}^+<\fz$.
Let $t\in(0,\infty)$, $r\in(\frac{n}{n+\dz},\fz)$, $q\in(1,\fz)$, and $d\in(\frac{n}{n+\dz},\fz)$.
Let $\lambda\in(\max\{1,\frac{2}{q}\},\infty)$ when
$\min\{d,p_{\Phi}^-,r\}\in(\frac{n}{n+\frac{\dz}{2}},\fz)$,
or $\lambda\in(\max\{1,\frac{2}{q}\}+\frac{2}{n},\infty)
\cap(\frac{2}{\min\{d,p_{\Phi}^-,r\}},\infty)$
when $\min\{d,p_{\Phi}^-,r\}\in(\frac{n}{n+\dz},\frac{n}{n+\frac{\dz}{2}}]$.
Let $\Lambda$ be $g$, $S$, or $g_{\lambda}^*$ in Definition \ref{g-f}.
Then, for any $f\in \mathcal{L}_{(E_\Phi^r)_t(\rn),q,0,d}(\rn)$,
$\Lambda(f)$ is either infinite everywhere or finite almost everywhere and, in the latter case,
there exists a positive constant $C$, independent of $f$, such that
\begin{align*}
\|\Lambda(f)\|_{\mathcal{L}^{\mathrm{low}}_{(E_\Phi^r)_t(\rn),q,0,d}(\rn)}
\leq C\|f\|_{\mathcal{L}_{(E_\Phi^r)_t(\rn),q,0,d}(\rn)}.
\end{align*}
Similar results also hold true for the space $\mathcal{L}_{(E_\Phi^r)_t(\rn),q,0}(\rn)$.	
\end{theorem}

\begin{proof}
Let all the symbols be the same as in the present theorem.
By \cite[Lemma 2.28]{zyyw2019}, we find that $(E_\Phi^r)_t(\rn)$ is a
ball quasi-Banach function space. Moreover, from \cite[Lemma 4.3]{zyyw2019},
we deduce that Assumption \ref{assump1}
with both $X:=(E_\Phi^r)_t(\rn)$ and $p_-:=\min\{r,p_{\Phi}^-\}$ is satisfied.
Thus, applying Theorems \ref{thm-B-Banach}, \ref{thm-B-Banach-S},
\ref{g-l-bounded3}, and \ref{g-l-bounded4} with $X:=(E_\Phi^r)_t(\rn)$,
we then obtain the desired conclusions and hence complete the proof of Theorem \ref{3t3}.
\end{proof}

\subsection{Special John--Nirenberg--Campanato Spaces via Congruent Cubes\label{g-bounded-JN}}

Recall that John and Nirenberg \cite{JN} introduced the
well-known space $\mathop{\mathrm{BMO}\,}(\rn)$ and, in the same article, they also
introduced the mysterious John--Nirenberg space.
Recall that Dafni et al. \cite{DHKY} showed
the non-triviality of the John--Nirenberg space. Indeed, the John--Nirenberg space now
attracts more and more attention. We refer the reader to
\cite{ABKY,FPW,M,MM,SXY,TYY19,TYY2,TYY20S} for more studies on these spaces.
However, it is still a challenging and open question to obtain the boundedness
of some important operators on the John--Nirenberg space.

To shed some light on the boundedness of Littlewood--Paley operators
on John--Nirenberg spaces, Jia et al. \cite[Definition 1.3]{jtyyz1} introduced
the special John--Nirenberg--Campanato space via congruent cubes,
$JN_{(p,q,s)_\alpha}^{\mathrm{con}}(\rn)$.
In this subsection, we use Lemmas \ref{lem-g-fz},
\ref{lem-s-fz}, and \ref{g-lam-01} to obtain the boundedness of
Littlewood--Paley operators on $JN_{(p,q,0)_\alpha}^{\mathrm{con}}(\rn)$.
We refer the reader to \cite{jtyyz2,jtyyz3}, respectively, for the boundedness of
Calder\'on--Zygmund operators and fractional integrals on
$JN_{(p,q,s)_\alpha}^{\mathrm{con}}(\rn)$.

In what follows, a \emph{cube} always has a finite
edge length and all its edges parallel to the coordinate
axes, but it is not necessary to be open or closed. Moreover,
for any $m\in \zz$, we use $\mathcal{D}_m(\rn)$ to
denote the set of all the subcubes of $\rn$ with side length $2^{-m}$.

\begin{definition}\label{Defin.jncc1}
Let $p,q\in[1,\infty]$ and $\alpha\in\rr$.
Then the \emph{special John--Nirenberg--Campanato space via congruent cubes},
$JN_{p,q,\alpha}^{\mathrm{con}}(\rn)$, is defined to be the set of all the
$f\in L^q_\loc(\rn)$ such that
\begin{align*}
&\|f\|_{JN_{p,q,\alpha}^{\mathrm{con}}(\rn)}\\
&\quad:=
\begin{cases}
\displaystyle
\sup_{m\in\zz}\sup_{\{Q_j\}_j\subset \mathcal{D}_m(\rn)}
\lf[\sum_{j}\lf|Q_{j}\r|\lf\{\lf|Q_{j}\r|^{-\alpha}\lf[\fint_{Q_{j}}
\lf|f(x)-f_{Q_{j}}\r|^{q}\,dx\r]^{\frac{1}{q}}\r\}^{p} \r]^{\frac{1}{p}}\\
\displaystyle\qquad\qquad\qquad\qquad\quad\quad\quad\ \text{if}\quad p\in[1,\fz),\\
\displaystyle\|f\|_{\mathcal{C}_{\alpha,q,0}(\rn)}
\qquad\qquad\qquad\ \ \ \,\text{if}\quad p=\fz
\end{cases}
\end{align*}
is finite, where $\|\cdot\|_{\mathcal{C}_{\alpha,q,0}(\rn)}$ is the same as in \eqref{campanato}.
\end{definition}

\begin{remark}
\begin{enumerate}
\item[\rm (i)]
Observe that $JN_{p,q,\alpha}^{\mathrm{con}}(\rn)$ is only a special
case when $s=0$ of the special John--Nirenberg--Campanato space via congruent cubes,
$JN_{(p,q,s)_\alpha}^{\mathrm{con}}(\rn)$,
introduced in \cite[Definition 1.3]{jtyyz1}. If we use the same notation as there,
$JN_{p,q,\alpha}^{\mathrm{con}}(\rn)$ here is just the space
$JN_{(p,q,0)_\alpha}^{\mathrm{con}}(\rn)$ therein; here, for
the simplicity of the presentation, we denote
$JN_{(p,q,0)_\alpha}^{\mathrm{con}}(\rn)$ simply by
$JN_{p,q,\alpha}^{\mathrm{con}}(\rn)$.

\item[\rm (ii)]
We point out that
$JN_{p,q,\alpha}^{\mathrm{con}}(\rn)$ [and hence $JN_{(p,q,s)_\alpha}^{\mathrm{con}}(\rn)$]
is not a ball quasi-Banach function space. Indeed, it is easy to show that,
for any $f\in\mathscr{M}(\rn)$,
$\|f\|_{JN_{p,q,\alpha}^{\mathrm{con}}(\rn)}=0$
if and only if $f$ is a constant almost everywhere on $\rn$,
which implies that $JN_{p,q,\alpha}^{\mathrm{con}}(\rn)$
does not satisfy Definition \ref{Debqfs}(i), and hence
$JN_{p,q,\alpha}^{\mathrm{con}}(\rn)$ is not a ball quasi-Banach function space.
\end{enumerate}
\end{remark}

To obtain the main results of this section, we first introduce a proper subspace
of $JN_{p,q,\alpha}^{\mathrm{con}}(\rn)$,
denoted by $LJN_{p,q,\alpha}^{\mathrm{con}}(\rn)$.

\begin{definition}\label{lxjnp}
Let $p,q\in[1,\fz]$ and $\alpha\in\rr$.
The \emph{space} $LJN_{p,q,\alpha}^{\mathrm{con}}(\rn)$
is defined to be the set of all the $f\in L^q_\loc(\rn)$ such that
\begin{align*}
&\|f\|_{LJN_{p,q,\alpha}^{\mathrm{con}}(\rn)}\\
&\quad:=
\begin{cases}
\displaystyle
\sup_{r\in(0,\infty)}\lf[\int_{\rn}\lf\{|B(y,r)|^{-\alpha}
\lf[\fint_{B(y,r)}\lf\{f(x)-\inf_{\widetilde{x}\in B(y,r)}
f(\widetilde{x})\r\}^q\,dx\r]^{\frac{1}{q}}\r\}^p\,dy\r]^{\frac{1}{p}}\\
\qquad\qquad\qquad\qquad\qquad\qquad\qquad\qquad\quad\text{\ if}\quad  p\in[1,\fz),\\
\displaystyle
\sup_{B\in\mathbb{B}(\rn)}\lf|B\r|^{-\alpha}
\lf[\fint_{B}\lf\{f(x)-\inf_{\widetilde{x}\in B}
f(\widetilde{x})\r\}^q\,dx\r]^{\frac{1}{q}}\quad\text{if}\quad  p=\fz
\end{cases}
\end{align*}
is finite.
\end{definition}

\begin{remark}\label{rem3.2}
Let $q\in[1,\infty)$ and $\alpha\in\rr$. In this case,
$LJN_{\fz,q,\alpha}^{\mathrm{con}}(\rn)$ coincides with
the space $\mathcal{C}_{\alpha,q,*}(\rn)$ which was
introduced by Hu et al. in \cite[Definition 1.4]{HMY2007}
and, moreover, $\mathcal{C}_{0,q,*}(\rn)$ coincides with
the space $\mathrm{BLO}\,(\rn)$ which was first
introduced by Coifman and Rochberg in \cite{CR80}.
\end{remark}

The following lemma
gives an equivalent norm of
$JN_{p,q,\alpha}^{\mathrm{con}}(\rn)$,
which is just a special
case of \cite[Proposition 2.2]{jtyyz1}.

\begin{lemma}\label{AC}
Let $p,q\in[1,\infty)$ and $\alpha\in \rr$.
Then $f\in JN_{p,q,\alpha}^{\mathrm{con}}(\rn)$ if and only if
$f\in L^q_{\mathrm{loc}}(\rn)$ and
\begin{align*}
[f]_{JN_{p,q,\alpha}^{\mathrm{con}}(\rn)}
:=\sup_{r\in(0,\infty)}\lf[\int_{\rn}\lf\{|B(y,r)|^{-\alpha}
\lf[\fint_{B(y,r)}\lf|f(x)-f_{B(y,r)}\r|^q\,dx
\r]^{\frac{1}{q}}\r\}^p\,dy\r]^{\frac{1}{p}}
\end{align*}
is finite. Moreover, for any $f\in JN_{p,q,\alpha}^{\mathrm{con}}(\rn)$,
$$
\|f\|_{JN_{p,q,\alpha}^{\mathrm{con}}(\rn)}
\sim [f]_{JN_{p,q,\alpha}^{\mathrm{con}}(\rn)},
$$
where the positive equivalence constants are independent of $f$.
\end{lemma}

By Lemma \ref{AC} and \eqref{rem-sub}, we have the following conclusion;
we omit the details here.

\begin{proposition}\label{embed00}
Let $p\in[1,\infty]$, $q\in[1,\infty)$, and $\alpha\in\rr$.
Then $$LJN_{p,q,\alpha}^{\mathrm{con}}(\rn)
\subset JN_{p,q,\alpha}^{\mathrm{con}}(\rn).$$
\end{proposition}

Next, we establish the boundedness of Littlewood--Paley $g$ functions
on $JN_{p,q,\alpha}^{\mathrm{con}}(\rn)$.

\begin{theorem}\label{g-bounded}
Let $p\in [1,\infty]$, $q\in(1,\infty)$,
$\alpha\in(-\infty,\frac{\delta}{n})$
with $\delta\in(0,1]$,
and $g$ be the Littlewood--Paley
$g$-function in \eqref{g-4}. For any
$f\in JN_{p,q,\alpha}^{\mathrm{con}}(\rn)$,
$g(f)$ is either infinite everywhere or
finite almost everywhere and, in the latter case,
there exists a positive constant $C$,
independent of $f$, such that
\begin{align}\label{g-b-1}
\lf\|g(f)\r\|_{LJN_{p,q,\alpha}^{\mathrm{con}}(\rn)}
\leq C\|f\|_{JN_{p,q,\alpha}^{\mathrm{con}}(\rn)}.
\end{align}
\end{theorem}

\begin{proof}
Let $p$, $q$, $\alpha$, $\dz$, and $g$ be the same as in the present theorem.
We only prove the case $p\in[1,\infty)$ because the proof of $p=\fz$ is similar.
Let $f\in JN_{p,q,\alpha}^{\mathrm{con}}(\rn)$.
We first claim that, if $g(f)(x_0)<\fz$ for one $x_0\in \rn$, then
$g(f)$ is finite almost everywhere. Indeed, using Lemma \ref{lem-g-fz},
$\alpha<\frac{\dz}{n}<\frac{\dz}{n}+\frac{1}{p}$, and the
definition of $\|\cdot\|_{JN_{p,q,\alpha}^{\mathrm{con}}(\rn)}$,
we find that, for any ball $B\in \mathbb{B}(\rn)$ containing $x_0$,
\begin{align*}
&|B|^{\frac{1}{p}-\alpha}
\lf[\fint_{B}\lf\{g(f)(x)
-\inf_{\widetilde{x}\in B}g(f)(\widetilde{x})\r\}^q\,dx
\r]^{\frac{1}{q}}\\
&\quad\ls \sum_{k=1}^{\infty}\frac{1}{2^{k(-n\alpha+\frac{n}{p}+\delta)}}
\lf|2^{k}B\r|^{\frac{1}{p}-\alpha}
\lf[\fint_{2^{k}B}\lf|f(z)-f_{2^{k}B}\r|^q\,dz\r]^{\frac{1}{q}}\\
&\quad\ls\|f\|_{JN_{p,q,\alpha}^{\mathrm{con}}(\rn)}<\fz,
\end{align*}
which further implies that $g(f)$ is finite almost everywhere on $B$.
By this and the arbitrariness of $B$, we conclude that, for almost every $x\in\rn$,
$g(f)(x)<\fz$ and hence the above claim holds true.

Now, we show that \eqref{g-b-1} holds true
if $g(f)(x_0)<\fz$ for one $x_0\in \rn$.
Indeed, by Lemma \ref{lem-g-fz}, the above claim,
the Minkowski inequality, the H\"older inequality,
$\alpha\in(-\infty,\frac{\delta}{n})$,
and Lemma \ref{AC}, we conclude that, for any $r\in(0,\fz)$,
\begin{align*}
&\lf[\int_{\rn}\lf\{|B(y,r)|^{-\alpha}
\lf[\fint_{B(y,r)}\lf\{g(f)(x)
-\inf_{\widetilde{x}\in B(y,r)}g(f)(\widetilde{x})
\r\}^q\,dx\r]^{\frac{1}{q}}\r\}^p\,dy\r]^{\frac{1}{p}}\\
&\quad\lesssim\lf\{\int_{\rn}\lf[|B(y,r)|^{-\alpha}
\sum_{k=1}^{\infty}\frac{1}{2^{k\delta}}\fint_{2^{k}B(y,r)}
\lf|f(x)-f_{2^{k}B(y,r)}\r|\,dx\r]^p\,dy\r\}^{\frac{1}{p}}\\
&\quad\lesssim\sum_{k=1}^{\infty}\frac{1}{2^{k\delta}}
\lf\{\int_{\rn}\lf[|B(y,r)|^{-\alpha}
\fint_{2^{k}B(y,r)}\lf|f(x)-f_{2^{k}B(y,r)}\r|\,dx\r]^p
\,dy\r\}^{\frac{1}{p}}\\
&\quad\sim\sum_{k=1}^{\infty}\frac{1}{2^{k(\delta-\alpha n)}}
\lf\{\int_{\rn}\lf[\lf|2^{k}B(y,r)\r|^{-\alpha}
\fint_{2^{k}B(y,r)}\lf|f(x)-f_{2^{k}B(y,r)}
\r|\,dx\r]^p\,dy\r\}^{\frac{1}{p}}\\
&\quad\lesssim\sum_{k=1}^{\infty}\frac{1}{2^{k(\delta-\alpha n)}}
\lf\{\int_{\rn}\lf[\lf|2^{k}B(y,r)\r|^{-\alpha}
\fint_{2^{k}B(y,r)}\lf|f(x)-f_{2^{k}B(y,r)}\r|^q\,dx
\r]^\frac{p}{q}\,dy\r\}^{\frac{1}{p}}\\
&\quad\lesssim\|f\|_{JN_{p,q,\alpha}^{\mathrm{con}}(\rn)}.
\end{align*}
This further implies that $$\|g(f)\|_{LJN_{p,q,\alpha}^{\mathrm{con}}(\rn)}
\ls\|f\|_{JN_{p,q,\alpha}^{\mathrm{con}}(\rn)},$$
and hence finishes the proof of Theorem \ref{g-bounded}.
\end{proof}

\begin{remark}\label{g-bounded-remark}
\begin{enumerate}
\item[\rm (i)]
Let $\dz=1$ and $p=\fz$. In this case, from the conclusion of
Theorem \ref{g-bounded}, we deduce the conclusion of \cite[Theorem 1]{Sun04}
because, by Theorem \ref{g-bounded} in this case, we obtain $g(f)\in \mathcal{C}_{\alpha,q,*}(\rn)$ which, together with $\mathcal{C}_{\alpha,q,*}(\rn)\subset \mathcal{C}_{\alpha,q}(\rn)$,
further implies $g(f)\in \mathcal{C}_{\alpha,q}(\rn)$,
the conclusion of \cite[Theorem 1]{Sun04}. Thus, a special
case of Theorem \ref{g-bounded} is also stronger than \cite[Theorem 1]{Sun04}.

\item[\rm (ii)]
Let $\dz=1$, $p=\fz$, and $\alpha=0$. In this case,
Theorem \ref{g-bounded} coincides with \cite[Corollary 1.1]{MY}.

\item[\rm (iii)]
Let $s\in\nn$. In this case, it is still unclear
whether or not the Littlewood--Paley $g$-function
is bounded on $JN_{(p,q,s)_\alpha}^{\mathrm{con}}(\rn)$
(see \cite[Definition 1.3]{jtyyz1} for the precise definition).
\end{enumerate}
\end{remark}

Next, we establish the boundedness of the Lusin-area function $S$ on the space
$JN_{p,q,\alpha}^{\mathrm{con}}(\rn)$, whose proof is just a repetition
of that of Theorem \ref{g-bounded} with both $g$ and Lemma \ref{lem-g-fz}
therein replaced, respectively, by both $S$ and Lemma \ref{lem-s-fz} here;
we omit details here.

\begin{theorem}\label{S-bounded''}
Let $p\in [1,\infty]$, $q\in(1,\infty)$,
$\alpha\in(-\infty,\frac{\dz}{n})$
with $\dz\in(0,1]$,
and $S$ be the Lusin-area function in \eqref{S}.
For any $f\in JN_{p,q,\alpha}^{\mathrm{con}}(\rn)$,
$S(f)$ is either infinite everywhere or finite
almost everywhere and, in the latter case,
there exists a positive constant $C$,
independent of $f$, such that
\begin{align*}
\lf\|S(f)\r\|_{LJN_{p,q,\alpha}^{\mathrm{con}}(\rn)}
\leq C\|f\|_{JN_{p,q,\alpha}^{\mathrm{con}}(\rn)}.
\end{align*}
\end{theorem}

\begin{remark}
\begin{enumerate}
\item[\rm (i)]
Let $\dz=1$ and $p=\fz$. In this case, from the conclusion of
Theorem \ref{S-bounded''}, we deduce the conclusion of \cite[Theorem 2]{Sun04}
because, by Theorem \ref{S-bounded''} in this case, we obtain $S(f)\in \mathcal{C}_{\alpha,q,*}(\rn)$ which, combined with $\mathcal{C}_{\alpha,q,*}(\rn)\subset \mathcal{C}_{\alpha,q}(\rn)$,
further implies $S(f)\in \mathcal{C}_{\alpha,q}(\rn)$, 
the conclusion of \cite[Theorem 2]{Sun04}. Thus, a special case
of Theorem \ref{S-bounded''} is also stronger than \cite[Theorem 2]{Sun04}.

\item[\rm (ii)]
Let $\dz=1$, $p=\fz$, and $\alpha=0$. In this case,
Theorem \ref{S-bounded''} coincides with \cite[Corollary 1.2]{MY}.

\item[\rm (iii)]
Let $s\in\nn$. In this case,
it is still unclear
whether or not the Lusin-area function
is bounded on $JN_{(p,q,s)_\alpha}^{\mathrm{con}}(\rn)$
(see \cite[Definition 1.3]{jtyyz1} for the precise definition).
\end{enumerate}
\end{remark}

Now, we estimate the Littlewood--Paley $g_\lambda^*$
function on the space $JN_{p,q,\alpha}^{\mathrm{con}}(\rn)$.

\begin{theorem}\label{g*-bounded}
Let $p\in [1,\infty]$, $q\in(1,\infty)$,
$\alpha\in\rr$, $\delta\in(0,1]$,
$\lambda\in (1,\infty)$, and $g_{\lambda}^{*}$ be the same as in \eqref{lam-function}.
For any $f\in JN_{p,q,\alpha}^{\mathrm{con}}(\rn)$,
$g_{\lambda}^{*}(f)$ is either infinite everywhere or finite
almost everywhere and, in the latter case,
there exists a positive constant $C$,
independent of $f$, such that
\begin{align}\label{g*-01x}
\lf\|g_\lambda^{*}(f)\r\|_{LJN_{p,q,\alpha}^{\mathrm{con}}(\rn)}
\leq C\|f\|_{JN_{p,q,\alpha}^{\mathrm{con}}(\rn)}
\end{align}
if either of the following statements holds true:
\begin{enumerate}
\item[\rm (i)]
$\alpha\in(-\fz,\frac{\dz}{2n})$ and $\lambda\in(\max\{1,\frac{2}{q}\},\infty)$;	

\item[\rm (ii)]
$\alpha\in[\frac{\dz}{2n},\frac{\dz}{n})$ and
$\lambda\in(\max\{1,\frac{2}{q}\}+\frac{2}{n},\infty)\cap(2+2\alpha,\infty)$.
\end{enumerate}		
\end{theorem}

\begin{proof}
Let $p$, $q$, $\alpha$, $\dz$,
and $g_\lambda^*$ be the same as in the present theorem.
The proofs of both (i) and (ii) of the present theorem
strongly depend, respectively, on both (i) and (ii) of Lemma \ref{g-lam-01},
and we only prove (i) because the proof of (ii) is similar.
Also, we only show the case $p\in[1,\infty)$ because the proof of $p=\fz$ is similar.
Let $f\in JN_{p,q,\alpha}^{\mathrm{con}}(\rn)$.
We first claim that, if $g_\lambda^*(f)(x_0)<\fz$ for one $x_0\in \rn$, then
$g_\lambda^*(f)$ is finite almost everywhere. Indeed,
using Lemma \ref{g-lam-01}(i), $\alpha\in (-\infty, \frac{\delta}{2n})$, and the
definition of $\|\cdot\|_{JN_{p,q,\alpha}^{\mathrm{con}}(\rn)}$,
we find that, for any ball $B\subset \rn$ containing $x_0$,
\begin{align*}
&|B|^{\frac{1}{p}-\alpha}\lf[\fint_{B}\lf\{g_\lambda^*(f)(x)
-\inf_{\widetilde{x}\in B}g_\lambda^*(f)(\widetilde{x})\r\}^q
\,dx\r]^{\frac{1}{q}}\\
&\quad\ls \sum_{k=1}^{\infty}\frac{1}{2^{k(-n\alpha+\frac{n}{p}+\frac{\delta}{2})}}
\lf|2^{k}B\r|^{\frac{1}{p}-\alpha}
\lf[\fint_{2^{k}B}\lf|f(z)-f_{2^{k}B}\r|^q\,dz\r]^{\frac{1}{q}}\\
&\quad\ls\|f\|_{JN_{p,q,\alpha}^{\mathrm{con}}(\rn)}<\fz,
\end{align*}
which further implies that $g_\lambda^*(f)$ is finite almost everywhere on $B$.
Using this and the arbitrariness of $B$, we find that, for almost every $x\in\rn$,
$g_\lambda^*(f)(x)<\fz$ and hence the above claim holds true.

Next, we show \eqref{g*-01x}. Indeed, from the above claim,
Lemma \ref{g-lam-01}, the Minkowski inequality,
$\alpha\in (-\infty, \frac{\delta}{2n})\subset (-\fz,\frac{1}{2n})$,
and Lemma \ref{AC}, we deduce that, for any $r\in(0,\fz)$,
\begin{align*}
&\lf[\int_{\rn}\lf\{|B(y,r)|^{-\alpha}
\lf[\fint_{B(y,r)}\lf\{g_\lambda^*(f)(x)
-\inf_{\widetilde{x}\in B(y,r)}
g_\lambda^*(f)(\widetilde{x})\r\}^q\,dx
\r]^{\frac{1}{q}}\r\}^q\,dy\r]^{\frac{1}{p}}\\
&\quad\lesssim\lf[\int_{\rn}\lf\{|B(y,r)|^{-\alpha}
\sum_{k=1}^{\infty}\frac{1}{2^{\frac{k\delta}{2}}}
\lf[\fint_{2^{k}B(y,r)}\lf|f(z)-f_{2^{k}B(y,r)}\r|^q
\,dz\r]^{\frac{1}{q}}\r\}^p\,dy\r]^{\frac{1}{p}}\\
&\quad\lesssim\sum_{k=1}^{\infty}\frac{1}{2^{\frac{k\delta}{2}}}
\lf[\int_{\rn}\lf\{|B(y,r)|^{-\alpha}
\lf[\fint_{2^{k}B(y,r)}\lf|f(z)-f_{2^kB(y,r)}\r|^q
\,dz\r]^{\frac{1}{q}}\r\}^p\,dy\r]^{\frac{1}{p}}\\
&\quad\lesssim\sum_{k=1}^{\infty}\frac{1}{2^{\frac{k\delta}{2}-\alpha kn}}
\lf[\int_{\rn}\lf\{\lf|2^kB(y,r)\r|^{-\alpha}
\lf[\fint_{2^{k}B(y,r)}\lf|f(z)-f_{2^{k}B(y,r)}\r|^q\,dz
\r]^{\frac{1}{q}}\r\}^p\,dy\r]^{\frac{1}{p}}\\
&\quad\lesssim\|f\|_{JN_{p,q,\alpha}^{\mathrm{con}}(\rn)}
\end{align*}
and hence
$$
\lf\|g_\lambda^{*}(f)\r\|_{LJN_{p,q,\alpha}^{\mathrm{con}}(\rn)}
\ls\|f\|_{JN_{p,q,\alpha}^{\mathrm{con}}(\rn)}.
$$
This finishes the proof of Theorem \ref{g*-bounded}.
\end{proof}

\begin{remark}\label{glam*remark}
\begin{enumerate}
\item[\rm (i)]	
Let $\dz=1$, $p=\fz$, and $\alpha=0$. In this case,
Theorem \ref{g*-bounded}(i) coincides with \cite[Corollary 1.3]{MY}.	

\item[\rm (ii)]	
Let $\dz=1$. In this case,
we point out that the boundedness of the Littlewood--Paley
$g_\lambda^*$ function on the Campanato space $\mathcal{C}_{\alpha,q,0}(\rn)$ in
Theorem \ref{g*-bounded} widens the range
$\lambda\in (3+\frac{2}{n},\fz)$ in \cite[Theorem 3]{Sun04} into
$\lambda\in(\max\{1,\frac{2}{q}\},\infty)$ when $\alpha\in(-\fz,\frac{1}{2n})$,
or $\lambda\in(\max\{1,\frac{2}{q}\}+\frac{2}{n},\infty)
\cap(2+2\alpha,\infty)$ when $\alpha\in [\frac{1}{2n},\frac{1}{n})$.

\item[\rm (iii)]
Let $s\in\nn$. In this case, it is still unclear
whether or not the Littlewood--Paley $g_\lambda^*$-function
is bounded on $JN_{(p,q,s)_\alpha}^{\mathrm{con}}(\rn)$
(see \cite[Definition 1.3]{jtyyz1} for the precise definition).
\end{enumerate}
\end{remark}

\noindent\textbf{Acknowledgements}\quad
Hongchao Jia and Yangyang Zhang would like to thank Jin Tao for some
useful discussions on the subject of this article.

\medskip

\noindent\textbf{Author contributions}\quad All the authors 
contributed to this article.

\medskip

\noindent\textbf{Data Availability Statement}\quad Data
sharing not applicable to
this article as no datasets were generated or
analysed during the current study.

\medskip

\noindent\textbf{Declarations}

\medskip

\noindent\textbf{Conflict of interest}\quad The author declares that they have no conflict interests.

\medskip

\noindent\textbf{Code availability}\quad Not applicable.

\bigskip

\noindent Hongchao Jia, Dachun Yang (Corresponding author),
Wen Yuan and Yangyang Zhang

\smallskip

\noindent  Laboratory of Mathematics and Complex Systems
(Ministry of Education of China),
School of Mathematical Sciences, Beijing Normal University,
Beijing 100875, The People Republic of China

\smallskip

\noindent {\it E-mails}: \texttt{hcjia@mail.bnu.edu.cn} (H. Jia)

\noindent\phantom{{\it E-mails:}} \texttt{dcyang@bnu.edu.cn} (D. Yang)

\noindent\phantom{{\it E-mails:}} \texttt{wenyuan@bnu.edu.cn} (W. Yuan)

\noindent\phantom{{\it E-mails:}} \texttt{yangyzhang@mail.bnu.edu.cn} (Y. Zhang)


\begin{thebibliography}{99}
\bibitem{syy2}
J. Sun, D. Yang and W. Yuan,
Molecular characterization of weak Hardy spaces associated with
ball quasi-Banach function spaces on spaces of homogeneous type
with its application to Littelwood--Paley function characterization,
Submitted.
\vspace{-0.3cm}
\end{thebibliography}
\end{document}